\documentclass[reqno,10pt]{amsart}
\usepackage{CJK}
\usepackage{hyperref}
\usepackage{cite}
\usepackage{amsmath}
\usepackage{mathrsfs}
\usepackage{extarrows}
\usepackage{indentfirst}
\usepackage{color}
\usepackage{pgfplots}

\pgfplotsset{compat=1.17}
\usepackage{amsmath}
\usepackage{subcaption}

\makeatletter
\@namedef{subjclassname@2020}{%
	\textup{2020} Mathematics Subject Classification}
\makeatother

\numberwithin{equation}{section}

\newtheorem{theorem}{Theorem}[section]
\newtheorem{lemma}[theorem]{Lemma}
\newtheorem{definition}[theorem]{Definition}

\newtheorem{proposition}[theorem]{Proposition}
\newtheorem{remark}[theorem]{Remark}
\newtheorem{corollary}[theorem]{Corollary}

\allowdisplaybreaks

\newcommand{ \mint }{ {\int\hspace{-0.38cm}-}}
\begin{document}
	
	\title[\hfil Regularity for parabolic perturbed fractional 1-Laplacian] {H\"{o}lder regularity for the parabolic perturbed fractional 1-Laplace equations}
	
		\author[D. Li and C. Zhang  \hfil \hfilneg]
		{Dingding Li  and Chao Zhang$^*$}
	
	\thanks{$^*$ Corresponding author.}
	
	\address{Dingding Li \hfill\break
			School of Mathematics, Harbin Institute of Technology, Harbin 150001, China
}
	\email{a87076322@163.com}

	\address{Chao Zhang  \hfill\break School of Mathematics and Institute for Advanced Study in Mathematics, Harbin Institute of Technology, Harbin 150001, China}
	\email{czhangmath@hit.edu.cn}

	
	\maketitle
	
\begin{abstract}
This paper studies the regularity of weak solutions to a class of parabolic perturbed fractional $1$-Laplace equations. Our analysis combines finite difference quotients, energy estimates, and iterative arguments, with a key step being the decomposition of the nonlocal integral into local and nonlocal components to handle their contributions separately. We aim to show the local H\"{o}lder continuity of weak solutions within the parabolic domain. More precisely, the solutions are spatially $\alpha$-H\"{o}lder continuous with $0<\alpha<\min\left\lbrace1, \frac{s_p p}{p-1} \right\rbrace$ and $\gamma$-H\"{o}lder continuous in time, where  the value of $\gamma$ is determined by the fractional differentiability indexes $s_1$, $s_p$ and the exponent $p$. For both the super-quadratic case ($p\ge 2$) and the sub-quadratic case ($1<p<2$), we establish the Sobolev regularity of solutions, which underpins the derivation of H\"{o}lder continuity. All estimates are quantitative and depend only on the structural parameters of the equation. To the best of our knowledge, this is the first attempt to develop a regularity theory for such nonlocal parabolic equations.
\end{abstract}

\section{Introduction}
\label{sec1}

In this paper, we consider the following parabolic perturbed fractional $1$-Laplace equation
\begin{align}
	\label{1.1}
	\partial_t u+(-\Delta_1)^{s_1}u+(-\Delta_p)^{s_p}u=0,
\end{align}
where $p\in (1, +\infty)$ and $s_1$, $s_p\in (0, 1)$. The operators $(-\Delta_1)^{s_1}$ and $(-\Delta_p)^{s_p}$ are defined in the sense of Cauchy principal value integrals as follows:
\begin{align*}
	(-\Delta_1)^{s_1}u(x,t):=2\,\mathrm{P.V.}\int_{\mathbb{R}^N}\frac{u(x,t)-u(y,t)}{|u(x,t)-u(y,t)|}\frac{\mathrm{d}y}{|x-y|^{N+s_1}}
\end{align*}
and
\begin{align*}
	(-\Delta_p)^{s_p}u(x,t):=2\,\mathrm{P.V.}\int_{\mathbb{R}^N}\frac{|u(x,t)-u(y,t)|^{p-2}(u(x,t)-u(y,t))}{|x-y|^{N+s_pp}}\,\mathrm{d}y
\end{align*}
for all $x\in \mathbb {R}^N$ and $t\in \mathbb {R}$.

\subsection{Overview of related results.}

Problems of the type considered in \eqref{1.1} were first introduced by G\'{o}rny, Maz\'{o}n and Toledo in their work \cite{GMT24} within the framework of random walk spaces. This model involves two distinct random walk structures, leading to functionals with different growth behaviors associated with each structure. The authors established a suitable notion of weak solutions and analyzed several fundamental properties of nonlocal evolution problems with $(1, p)$-growth. Equation \eqref{1.1} can also be viewed as a nonlocal counterpart of the classical parabolic perturbed $1$-Laplacian problem
\begin{align}
	\label{local}
	\partial_t u-\mathrm{div}\left( \frac{\nabla u}{|\nabla u|}\right)-\mathrm{div}\left( |\nabla u|^{p-2}\nabla u\right)=0,  
\end{align}
which arises in fluid mechanics \cite{DL76} and materials science \cite{S93}. Due to the lack of smoothness of the associated functional, the equations with $1$-growth have been studied only in relatively restricted settings. For the elliptic case of \eqref{local}, Tsubouchi in \cite{T21}  established the local Lipschitz continuity of weak solutions by means of a distinct strategy that differs from the approach in \cite{BM20}. Subsequently, by combining De Giorgi’s truncation technique with the freezing coefficients method, the authors in \cite{TG22, T24} obtained $C^1_\mathrm{loc}$-regularity results. For the parabolic case \eqref{local}, the continuity of the spatial gradient of weak solutions was further proven in \cite{T25, T24A, T25A}. 

Although the $(1, p)$-growth problem is structurally related to the well-known $(p, q)$-growth problem, it cannot be treated as a standard double phase model, primarily due to the fundamental difference in the representation of weak solutions. Since the functional associated with the $1$-structure is non-differentiable, one must select an element $Z$ from its subdifferential such that $Z\in\partial\left( |\nabla u|\right) $. In addition, the auxiliary quantity $Z$ may exhibit discontinuous behavior, owing to this characteristic, pointwise regularity cannot be expected for $Z$, which significantly restricts the regularity results that can be derived for the solutions when compared with classical $p$-Laplace equations.

We next briefly recall some developments in the study of nonlocal $p$-Laplace problems. Brasco et al. \cite{BL17,BLS18} developed the difference quotient method for the fractional $p$-Laplace equations
\begin{align}
	\label{e1}
	(-\Delta_p)^s u=0
\end{align}
as well as their parabolic counterpart \cite{BLS21}
\begin{align}
	\label{e2}
	\partial_tu+(-\Delta_p)^s u=0.
\end{align}
They identified a  critical  threshold $\frac{p-1}{p}$ for the existence of gradients of weak solutions: the gradient exists whenever $s>\frac{p-1}{p}$, while  the case $s\le\frac{p-1}{p}$ is considerably more delicate. Later, B\"{o}gelein et al. \cite{BDL2401, BDL2402} further lowered this threshold to $\frac{p-2}{p}$ for equation \eqref{e1} by extracting additional information from the nonlocal terms. As a consequence, the gradient of weak solutions always exists in the sub-quadratic case. Meanwhile, the regularity theory for elliptic equation \eqref{e1} has advanced significantly. Building on the Lipschitz continuity of viscosity solutions \cite{BT25} and the equivalence between weak and viscosity solutions \cite{KKL19}, it has been established that solutions are $C^{1, \alpha}_{\mathrm{loc}}$-regular; we refer to \cite{GJS25} for a comprehensive treatment and to \cite{JSU26} for the corresponding Lipschitz continuity result for the parabolic equation \eqref{e2}.

When it comes to the fractional $1$-Laplacian, significant contributions have been made by Maz\'{o}n and his collaborators. Existence results for the elliptic case can be found in \cite{B23, BDLM23, MPRT16, MRT19}, while those for the parabolic case are available in \cite{AMRT08, AMRT09, MRT16}. Under suitable structural assumptions, Bucur et al. \cite{BDLV25} proved that $s$-minimal functions belong to $C(\Omega)\cap L^\infty(\Omega)$ and admit continuous extensions.
Furthermore, Novaga and Onoue \cite{NO23} investigated the minimizers of a nonlocal variational model arising from image denoising. In two dimensions, they established that if the datum $f$ is locally $\beta$-H\"{o}lder continuous with $\beta\in(1-s, 1]$, then the minimizer inherits the same local H\"{o}lder regularity. Despite these substantial advances, regularity theory for the parabolic fractional $1$-Laplace type equations remains largely undeveloped and it is still  an open problem in the field.

Motivated by the aforementioned works and our previous studies \cite{LZ25, LZ2501}, the purpose of this paper is to investigate the regularity of weak solutions to problem \eqref{1.1}. To this end, we revisit and refine the difference quotient method, which enables us to extract effective Sobolev regularity information from the nonlocal structure of the equation and thereby establish the desired Hölder continuity estimates. The analysis is plagued by two essential difficulties. The first lies in the nonlocal formulation of the 1-growth term, which, unlike classical settings, prevents the direct use of mollified approximations to the local 1-Laplacian in the study of approximation equations related to \eqref{1.1}. The second concerns the auxiliary function $Z\in\mathrm{sgn}(u(x,t)-u(y,t))$, whose pointwise regularity is generally unavailable and difficult to control. Consequently, new tailored strategies must be developed to simultaneously cope with the local and nonlocal components of the $1$-growth term, which in turn poses a severe impact on the construction and convergence analysis of the subsequent iterative scheme.

\subsection{Main results}
Before stating our main results, we first introduce the notion of weak solutions to \eqref{1.1}.

\begin{definition}
	\label{def1}
	Let $I$ be an open time interval. We say that $u$ is a local weak solution to problem \eqref{1.1} in the bounded domain $\Omega_I:=\Omega\times I$, if
	\begin{align*}
		u\in L^1\bigl(J; W^{s_1, 1}_\mathrm{loc}(\Omega)\bigl)\cap L^p\bigl(J; W^{s_p, p}_\mathrm{loc}(\Omega)\bigl)\cap C\left(J; L^2(\Omega)\right)\cap L^{\infty}\bigl(J;L^{p-1}_{s_pp}(\mathbb{R}^N)\bigl)
	\end{align*}
	for any $J:=(T_0, T_1)\subset\subset I$ and there exists a function $Z\in \mathrm{sgn}\left( u(x,t)-u(y,t)\right)$ such that
	\begin{align}
		\label{1.2}
		0&=-\int_{T_0}^{T_1}\int_{\Omega}u(x,t)\varphi_t(x,t)\,\mathrm{d}x\mathrm{d}t+\int_{\Omega}u(x,T_1)\varphi(x,T_1)\,\mathrm{d}x-\int_{\Omega}u(x,T_0)\varphi(x,T_0)\,\mathrm{d}x\nonumber\\
		&\quad+\int_{T_0}^{T_1}\int_{\mathbb{R}^N}\int_{\mathbb{R}^N}Z\frac{\varphi(x,t)-\varphi(y,t)}{|x-y|^{N+s_1}}\,\mathrm{d}x\mathrm{d}y\mathrm{d}t\nonumber\\
		&\quad+\int_{T_0}^{T_1}\int_{\mathbb{R}^N}\int_{\mathbb{R}^N}\frac{J_p\left( u(x,t)-u(y,t)\right)\left( \varphi(x,t)-\varphi(y,t)\right)}{|x-y|^{N+s_pp}}\,\mathrm{d}x\mathrm{d}y\mathrm{d}t
	\end{align}
	for any $\varphi\in L^1\bigl(J;W^{s_1,1}_0(\Omega)\bigl)\cap L^p\bigl( J;W^{s_p,p}_0(\Omega)\bigl)\cap C^1\bigl(J;L^2(\Omega)\bigl)$.
\end{definition}

\begin{remark}
	\label{rem28}
	If one only studies the regularity of weak solutions with respect to the spatial variables, the assumption $u\in L^{p-1}\bigl(J;L^{p-1}_{s_p p}(\mathbb{R}^N)\bigl)$ is sufficient. However, in order to investigate the regularity of weak solutions in the time direction, we have to strengthen this assumption to $u\in L^{\infty}\bigl(J;L^{p-1}_{s_pp}(\mathbb{R}^N)\bigl)$. It is worth emphasizing that our assumption is still weaker than the global boundedness assumption in \cite{BLS21}, which necessitates the use of a completely different iteration strategy in our proof.
\end{remark}

For the sake of clarity, we introduce the quantity
\begin{align*}
	\mathcal{T}:=\|u\|_{L^\infty(B_{R}(x_0)\times(T_0,T_1))}+\operatorname*{ess\,sup}_{t\in(T_0,T_1)}\mathrm{Tail}(u;x_0,R,t),
\end{align*}
where $\operatorname{Tail}(u; x_0, R,t)$ denotes the nonlocal tail associated with the function $u(\cdot,t)$ at the point $x_0$ and radius $R$ (see Section \ref{sec2} for the precise definition).

\smallskip

Now we are ready to state our main result of this paper.

\begin{theorem}
	\label{th27}
	Let $p>1$, $s_1,s_p\in(0,1)$ and suppose that $u$ is a locally bounded weak solution to problem \eqref{1.1} in the sense of Definition \ref{def1}. Then, $u$ is locally H\"{o}lder continuous in $\Omega_I$. Moreover, for any $B_R(x_0)\subset\subset\Omega$ with $R\in(0,1)$, $r\in(0,R)$, $(t_0-2S,t_0)\subset\subset I$ with $S\in\left( 0,\min\left\lbrace1,\frac{t_0}{2} \right\rbrace \right) $, $\alpha\in \left( s_p,\min\left\lbrace 1,\frac{s_pp}{p-1}\right\rbrace \right) $ and $\gamma\in\left( 0,\Gamma\right)$, there exist positive constants $C$, $\kappa$, $\kappa'$, $\iota$ and $\iota'$ depending on $N, p, s_1, s_p, \alpha$ and $\gamma$ such that
	\begin{align*}
		&\quad\sup_{(x_1, \tau_1),(x_2, \tau_2)\in B_r(x_0)\times(t_0-S, t_0)}|u(x_1, \tau_1)-u(x_2, \tau_2)|\\
		&\le \frac{C\left[ \left( \mathcal{T}+1\right)^2+\left( \int_{t_0-2S}^{t_0}[u(\cdot,t)]^p_{W^{s_p, p}(B_R(x_0))}\,\mathrm{d}t\right) ^\frac{2}{p}\right] }{S^{\kappa'}(R-r)^{\kappa}}|x_1-x_2|^{\alpha}\\
		&\quad+\frac{C\left[ \left( \mathcal{T}+1\right)^{2p}+\left( \int_{t_0-2S}^{t_0}[u(\cdot,t)]^p_{W^{s_p,p}(B_R(x_0))}\,\mathrm{d}t\right) ^2\right] }{S^{\iota'}(R-r)^{\iota}}|\tau_1-\tau_2|^\gamma,
	\end{align*}
	where $\Gamma:=\begin{cases}
		\frac{1}{\max\left\lbrace s_1, s_pp\right\rbrace+1},\quad&\text{if }s_p\ge\frac{p-1}{p},\\[2mm]
		\frac{s_pp}{(p-1)\max\left\lbrace s_1,s_pp\right\rbrace+s_pp},\quad&\text{if }0<s_p<\frac{p-1}{p}.
	\end{cases}$
\end{theorem}

\begin{remark}
	\label{rem29}
	When establishing the regularity of weak solutions with respect to the time variable, we first need to derive pointwise regularity in the spatial variables. However, part of this spatial regularity is ``consumed" by the operators $(-\Delta_1)^{s_1}$ and $(-\Delta_p)^{s_p}$, which accounts for the specific form of the exponent $\Gamma$ appearing in Theorem \ref{th27}. To visually illustrate the dependence of $\Gamma$ on the parameters $s_1$ and $s_p$, we plot the three-dimensional surfaces of $\Gamma(s_1,s_p)$ for the sub-quadratic case $\left( p=\frac{3}{2}\right) $ and the super-quadratic case $(p=3)$. By comparing these two surfaces, readers can clearly perceive how $\Gamma$ varies with the parameters, thereby gaining an intuitive understanding of the behavior of weak solutions under different nonlinear intensities.
	\begin{figure}[htbp]
		\centering
		\begin{tikzpicture}
			\begin{axis}[
				view={60}{30},
				width=4cm, height=4cm,
				xlabel={$s_1$},
				ylabel={$s_p$},
				xmin=0, xmax=1,
				ymin=0, ymax=1,
				zmin=0, zmax=1,
				xtick={0,1}, ytick={0,1}, ztick={0,0.5,1},
				domain=0:1, y domain=0:1,
				samples=25, samples y=25,
				colormap/viridis,
				opacity=0.6,
				scale only axis,
				at={(0cm,0cm)}
				]
				\def\p{1.5}
				\addplot3[
				surf, shader=interp, z buffer=sort
				](
				x,y,
				{
					y >= (\p-1)/\p ?
					1/(max(x,\p*y)+1)
					:
					(\p*y)/((\p-1)*max(x,\p*y)+\p*y)
				}
				);
			\end{axis}
			
			\begin{axis}[
				view={60}{30},
				width=4cm, height=4cm,
				xlabel={$s_1$},
				ylabel={$s_p$},
				xmin=0, xmax=1,
				ymin=0, ymax=1,
				zmin=0, zmax=1,
				xtick={0,1}, ytick={0,1}, ztick={0,0.5,1},
				domain=0:1, y domain=0:1,
				samples=25, samples y=25,
				colormap/viridis,
				opacity=0.6,
				scale only axis,
				at={(5cm,0cm)}
				]
				\def\p{3}
				\addplot3[
				surf, shader=interp, z buffer=sort
				](
				x,y,
				{
					y >= (\p-1)/\p ?
					1/(max(x,\p*y)+1)
					:
					(\p*y)/((\p-1)*max(x,\p*y)+\p*y)
				}
				);
			\end{axis}
			
		\end{tikzpicture}
		\caption{3D surfaces of $\Gamma(s_1, s_p)$ for $p=\frac{3}{2}$ (left) and $p=3$ (right).}
	\end{figure}
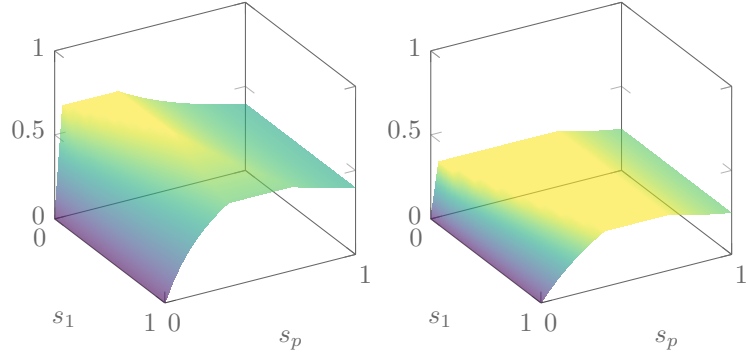
	
	In the sub-quadratic case, although the values of $\Gamma$ can ideally be larger than those in the super-quadratic case, $s_1$ is more likely to exceed $s_pp$ in this regime. As a result, an increase in $s_1$ has a more pronounced effect on $\Gamma$ and reduces its maximum value when $s_1>s_pp$. On the other hand, the H\"{o}lder exponent with respect to the spatial variables always satisfies $\alpha<\min\left\lbrace 1, \frac{s_pp}{p-1}\right\rbrace$,
	which reflects the influence of  $1$-growth term in equation \eqref{1.1}. In the absence of this $1$-structure, one could choose $\alpha\in(0, 1)$ in the super-quadratic case and $\alpha\in\left(0, \min\left\lbrace 1,\frac{s_pp}{p-2}\right\rbrace\right)$
	in the sub-quadratic case. Consequently, the admissible range of the H\"{o}lder exponent in time $\gamma$ would also be enlarged. Therefore, the presence of the $1$-growth term has a significant restrictive effect on the regularity of solutions, which distinguishes the $(1, p)$-growth problem from the classical $(p, q)$-growth framework.
\end{remark}

\subsection{Approach overview}

We now briefly outline the main ideas of the proof. In contrast to our previous works \cite{LZ25, LZ2501}, problem \eqref{1.1} requires additional analysis of the time derivative of weak solutions. To handle this term properly, the weak solution must be allowed to be part of the test function. To ensure the test function is differentiable with respect to time, we perform a mollification in the time direction. A crucial challenge is to maintain the stability of the weak formulation under this mollification procedure. To address this difficulty, we establish Lemma \ref{lem2}, which guarantees the desired convergence properties after the mollification. Based on this result, we derive inequality \eqref{3.0}. Although this identity is not exact, it is sufficient for our purposes. In particular, by employing a new mollification technique, we are able to simplify certain arguments used in \cite{BLS21}.

The next step is to derive a difference-quotient version of the Caccioppoli-type inequality. In the parabolic setting, an additional term arises in the estimates, which must be carefully controlled to provide sufficient information regarding the step size $|h|$. This information is critical for the subsequent iteration scheme. We adopt a similar strategy for both the super-quadratic and sub-quadratic cases: through an iterative procedure, we progressively improve the Sobolev regularity of weak solutions, gradually approaching the level of weak differentiability and integrability required for the analysis. Once the spatial Sobolev regularity is established, we derive the H\"{o}lder continuity with respect to the spatial variables. Building on this result, we further investigate the regularity of weak solutions in time, for which we assume that the tail energy of the weak solution is locally bounded with respect to the time variable.
 
As far as we know, under the assumptions considered in this paper, the obtained regularity results are essentially optimal with respect to the H\"{o}lder exponents.
The main obstacle to further improvements lies in the representation of the auxiliary function $Z$ through difference quotients of weak solutions. This representation is inherently imprecise, and even if it was exact, the possible discontinuous behavior of $Z$ would still prevent a significant improvement of the higher regularity of solutions. Finally, in order to clearly highlight the factors influencing the regularity of weak solutions, all the estimates in this paper are formulated within normalized parabolic cylinders.

\smallskip

The remainder of this paper is organized as follows. In Section \ref{sec2}, we introduce the notations and several auxiliary results that will be used throughout the paper. In Sections \ref{sec3} and \ref{sec4}, we establish Sobolev regularity of weak solutions with respect to the spatial variables for the super-quadratic case and the sub-quadratic case, respectively. In Section \ref{sec5}, we first derive H\"{o}lder estimates for weak solutions with respect to the spatial variables at a fixed time, and then use these estimates to obtain H\"{o}lder continuity in time, thereby completing the proof of the main theorem.

\section{Preliminaries}
\label{sec2}
In this section, we introduce the notations and collect several auxiliary results that will be used later. Most of these results are standard, while for the reader’s convenience we record them here. For any $z \in \mathbb{R}^N$, let $|z|$ denote its Euclidean norm. Throughout the paper, we denote by $C$  generic constants whose values may vary even within the same inequality. When necessary, we explicitly indicate their dependencies, for instance, if a constant depends on $N, p, s_1, s_p$, it will be written as $C(N, p, s_1, s_p)$.

For a function $w: \Omega\rightarrow \mathbb{R}$ and $\alpha\in (0, 1)$, we denote by $w\in C^{0, \alpha}_\mathrm{loc}(\Omega)$ that $w\in C^{0, \alpha}\left(B_R(x_0)\right)$ for any ball $B_R(x_0)\subset\subset \Omega$. The semi-norm of $C^{0, \alpha}$ is defined by
\begin{align*}
	[w]_{C^{0, \alpha}\left(B_R(x_0)\right)}:=\sup\limits_{x, y\in B_R(x_0),  x\neq y}\frac{|w(x)-w(y)|}{|x-y|^\alpha}.
\end{align*}

Now, we recall several function spaces that will be used in the sequel. We introduce the (fractional) Sobolev space $W^{1, q}(\Omega)$ and $W^{\gamma, q}(\Omega)$, where $q\ge1$ and $\gamma\in(0, 1)$. These spaces are respectively defined as

\begin{align*}
	\left\lbrace u\in L^q(\Omega)\bigg|\|u\|_{W^{1,q}(\Omega)}:=\|u\|_{L^q(\Omega)}+\|\nabla u\|_{L^q(\Omega)}<+\infty\right\rbrace 
\end{align*}
and
\begin{align*}
	\left\lbrace u\in L^q(\Omega)\bigg|[u]_{W^{\gamma,q}(\Omega)}:=\left( \int_{\Omega}\int_{\Omega}\frac{|u(x)-u(y)|^q}{|x-y|^{N+\gamma q}}\,\mathrm{d}x\mathrm{d}y\right)^\frac{1}{q} <+\infty\right\rbrace .
\end{align*}

We set the tail space associated with nonlocal operators. For $q,\gamma>0$, the tail space $L^q_\gamma(\mathbb{R}^N)$ contains all function $w\in L^q_\mathrm{loc}(\mathbb{R}^N)$ satisfying
\begin{align*}
	\int_{\mathbb{R}^N}\frac{|w|^q}{(1+|x|)^{N+\gamma}}\,\mathrm{d}x<+\infty
\end{align*}
and we define $\mathrm{Tail}(u;x_0,R,t)$ for a function $u(\cdot,t)$ belonging to the space $L^{p-1}_{s_pp}(\mathbb{R}^N)$ as
\begin{align*}
	\mathrm{Tail}(u;x_0,R,t):=\left( R^{s_pp}\int_{\mathbb{R}^N\backslash B_R(x_0)}\frac{|u(x,t)|^{p-1}}{|x-x_0|^{N+s_p p}}\,\mathrm{d}x\right)^\frac{1}{p-1}. 
\end{align*}

Under the assumption that $u$ is locally bounded, the following estimate for the nonlocal tail holds.

\begin{lemma}
	\label{lem8}
	Let $p>1$ and $s_p\in (0, 1)$. For any $u\in L^\infty\bigl(T_0,T_1; L^{p-1}_{s_pp}(\mathbb{R}^N)\bigl) $, and any ball $B_R\equiv B_R(x_0)\subset\subset\Omega$, $r\in (0,R)$, we have
	\begin{align*}
		\operatorname*{ess\,sup}_{t\in(T_0,T_1)}\mathrm{Tail}(u;x_0,r,t)\le C(N)\left( \frac{R}{r}\right)^N\left(\operatorname*{ess\,sup}_{t\in(T_0,T_1)} \mathrm{Tail}(u;x_0,R,t)+\|u\|_{L^\infty(B_R\times(T_0,T_1))}\right).
	\end{align*}
\end{lemma}

We introduce the difference quotient operator, which will play a crucial role in our regularity analysis. For a direction vector $h\in \mathbb R^N$ and measurable $u: \Omega\to \mathbb R$, define the difference operator 
\begin{align*}
	\tau_hu(x):=u(x+h)-u(x), \quad x\in\mathbb{R}^N,
\end{align*}
and for $\gamma>0$ and $a\in\mathbb{R}$, we define
\begin{align*}
	J_\gamma(a):=|a|^{\gamma-2}a.
\end{align*}

For completeness, we collect several auxiliary lemmas that will be used in the subsequent analysis. Most of them can be found in \cite[Section 2]{LZ25,LZ2501}. These lemmas include two classical algebraic inequalities, the embedding theorem for Sobolev spaces, and results revealing the relationship between the difference quotients of a function and its Sobolev regularity.

\begin{lemma}
	\label{lem9}
	For any $\gamma>0$ and any $a,b\in \mathbb{R}$, we have
	\begin{align*}
		C_1\left( |a|+|b|\right) ^{\gamma-1}|b-a|\le \left| J_{\gamma+1}(b)-J_{\gamma+1}(a)\right| \le C_2\left( |a|+|b|\right) ^{\gamma-1}|b-a|
	\end{align*}
	where $C_1=\min\left\lbrace \gamma,2^{1-\gamma}\right\rbrace $, $C_2=\max\left\lbrace \gamma,2^{1-\gamma}\right\rbrace $.
\end{lemma}

\begin{lemma}
	\label{lem16}
	Let $p\in(1,2)$ and $\delta\ge1$. For any $a, b, c, d\in\mathbb{R}$ and any $e, f\in \mathbb{R}^+$, we have
	\begin{align*}
		&\quad\left( J_p(a-b)-J_p(c-d)\right)\left( J_{\delta+1}(a-c)e^p-J_{\delta+1}(b-d)f^p\right)\\
		&\ge \frac{p-1}{2^{\delta+1}}\left( |a-b|+|c-d|\right)^{p-2}\left( |a-c|+|b-d|\right)^{\delta-1}\left| (a-c)-(b-d)\right|^2\left( e^p+f^p\right)\\
		&\quad -\left( \frac{2^{\delta+1}}{p-1}\right)^{p-1}\left( |a-c|+|b-d|\right)^{p+\delta-1}|e-f|^p . 
	\end{align*}
\end{lemma}

\begin{lemma}[\text{Embedding $W^{1,q}\hookrightarrow W^{\gamma,q}$}]
	\label{lem14.1}
	Let $q\ge1$ and $\gamma\in(0,1)$. Then, for any $w\in L^q\left( T_0,T_1;W^{1,q}(B_R)\right) $, we have
	\begin{align*}
		\int_{T_0}^{T_1}\int_{B_R}\int_{B_R}\frac{|w(x)-w(y)|^q}{|x-y|^{N+\gamma q}}\,\mathrm{d}x\mathrm{d}y\mathrm{d}t\le C(N)\frac{R^{(1-\gamma)q}}{(1-\gamma)q}\int_{T_0}^{T_1}\int_{B_R}|\nabla w|^q\,\mathrm{d}x\mathrm{d}t.
	\end{align*}
\end{lemma}

\begin{lemma}[\text{Embedding $W^{\gamma,q}\hookrightarrow C^{0,\gamma-\frac{N}{q}}$}]
	\label{lem23}
	Let $q\ge1$ and $\gamma\in (0,1)$ such that $\gamma q >N$. Then, there exists a constant $C=C(N, q, \gamma)$ such that for any $w\in W^{\gamma,q}(B_R)$, we have
	\begin{align*}
		[w]_{C^{0,\gamma-\frac{N}{q}}(B_R)}\le C[w]_{W^{\gamma, q}(B_R)}.
	\end{align*}
\end{lemma}

\begin{lemma}
	\label{lem11}
	Let $q\ge1$, $\gamma\in (0,1]$, $M\ge1$, $0<d<R$ and $T_0<T_1<T_0+1$. Then, there exists a constant $C=C(N,q)$ such that whenever $w\in L^q\left(B_{R+d}\times(T_0,T_1)\right)$ satisfies
	\begin{align*}
		\int_{T_0}^{T_1}\int_{B_R}|\tau_hw|^q\,\mathrm{d}x\mathrm{d}t\le M^q|h|^{\gamma q}
	\end{align*}
	for any $h\in B_d\backslash\left\lbrace 0\right\rbrace $, $w\in W^{\beta,q}(B_R)$ for any $\beta\in (0,\gamma)$. Moreover, we have
	\begin{align*}
		\int_{T_0}^{T_1}\int_{B_R}\int_{B_R}\frac{|w(x,t)-w(y,t)|^q}{|x-y|^{N+\beta q}}\,\mathrm{d}x\mathrm{d}y\mathrm{d}t\le C\left[ \frac{d^{(\gamma-\beta)q}}{\gamma-\beta}M^q+\frac{1}{\beta d^{\beta q}}\|w\|^q_{L^q\left(B_R\times(T_0,T_1)\right) }\right]. 
	\end{align*}
\end{lemma}

\begin{lemma}
	\label{lem7}
	Let $q>1$, $\gamma\in (0,1)$, $R\in(0,1)$ and $d\in (0,R)$. Then there exists a constant $C=C(N,q)$ such that for any $w\in L^q\left( T_0,T_1;W^{\gamma,q}(B_R)\right)$, we have
	\begin{align*}
		\int_{T_0}^{T_1}\int_{B_{R-d}}|\tau_hu|^q\,\mathrm{d}x\mathrm{d}t\le C|h|^{\gamma q}\left[(1-\gamma)\int_{T_0}^{T_1}[w]^q_{W^{\gamma,q}(B_R)}\mathrm{d}t+\frac{1}{\gamma d^q}\|w\|_{L^q\left( B_R\times(T_0,T_1)\right) } \right] 
	\end{align*}
	for any $h\in B_d\backslash\left\lbrace0\right\rbrace $.
\end{lemma}

\begin{lemma}
	\label{lem12.1}
	Let $q>1$, $M>0$ and $0<d<R$. Then, any $w\in L^q\bigl(B_R\times(T_0,T_1)\bigl)$ that satisfies
	\begin{align*}
		\int_{T_0}^{T_1}\int_{B_{R-d}}|\tau_hw|^q\,\mathrm{d}x\mathrm{d}t\le M^q|h|^q
	\end{align*}
	for any $h\in B_d\backslash \left\lbrace 0\right\rbrace $ is weakly differentiable in $B_{R-d}$. Moreover, we have
	\begin{align*}
		\int_{T_0}^{T_1}\int_{B_{R-d}}|\nabla u|^q\,\mathrm{d}x\mathrm{d}t\le C(N)M^q.
	\end{align*}
\end{lemma}

\begin{lemma}
	\label{lem10}
	Let $q\ge1$, $\gamma>0$, $M\ge0$, $0<r<R$, $T_0<T_1<T_0+1$ and $d\in \left( 0,\frac{1}{2}(R-r)\right] $. Then, there exists a constant $C=C(q)$ such that whenever $w\in L^q\bigl(B_R\times(T_0,T_1)\bigl)$ satisfies
	\begin{align*}
		\int_{T_0}^{T_1}\int_{B_r}|\tau_h(\tau_hw)|^q\,\mathrm{d}x\mathrm{d}t\le M^q|h|^{\gamma q}
	\end{align*}
	for any $h\in B_d\backslash \left\lbrace 0\right\rbrace $. Then in the case $\gamma\in (0,1)$, we have, for any $h\in B_{\frac{1}{2}d}\backslash\left\lbrace 0\right\rbrace $,
	\begin{align*}
		\int_{T_0}^{T_1}\int_{B_r}|\tau_hw|^q\,\mathrm{d}x\mathrm{d}t\le C(q)|h|^{\gamma q}\left[ \left( \frac{M}{1-\gamma}\right)^q+\frac{1}{d^{\gamma q}}\int_{T_0}^{T_1}\int_{B_R}|w|^q\,\mathrm{d}x\mathrm{d}t \right].
	\end{align*}
	In the case $\gamma>1$, we have
	\begin{align*}
		\int_{T_0}^{T_1}\int_{B_r}|\tau_hw|^q\,\mathrm{d}x\mathrm{d}t\le C(q)|h|^q\left[ \left( \frac{M}{\gamma-1}\right)^qd^{(\gamma-1)q}+\frac{1}{d^q}\int_{T_0}^{T_1}\int_{B_R}|w|^q\,\mathrm{d}x\mathrm{d}t \right].
	\end{align*}
	In the case $\gamma=1$, we have
	\begin{align*}
		\int_{T_0}^{T_1}\int_{B_r}|\tau_hw|^q\,\mathrm{d}x\mathrm{d}t\le C(q)|h|^{\lambda q}\left[ \left( \frac{M}{1-\lambda}\right)^qd^{(1-\lambda)q}+\frac{1}{d^{\lambda q}}\int_{T_0}^{T_1}\int_{B_R}|w|^q\,\mathrm{d}x\mathrm{d}t \right] 
	\end{align*}
	for any $0<\lambda<1$.
\end{lemma}

We conclude this section with two technical lemmas that will be essential in the proof of our main results. Lemma \ref{lem2} below is used to address the convergence issue in the formulation when the weak solution itself is involved in the test function, while Lemma \ref{lem3} will serve as the starting point for our application of the difference quotient technique.

For an open time interval $I$, we set $T:=\sup I$ and introduce the time mollification
\begin{align*}
	u_\varepsilon(x,t):=\frac{1}{\varepsilon}\int_t^{T}e^{\frac{t-\tau}{\varepsilon}}u(x,\tau)\,\mathrm{d}\tau.
\end{align*}

\begin{lemma}
	\label{lem2}
	Suppose $u\in L^1_\mathrm{loc}\bigl(I;W^{s_1,1}_\mathrm{loc}(\Omega)\bigl)\cap L^p_\mathrm{loc}\bigl( I;W^{s_p,p}_\mathrm{loc}(\Omega)\bigl)\cap L^{p-1}_\mathrm{loc}\bigl( I;L^{p-1}_{s_pp}(\mathbb{R}^N)\bigl)$, $v\in L^1\bigl( I;W^{s_1,1}_\mathrm{loc}(\Omega)\bigl)\cap L^p\bigl( I;W^{s_p,p}_\mathrm{loc}(\Omega)\bigl)\cap L^\infty_\mathrm{loc}(\Omega_I)$ and $Z\in L^\infty\bigl( \mathbb{R}^N\times\mathbb{R}^N\times I\bigl) $. Then, for any $\delta\ge1$ and $B_R(x_0)\times(T_0,T_1)\subset\subset\Omega\times I$, we have
	\begin{align}
		\label{1.3}
		&\quad\lim\limits_{\varepsilon\rightarrow0}\int_{T_0}^{T_1}\int_{\mathbb{R}^N}\int_{\mathbb{R}^N}Z\frac{\left( J_{\delta+1}(v_\varepsilon)\eta\right) (x,t)-\left( J_{\delta+1}(v_\varepsilon)\eta\right) (y,t)}{|x-y|^{N+s_1}}\,\mathrm{d}x\mathrm{d}y\mathrm{d}t\nonumber\\
		&=\int_{T_0}^{T_1}\int_{\mathbb{R}^N}\int_{\mathbb{R}^N}Z\frac{\left( J_{\delta+1}(v)\eta\right) (x,t)-\left( J_{\delta+1}(v)\eta\right) (y,t)}{|x-y|^{N+s_1}}\,\mathrm{d}x\mathrm{d}y\mathrm{d}t
	\end{align}
	and
	\begin{align}
		\label{1.4}
		&\quad\lim\limits_{\varepsilon\rightarrow0}\int_{T_0}^{T_1}\int_{\mathbb{R}^N}\int_{\mathbb{R}^N}\frac{J_p\left( u(x,t)-u(y,t)\right)\left[ \left( J_{\delta+1}(v_\varepsilon)\eta\right)(x,t)-\left( J_{\delta+1}(v_\varepsilon)\eta\right)(y,t)\right]}{|x-y|^{N+s_pp}}\,\mathrm{d}x\mathrm{d}y\mathrm{d}t\nonumber\\
		&=\int_{T_0}^{T_1}\int_{\mathbb{R}^N}\int_{\mathbb{R}^N}\frac{J_p\left( u(x,t)-u(y,t)\right)\left[ \left( J_{\delta+1}(v)\eta\right)(x,t)-\left( J_{\delta+1}(v)\eta\right)(x,y)\right] }{|x-y|^{N+s_pp}}\,\mathrm{d}x\mathrm{d}y\mathrm{d}t,
	\end{align}
	where $\eta\in C^1\bigl([T_0,T_1];C^1_0(B_R(x_0))\bigl)$.
\end{lemma}

\begin{proof}
	We set $d:=\frac{1}{2}\min{\left\lbrace \mathrm{dist}\left( B_R,\partial\Omega\right) , T_0, T-T_1\right\rbrace }$ and divide the term on the left-hand side of \eqref{1.3} into two parts:
	\begin{align*}
		&\quad\int_{T_0}^{T_1}\int_{\mathbb{R}^N}\int_{\mathbb{R}^N}Z\frac{\left( J_{\delta+1}(v_\varepsilon)\eta\right) (x,t)-\left( J_{\delta+1}(v_\varepsilon)\eta\right) (y,t)}{|x-y|^{N+s_1}}\,\mathrm{d}x\mathrm{d}y\mathrm{d}t\\
		&=\int_{T_0}^{T_1}\int_{B_{R+d}(x_0)}\int_{B_{R+d}(x_0)}Z\frac{\left( J_{\delta+1}(v_\varepsilon)\eta\right) (x,t)-\left( J_{\delta+1}(v_\varepsilon)\eta\right) (y,t)}{|x-y|^{N+s_1}}\,\mathrm{d}x\mathrm{d}y\mathrm{d}t\\
		&\quad+2\int_{T_0}^{T_1}\int_{B_{R}(x_0)}\int_{\mathbb{R}^N\backslash B_{R+d}(x_0)}Z\frac{\left( J_{\delta+1}(v_\varepsilon)\eta\right)(x,t)}{|x-y|^{N+s_1}}\,\mathrm{d}x\mathrm{d}y\mathrm{d}t\\
		&=:S_{1,1}+2S_{1,2}.
	\end{align*}
	For the local term $S_{1,1}$, note that
	\begin{align*}
		&\quad\left( J_{\delta+1}(v_\varepsilon)\eta\right)(x,t)-\left( J_{\delta+1}(v_\varepsilon)\eta\right)(y,t)\\
		&=\left[ J_{\delta+1}(v_\varepsilon)(x,t)-J_{\delta+1}(v_\varepsilon)(y,t)\right]\frac{\eta(x,t)+\eta(y,t)}{2}\\
		&\quad+\left( \eta(x,t)-\eta(y,t)\right)\frac{J_{\delta+1}(v_\varepsilon)(x,t)+J_{\delta+1}(v_\varepsilon)(y,t)}{2},
	\end{align*}
	we have
	\begin{align*}
		S_{1,1}&=\int_{T_0}^{T_1}\int_{B_{R+d}(x_0)}\int_{B_{R+d}(x_0)}Z\frac{J_{\delta+1}(v_\varepsilon)(x,t)-J_{\delta+1}(v_\varepsilon)(y,t)}{|x-y|^{N+s_1}}\cdot\frac{\eta(x,t)+\eta(y,t)}{2}\,\mathrm{d}x\mathrm{d}y\mathrm{d}t\\
		&\quad+\int_{T_0}^{T_1}\int_{B_{R+d}(x_0)}\int_{B_{R+d}(x_0)}Z\frac{\eta(x,t)-\eta(y,t)}{|x-y|^{N+s_1}}\cdot\frac{J_{\delta+1}(v_\varepsilon)(x,t)+J_{\delta+1}(v_\varepsilon)(y,t)}{2}\,\mathrm{d}x\mathrm{d}y\mathrm{d}t\\
		&=:S_{1,3}+S_{1,4}.
	\end{align*}
	Recalling that
	\begin{align*}
		v_\varepsilon\rightarrow v\quad\text{in }  L^1_{\mathrm{loc}}\bigl(I; W^{s,1}\left(B_{R+d}(x_0)\right)\bigl),
	\end{align*}
	we obatin
	\begin{align}
		\label{1.5}
		J_{\delta+1}(v_\varepsilon)\rightarrow J_{\delta+1}(v)\quad\text{a.e. in }B_{R+d}\times(T_0,T_1).
	\end{align}
	By $v$ is locally bounded in $\Omega_I$, one can get
	\begin{align}
		\label{1.51}
		\left| J_{\delta+1}(v_\varepsilon)(x,t)-J_{\delta+1}(v_\varepsilon)(y,t)\right|\le \delta\|v\|_{L^\infty(B_{R+d}\times(T_0,T_1))}^{\delta-1}|v_\varepsilon(x,t)-v_\varepsilon(y,t)|.
	\end{align}
	
	Since
	\begin{align*}
		\frac{v_{\varepsilon}(x,t)-v_{\varepsilon}(y,t)}{|x-y|^{N+s_1}}\rightarrow\frac{v(x,t)-v(y,t)}{|x-y|^{N+s_1}}\quad\text{in }L^1\bigl( B_{R+d}(x_0)\times B_{R+d}(x_0)\times (T_0,T_1)\bigl), 
	\end{align*}
	one has, for $\dfrac{\varepsilon'}{\delta\|u\|^{\delta-1}_{L^\infty\left( B_{R+d}(x_0)\times(T_0-d,T_1+d)\right)}}>0$, there exists $\delta_0$ such that for any $E\subset B_{R+d}(x_0)\times B_{R+d}(x_0)\times(T_0,T_1)$, satisfying $|E|<\delta_0$, there holds
	\begin{align*}
		\iiint_E\frac{v_{\varepsilon}(x,t)-v_{\varepsilon}(y,t)}{|x-y|^{N+s_1}}\,\mathrm{d}x\mathrm{d}y\mathrm{d}t<\dfrac{\varepsilon'}{\delta\|u\|^{\delta-1}_{L^\infty\left( B_{R+d}(x_0)\times(T_0-d,T_1+d)\right)}}.
	\end{align*}
	The above fact enables us to obtain
	\begin{align}
		\label{1.6}
		\iiint_E\frac{J_{\delta+1}(v_\varepsilon)(x,t)-J_{\delta+1}(v_\varepsilon)(y,t)}{|x-y|^{N+s_1}}\,\mathrm{d}x\mathrm{d}y\mathrm{d}t<\varepsilon'.
	\end{align}
	Combining \eqref{1.5}, \eqref{1.6} and using Vitali's convergence theorem, we have
	\begin{align*}
		J_{\delta+1}(v_\varepsilon)\rightarrow J_{\delta+1}(v)\quad\text{in }L^1\bigl(T_0,T_1;W^{s_1,1}(B_{R+d}(x_0))\bigl).
	\end{align*}
	Thus, we deduce
	\begin{align*}
		\lim\limits_{\varepsilon\rightarrow0}S_{1,3}=\int_{T_0}^{T_1}\int_{B_{R+d}(x_0)}\int_{B_{R+d}(x_0)}Z\frac{J_{\delta+1}(v)(x,t)-J_{\delta+1}(v)(y,t)}{|x-y|^{N+s_1}}\cdot\frac{\eta(x,t)+\eta(y,t)}{2}\,\mathrm{d}x\mathrm{d}y\mathrm{d}t.
	\end{align*}
	
	As to the terms $S_{1,2}$ and $S_{1,4}$, from
	\begin{align*}
		J_{\delta+1}(v_\varepsilon)\rightarrow J_{\delta+1}(v)\quad\text{a.e. in }B_{R+d}(x_0)\times(T_0-d,T_1+d)
	\end{align*}
	and
	\begin{align*}
		\|J_{\delta+1}(v_\varepsilon)\|_{L^\infty\left( B_{R+d}(x_0)\times(T_0,T_1)\right) }\le\|v\|^\delta_{L^\infty\bigl( B_{R+d}(x_0)\times(T_0-d,T_1+d)\bigl) },
	\end{align*}
	we have
	\begin{align}
		\label{1.61}
		J_{\delta+1}(v_\varepsilon)\rightarrow J_{\delta+1}(v)\quad\text{weakly-}^*\text{ in }L^\infty\bigl( B_{R+d}(x_0)\times(T_0,T_1)\bigl). 
	\end{align}
	Therefore, one can get
	\begin{align*}
		\lim\limits_{\varepsilon\rightarrow0}S_{1,2}=\int_{T_0}^{T_1}\int_{B_R(x_0)}\int_{\mathbb{R}^N\backslash B_{R+d}(x_0)}Z\frac{\left[ J_{\delta+1}(v)\eta\right](x,t)}{|x-y|^{N+s_1}}\,\mathrm{d}x\mathrm{d}y\mathrm{d}t
	\end{align*}
	and
	\begin{align*}
		\lim\limits_{\varepsilon\rightarrow0}S_{1,4}=\int_{T_0}^{T_1}\int_{B_{R+d}(x_0)}\int_{B_{R+d}(x_0)}Z\frac{\eta(x,t)-\eta(y,t)}{|x-y|^{N+s_1}}\frac{J_{\delta+1}(v)(x,t)+J_{\delta+1}(v)(y,t)}{2}\,\mathrm{d}x\mathrm{d}y\mathrm{d}t,
	\end{align*}
	where we utilized the fact
	\begin{align*}
		&\quad\int_{T_0}^{T_1}\int_{B_R(x_0)}\int_{\mathbb{R}^N\backslash B_{R+d}(x_0)}\frac{|Z(x,y,t)\eta(x,t)|}{|x-y|^{N+s_1}}\,\mathrm{d}x\mathrm{d}y\mathrm{d}t\\
		&\le \|Z\|_{L^\infty\bigl(\mathbb{R}^N\times\mathbb{R}^N\times(0,T)\bigl)}\|\eta\|_{L^\infty\bigl(B_R(x_0)\times(T_0,T_1)\bigl)}|T_1-T_0||B_R(x_0)|\int_{\mathbb{R}^N\backslash B_{R+d}}\frac{\mathrm{d}y}{|y|^{N+s_1}}\\
		&<+\infty.
	\end{align*}
	
	Similar to the proof of showing \eqref{1.3}, we deal with the following three parts one by one:
	\begin{align*}
		&S_{p,1}:=\int_{T_0}^{T_1}\int_{B_{R+d}(x_0)}\int_{B_{R+d}(x_0)}\frac{J_p\left( u(x,t)-u(y,t)\right)\left( J_{\delta+1}(v_\varepsilon)(x,t)-J_{\delta+1}(v_\varepsilon)(y,t)\right)}{|x-y|^{N+s_pp}}\\
		&\qquad\qquad\qquad\qquad\qquad\qquad\qquad\times\frac{\eta(x,t)+\eta(y,t)}{2}\,\mathrm{d}x\mathrm{d}y\mathrm{d}t,\\
		&S_{p,1}:=\int_{T_0}^{T_1}\int_{B_{R+d}(x_0)}\int_{B_{R+d}(x_0)}\frac{J_p\left( u(x,t)-u(y,t)\right)\left( \eta(x,t)-\eta(y,t)\right)}{|x-y|^{N+s_pp}}\\
		&\qquad\qquad\qquad\qquad\qquad\qquad\qquad\times\frac{J_{\delta+1}(v_\varepsilon)(x,t)+J_{\delta+1}(v_\varepsilon)(y,t)}{2}\,\mathrm{d}x\mathrm{d}y\mathrm{d}t,\\
		&S_{p,3}:=\int_{T_0}^{T_1}\int_{B_{R}(x_0)}\int_{\mathbb{R}^N\backslash B_{R+d}(x_0)}\frac{J_p\left( u(x,t)-u(y,t)\right)\left[ J_{\delta+1}(v_\varepsilon)\eta\right](x,t)}{|x-y|^{N+s_pp}}\,\mathrm{d}x\mathrm{d}y\mathrm{d}t.
	\end{align*}
	From \eqref{1.51}, we find
	\begin{align*}
		\frac{\left| J_{\delta+1}(v_\varepsilon)(x,t)-J_{\delta+1}(v_\varepsilon)(y,t)\right| }{|x-y|^{\frac{N}{p}+s_p}}\le \frac{\delta\|v\|_{L^\infty\left( B_{R+d}(x_0)\times(T_0-d,T_1+d)\right)}\left|v_\varepsilon(x,t)-v_\varepsilon(y,t)\right| }{|x-y|^{\frac{N}{p}+s_p}},
	\end{align*}
	which implies that $\left\lbrace J_{\delta+1}(v_\varepsilon)\right\rbrace_{\varepsilon>0} $ is bounded in $L^p\left( T_0,T_1;W^{s_p,p}(B_{R+d}(x_0))\right)$. It is easy to see that, up a subsequence,
	\begin{align}
		\label{1.7}
		J_{\delta+1}(v_\varepsilon)\rightarrow J_{\delta+1}(v)\quad\text{weakly in }L^p\left( T_0,T_1;W^{s_p,p}(B_{R+d}(x_0))\right). 
	\end{align}
	Since $u\in L^p_\mathrm{loc}\bigl(I;W^{s_p,p}_\mathrm{loc}(\Omega)\bigl)\cap L^\infty\bigl(I;L^{p-1}_{s_pp}(\mathbb{R}^N)\bigl)$, we know that
	\begin{align*}
			&\quad\int_{T_0}^{T_1}\int_{B_{R+d}(x_0)}\int_{B_{R+d}(x_0)}\frac{\left| J_p\left(u(x,t)-u(y,t)\right)\right|^{\frac{p}{p-1}}}{|x-y|^{N+s_pp}}\,\mathrm{d}x\mathrm{d}y\mathrm{d}t<+\infty,\\
			&\quad\int_{T_0}^{T_1}\int_{B_{R+d}(x_0)}\int_{B_{R+d}(x_0)}\frac{\left| J_p\left(u(x,t)-u(y,t)\right)(\eta(x,t)-\eta(y,t))\right|}{|x-y|^{N+s_pp}}\,\mathrm{d}x\mathrm{d}y\mathrm{d}t<+\infty
	\end{align*}
	and
	\begin{align*}
		&\quad\int_{T_0}^{T_1}\int_{B_R(x_0)}\int_{\mathbb{R}^N\backslash B_{R+d}(x_0)}\frac{\left| J_p\left(u(x,t)-u(y,t)\right)\right| }{|x-y|^{N+s_pp}}\,\mathrm{d}x\mathrm{d}y\mathrm{d}t\\
		&\le\int_{T_0}^{T_1}\int_{B_R(x_0)}\int_{\mathbb{R}^N\backslash B_{R+d}(x_0)}\frac{|u(x,t)|^{p-1}+|u(y,t)|^{p-1}}{|x-y|^{N+s_pp}}\,\mathrm{d}x\mathrm{d}y\mathrm{d}t\\
		&\le \frac{C(N,s_p,p)}{d^{s_pp}}\int_{T_0}^{T_1}\int_{B_R(x_0)}|u(x,t)|^{p-1}\,\mathrm{d}x\mathrm{d}t\\
		&\quad+\int_{T_0}^{T_1}\int_{B_R(x_0)}\int_{\mathbb{R}^N\backslash B_{R+d}(x_0)}\frac{|u(y,t)|^{p-1}}{|x-y|^{N+s_pp}}\,\mathrm{d}x\mathrm{d}y\mathrm{d}t\\
		&<+\infty.
	\end{align*}
	Therefore, combining \eqref{1.6}, \eqref{1.7} and the above three facts, we obtain
	\begin{align*}
		&\quad\lim\limits_{\varepsilon\rightarrow0}\left( S_{p,1}+S_{p,2}+2S_{p,3}\right)\\
		&=\int_{T_0}^{T_1}\int_{\mathbb{R}^N}\int_{\mathbb{R}^N}\frac{J_p\left( u(x,t)-u(y,t)\right)\left[ \left( J_{\delta+1}(v)\eta\right)(x,t)-\left( J_{\delta+1}(v)\eta\right)(x,y)\right] }{|x-y|^{N+s_pp}}\,\mathrm{d}x\mathrm{d}y\mathrm{d}t.
	\end{align*}
This finishes the proof.
\end{proof}

In the following lemma, we set function $\eta\in C^1_0\bigl(B_R(x_0)\bigl)$ and $\xi\in C^1\left( [T_0,T_1]\right)$ that satisfies $\xi(T_0)=0$, $\xi(t)\ge0$, $\xi(t)=1$ in $[T_0+S,T_1]$ and $|\xi'(t)|\le \frac{C}{S}$ for $S\in (0,T_1-T_0)$.

\begin{lemma}
	\label{lem3}
	Let $B_R(x_0)\times(T_0,T_1)\subset\subset\Omega\times I$ and let $u$ be a locally bounded weak solution to problem \eqref{1.1}. Then, for any $h\in B_d$ with $d:=\frac{1}{2}\min\left\lbrace \mathrm{dist}(B_R,\partial\Omega), T_0, \sup I-T_1\right\rbrace$, we have
	\begin{align}
		\label{3.0}
		&\quad\int_{T_0}^{T_1}\int_{\mathbb{R}^N}\int_{\mathbb{R}^N}(Z-Z_h)\frac{J_{\delta+1}(\tau_hu)(x,t)\eta^p(x)\xi(t)-J_{\delta+1}(\tau_hu)(y,t)\eta^p(y)\xi(t)}{|x-y|^{N+s_1}}\,\mathrm{d}x\mathrm{d}y\mathrm{d}t\nonumber\\
		&\quad+\int_{T_0}^{T_1}\int_{\mathbb{R}^N}\int_{\mathbb{R}^N}\frac{\left[ J_{p}\left(u_h(x,t)-u_h(y,t)\right)-J_{p}\left(u(x,t)-u(y,t)\right)\right] }{|x-y|^{N+s_pp}}\nonumber\\
		&\qquad\qquad\qquad\qquad\times\left[ J_{\delta+1}(\tau_hu)(x,t)\eta^p(x)\xi(t)-J_{\delta+1}(\tau_hu)(y,t)\eta^p(y)\xi(t)\right]\mathrm{d}x\mathrm{d}y\mathrm{d}t\nonumber\\
		&\quad+\int_{B_R(x_0)}\frac{|\tau_hu|^{\delta+1}(x,T_0+S)}{\delta+1}\eta(x)\,\mathrm{d}x\nonumber\\
		&\le \int_{T_0}^{T_1}\int_{B_R(x_0)}\frac{|\tau_hu|^{\delta+1}}{\delta+1}\eta(x)\xi'(t)\,\mathrm{d}x\mathrm{d}t,
	\end{align}
	where $Z_h:=Z(x+h,y+h, t)$.
\end{lemma}
\begin{proof}
	Since $(h,h)\in B_{d}\times B_{d}$, by performing a transformation on the weak formulation \eqref{1.2} with replacing test function $\varphi(x,t)$ by $\varphi_{-h}(x,t):=\varphi(x-h,t)$, we obtain
	\begin{align}
		\label{3.1}
		0&=-\int_{T_0}^{T_1}\int_\Omega u_h(x,t)\partial_t\varphi(x,t)\,\mathrm{d}x\mathrm{d}t+\int_{\Omega}u_h(x,T_1)\varphi(x,T_1)\,\mathrm{d}x\nonumber\\
		&\quad+\int_{T_0}^{T_1}\int_{\mathbb{R}^N}\int_{\mathbb{R}^N}Z_h\frac{\varphi(x,t)-\varphi(y,t)}{|x-y|^{N+s_1}}\,\mathrm{d}x\mathrm{d}y\mathrm{d}t\nonumber\\
		&\quad+\int_{T_0}^{T_1}\int_{\mathbb{R}^N}\int_{\mathbb{R}^N}\frac{J_p\left(u_h(x,t)-u_h(y,t)\right)\left(\varphi(x,t)-\varphi(y,t)\right)}{|x-y|^{N+s_pp}}\,\mathrm{d}x\mathrm{d}y\mathrm{d}t.
	\end{align}
	Subtracting \eqref{1.2} from \eqref{3.1}, we have
	\begin{align}
		\label{3.2}
		&\quad\int_{T_0}^{T_1}\int_{\mathbb{R}^N}\int_{\mathbb{R}^N}\frac{\left[J_p\left(u_h(x,t)-u_h(y,t)\right)-J_p\left(u(x,t)-u(y,t)\right)\right]\left(\varphi(x,t)-\varphi(y,t)\right)}{|x-y|^{N+s_pp}}\,\mathrm{d}x\mathrm{d}y\mathrm{d}t\nonumber\\
		&\quad+\int_{T_0}^{T_1}\int_{\mathbb{R}^N}\int_{\mathbb{R}^N}(Z_h-Z)\frac{\varphi(x,t)-\varphi(y,t)}{|x-y|^{N+s_1}}\,\mathrm{d}x\mathrm{d}y\mathrm{d}t+\int_{\Omega}\left(\tau_hu\right)(x,T_1)\varphi(x,T_1)\,\mathrm{d}x\nonumber\\
		&=\int_{T_0}^{T_1}\int_{\Omega}(\tau_hu)\partial_t\varphi\,\mathrm{d}x\mathrm{d}t.
	\end{align}
	We set $\varphi(x,t):=J_{\delta+1}\left((\tau_hu)_{\varepsilon}\right)(x,t)\eta^p(x)\xi(t)$ in \eqref{3.2} to get
	\begin{align}
		\label{3.3}
		&\quad\int_{T_0}^{T_1}\int_{\mathbb{R}^N}\int_{\mathbb{R}^N}(Z_h-Z)\frac{J_{\delta+1}\left((\tau_hu)_{\varepsilon}\right)(x,t)\eta^p(x)\xi(t)-J_{\delta+1}\left((\tau_hu)_{\varepsilon}\right)(y,t)\eta^p(y)\xi(t)}{|x-y|^{N+s_1}}\,\mathrm{d}x\mathrm{d}y\mathrm{d}t\nonumber\\
		&\quad+\int_{T_0}^{T_1}\int_{\mathbb{R}^N}\int_{\mathbb{R}^N}\frac{J_p\left(u_h(x,t)-u_h(y,t)\right)-J_p\left(u(x,t)-u(y,t)\right)}{|x-y|^{N+s_pp}}\nonumber\\
		&\qquad\qquad\qquad\qquad\times\left[J_{\delta+1}\left((\tau_hu)_{\varepsilon}\right)(x,t)\eta^p(x)\xi(t)-J_{\delta+1}\left((\tau_hu)_{\varepsilon}\right)(y,t)\eta^p(y)\xi(t)\right]\,\mathrm{d}x\mathrm{d}y\mathrm{d}t\nonumber\\
		&\quad+\int_{\Omega}\left(\tau_hu\right)(x,T_1)J_{\delta+1}\left((\tau_hu)_{\varepsilon}\right)(x,T_1)\eta^p(x)\,\mathrm{d}x\nonumber\\
		&=\int_{T_0}^{T_1}\int_{\Omega}(\tau_hu)\partial_t\left[J_{\delta+1}\left((\tau_hu)_{\varepsilon}\right)(x,t)\eta^p(x)\xi(t)\right]\,\mathrm{d}x\mathrm{d}t. 
	\end{align}
	
	By letting $\varepsilon\rightarrow0^+$ and utilizing Lemma \ref{lem2}, the first two terms of \eqref{3.3} become
	\begin{align}
		\label{3.31}
		\int_{T_0}^{T_1}\int_{\mathbb{R}^N}\int_{\mathbb{R}^N}(Z_h-Z)\frac{J_{\delta+1}\left((\tau_hu)\right)(x,t)\eta^p(x)\xi(t)-J_{\delta+1}\left((\tau_hu)\right)(y,t)\eta^p(y)\xi(t)}{|x-y|^{N+s_1}}\,\mathrm{d}x\mathrm{d}y\mathrm{d}t
	\end{align}
	and
	\begin{align}
		\label{3.32}
		&\int_{T_0}^{T_1}\int_{\mathbb{R}^N}\int_{\mathbb{R}^N}\frac{J_p\left(u_h(x,t)-u_h(y,t)\right)-J_p\left(u(x,t)-u(y,t)\right)}{|x-y|^{N+s_pp}}\nonumber\\
		&\qquad\qquad\qquad\quad\times\left[J_{\delta+1}\left((\tau_hu)\right)(x,t)\eta^p(x)\xi(t)-J_{\delta+1}\left((\tau_hu)\right)(y,t)\eta^p(y)\xi(t)\right]\,\mathrm{d}x\mathrm{d}y\mathrm{d}t.
	\end{align}
	
	For the term
	\begin{align*}
		\int_{\Omega}\left(\tau_hu\right)(x,T_1)J_{\delta+1}\left((\tau_hu)_{\varepsilon}\right)(x,T_1)\eta^p(x)\,\mathrm{d}x,
	\end{align*}
	since $u\in C_{\mathrm{loc}}\left(I;L^2(\Omega)\right)\cap L_\mathrm{loc}^\infty\left(\Omega_I\right)$, one has
	\begin{align*}
		&\quad\lim\limits_{\varepsilon\rightarrow0}\left| \int_{\Omega}(\tau_hu)(x,T_1)\eta^p(x)\left[J_{\delta+1}\left((\tau_hu)_{\varepsilon}\right)(x,T_1)-J_{\delta+1}\left((\tau_hu)\right)(x,T_1)\right]\,\mathrm{d}x\right|\\
		&\le C(\delta)\|\eta\|^p_{L^\infty(B_R)}\|u\|^\delta_{L^\infty\left(B_{R+d}(x_0)\times(T_0-d,T_1+d)\right)}\lim\limits_{\varepsilon\rightarrow0}\int_{\Omega}\left|(\tau_hu)_\varepsilon-\tau_hu\right|(x,T_1)\,\mathrm{d}x.
	\end{align*}
	Thus, we get
	\begin{align}
		\label{3.4}
		\lim\limits_{\varepsilon\rightarrow0}\int_{\Omega}(\tau_hu)(x,T_1)J_{\delta+1}\left((\tau_hu)_{\varepsilon}\right)(x,T_1)\eta^p(x)\,\mathrm{d}x=\int_{\Omega}\left|\tau_hu\right|^{\delta+1}(x,T_1)\eta^p(x)\,\mathrm{d}x.
	\end{align}
	
	Finally, we consider the term on the right-hand side of \eqref{3.3}, which can be decomposed as
	\begin{align}
		\label{3.5}
		&\quad\int_{T_0}^{T_1}\int_{\Omega}(\tau_hu)\partial_t\left[J_{\delta+1}\left((\tau_hu)_{\varepsilon}\right)(x,t)\eta^p(x)\xi(t)\right]\,\mathrm{d}x\mathrm{d}t\nonumber\\
		&=\int_{T_0}^{T_1}\int_{\Omega}(\tau_hu)_\varepsilon\partial_t\left[J_{\delta+1}\left((\tau_hu)_{\varepsilon}\right)(x,t)\eta^p(x)\xi(t)\right]\,\mathrm{d}x\mathrm{d}t\nonumber\\
		&\quad+\int_{T_0}^{T_1}\int_{\Omega}\left[ (\tau_hu)-(\tau_hu)_\varepsilon\right] \partial_t\left[J_{\delta+1}\left((\tau_hu)_{\varepsilon}\right)(x,t)\eta^p(x)\xi(t)\right]\,\mathrm{d}x\mathrm{d}t\nonumber\\
		&=:M_1+M_2.
	\end{align}
	To estimate $M_1$, by integration by parts, we observe that
	\begin{align*}
		M_1&=-\int_{T_0}^{T_1}\int_{\Omega}\partial_t[(\tau_hu)_\varepsilon]\left[J_{\delta+1}\left((\tau_hu)_{\varepsilon}\right)(x,t)\eta^p(x)\xi(t)\right]\,\mathrm{d}x\mathrm{d}t\\
		&\quad+\int_{\Omega}\left|(\tau_hu)_\varepsilon\right|^{\delta+1}(x,T_1)\eta^p(x)\,\mathrm{d}x\\
		&=-\int_{T_0}^{T_1}\int_{\Omega}\partial_t\left[\frac{|(\tau_hu)_\varepsilon|^{\delta+1}}{\delta+1}\right]\eta^p(x)\xi(t)\,\mathrm{d}x\mathrm{d}t+\int_{\Omega}\left|(\tau_hu)_\varepsilon\right|^{\delta+1}(x,T_1)\eta^p(x)\,\mathrm{d}x\\
		&=\int_{T_0}^{T_1}\int_{\Omega}\frac{|(\tau_hu)_\varepsilon|^{\delta+1}}{\delta+1}\eta^p(x)\xi'(t)\,\mathrm{d}x\mathrm{d}t-\int_{\Omega}\frac{|(\tau_hu)_\varepsilon|^{\delta+1}}{\delta+1}\eta^p(x)\,\mathrm{d}x\\
		&\quad+\int_{\Omega}\left|(\tau_hu)_\varepsilon\right|^{\delta+1}(x,T_1)\eta^p(x)\,\mathrm{d}x.
	\end{align*}
	Using similar method of showing \eqref{3.4}, we have
	\begin{align}
		\label{3.6}
		\lim\limits_{\varepsilon\rightarrow0}M_1&=\int_{T_0}^{T_1}\int_{\Omega}\frac{|\tau_hu|^{\delta+1}}{\delta+1}\eta^p(x)\xi'(t)\,\mathrm{d}x\mathrm{d}t-\int_{\Omega}\frac{|\tau_hu|^{\delta+1}}{\delta+1}\eta^p(x)\,\mathrm{d}x\nonumber\\
		&\quad+\int_{\Omega}\left|\tau_hu\right|^{\delta+1}(x,T_1)\eta^p(x)\,\mathrm{d}x.
	\end{align}
	Note that
	\begin{align*}
		\partial_t\left[ (\tau_hu)_{\varepsilon}\right]=\frac{1}{h}\left[(\tau_hu)_\varepsilon-(\tau_hu)\right],  
	\end{align*}
	one can estimate $M_2$ to get
	\begin{align*}
		M_2&=\int_{T_0}^{T_1}\int_{\Omega}\left[(\tau_hu)-(\tau_hu)_\varepsilon\right]J_{\delta+1}\left((\tau_hu)_\varepsilon\right)\eta^p(x)\xi'(t)\,\mathrm{d}x\mathrm{d}t\\
		&\quad+\frac{\delta}{\varepsilon}\int_{T_0}^{T_1}\int_{\Omega}\left[(\tau_hu)-(\tau_hu)_\varepsilon\right]\left|(\tau_hu)_\varepsilon\right|^\delta\left[(\tau_hu)_\varepsilon-(\tau_hu)\right]\eta^p(x)\xi(t)\,\mathrm{d}x\mathrm{d}t\\
		&\le \int_{T_0}^{T_1}\int_{\Omega}\left[(\tau_hu)-(\tau_hu)_\varepsilon\right]J_{\delta+1}\left((\tau_hu)_\varepsilon\right)\eta^p(x)\xi'(t)\,\mathrm{d}x\mathrm{d}t\\
		&\le \frac{C(\delta)}{S}\|\eta\|^p_{L^\infty\left(B_{R+d}(x_0)\times(T_0-d,T_1+d)\right)}\int_{T_0}^{T_1}\int_{\Omega}\left|\tau_hu-(\tau_hu)_\varepsilon\right|(x,T_1)\,\mathrm{d}x\mathrm{d}t.
	\end{align*}
	The above estimate enables us to deduce
	\begin{align}
		\label{3.7}
		\lim\limits_{\varepsilon\rightarrow0}M_2\le0.
	\end{align}
	
Letting $\varepsilon$ go to $0$ in \eqref{3.3} and combining \eqref{3.31}--\eqref{3.7}, we deduce the desired inequality.
\end{proof}

\section{Sobolev regularity  for the super-quadratic case}
\label{sec3} 
In this section, we study the Sobolev regularity of weak solutions to problem \eqref{1.1} in the super-quadratic case and derive the conditions required for the difference quotient technique to improve weak differentiability. 

A key feature of this approach is that, in the proposition below, every term on the right-hand side of inequality \eqref{4.0} can be traced back to a specific component of the equation. This enables us to quantify and distinguish the contributions of different structural terms in \eqref{1.1} to the overall estimates. In particular, we observe that the $1$-structure acts as a nonlinear switch in the equation, whose influence on the regularity exponents is independent of the parameter $s_1$.

Set
\begin{align*}
	\mu^{s_p}_p:=
	\begin{cases}
		\frac{s_pp-2}{p-2}\quad &\text{if }p>2,\\[2mm]
		0\quad &\text{if }p=2.
	\end{cases}
\end{align*}
We begin with a basic estimate for the difference quotients of weak solutions.

\begin{proposition}
	\label{pro4}
	Let $p\ge2$, $s_1,s_p\in (0,1)$, $q\ge p$ and $\sigma\in\left( \max\left\lbrace 0,\mu^{s_p}_p\right\rbrace, \min\left\lbrace 1,\frac{s_pp}{p-1}\right\rbrace  \right) $. There exists a constant $C=C(N,p,q,s_1,s_p,\sigma)$ such that for any $u\in L^q_{\mathrm{loc}}\bigl( I;W^{\sigma,q}_{\mathrm{loc}}(\Omega)\bigl)$ being a locally bounded weak solution to problem \eqref{1.1} in the sense of Definition \ref{def1}, we have, for any $r\in(0,R)$ with $R\in(0,1)$, $d\in \left( 0,\frac{1}{4}(R-r)\right]$, $h\in B_d\backslash\left\lbrace0\right\rbrace$,  $B_{R+d}(x_0)\times(T_0,T_1)\subset\subset\Omega\times I$ and $S\in (0,\min\left\lbrace T_1-T_0,1\right\rbrace )$,
	\begin{align}
		\label{4.0}
		&\quad\int_{T_0+S}^{T_1}\int_{B_r(x_0)}\int_{B_r(x_0)}\frac{\left| J_{\frac{q}{p}+1}(\tau_hu)(x,t)-J_{\frac{q}{p}+1}(\tau_hu)(y,t)\right|^p}{|x-y|^{N+s_pp}}\,\mathrm{d}x\mathrm{d}y\mathrm{d}t\nonumber\\
		&\le \frac{CR^{1-s_1}}{R-r}\int_{T_0}^{T_1}\int_{B_R(x_0)}\left|\tau_hu\right|^{q-p+1}\,\mathrm{d}x\mathrm{d}t+\frac{C}{(R-r)^{s_1}}\int_{T_0}^{T_1}\int_{B_R(x_0)}\left|\tau_hu\right|^{q-p+1}\,\mathrm{d}x\mathrm{d}t\nonumber\\
		&\quad+\frac{CR^{(N-\varepsilon)\frac{q-p+2}{q}}}{(R-r)^2}\left( \int_{T_0}^{T_1}[u(\cdot,t)]^{q}_{W^{\sigma,q}(B_R(x_0))}\,\mathrm{d}t\right)^\frac{p-2}{p} \left(\int_{T_0}^{T_1}\int_{B_{R}(x_0)}\left|\tau_hu\right|^{q}\,\mathrm{d}x\mathrm{d}t\right)^{\frac{q-p+2}{q}}\nonumber\\
		&\quad+\frac{C}{S}\int_{T_0}^{T_1}\int_{B_R(x_0)}\left|\tau_hu\right|^{q-p+2}\,\mathrm{d}x\mathrm{d}t+\frac{C\mathcal{T}_{R+d}^{p-2}}{(R-r)^{N+s_pp}}\int_{T_0}^{T_1}\int_{B_R(x_0)}\left|\tau_hu\right|^{q-p+2}\,\mathrm{d}x\mathrm{d}t\nonumber\\
		&\quad+\frac{C\mathcal{T}_{R+d}^{p-1}|h|}{(R-r)^{N+s_pp+1}}\int_{T_0}^{T_1}\int_{B_R(x_0)}\left|\tau_hu\right|^{q-p+1}\,\mathrm{d}x\mathrm{d}t,
	\end{align}
	where
	\begin{align*}
		\varepsilon:=\left[N+s_pp-2-\frac{(N+\sigma q)(p-2)}{q}\right]\frac{q}{q-p+2}
	\end{align*}
	and
	\begin{align*}
		\mathcal{T}_{R+d}:=\|u\|_{L^\infty\left(B_{B+d}(x_0)\times(T_0,T_1)\right)}+\operatorname*{ess\,sup}_{t\in(T_0,T_1)}\mathrm{Tail}(u;x_0,R+d,t). 
	\end{align*}
\end{proposition}
\begin{proof}
	We begin our proof from the inequality \eqref{3.0} with $\delta:=q-p+1$. Let $\eta$ be a cut-off function that satisfies $\eta\equiv1$ in $B_r(x_0)$, $\eta\in C^1_0\bigl(B_{\frac{1}{2}(R+r)}(x_0)\bigl)$ and $|\nabla\eta|\le \frac{C}{R-r}$. To deal with the $1$-growth and $p$-growth terms respectively, we decompose the two structures into local terms and nonlocal terms:
	\begin{align*}
		&\quad\int_{T_0}^{T_1}\int_{\mathbb{R}^N}\int_{\mathbb{R}^N}(Z_h-Z)\frac{J_{q-p+2}(\tau_hu)(x,t)\eta^p(x)\xi(t)-J_{q-p+2}(\tau_hu)(y,t)\eta^p(y)\xi(t)}{|x-y|^{N+s_1}}\,\mathrm{d}x\mathrm{d}y\mathrm{d}t\\
		&=\int_{T_0}^{T_1}\int_{B_R(x_0)}\int_{B_R(x_0)}\frac{J_{q-p+2}(\tau_hu)(x,t)\eta^p(x)\xi(t)-J_{q-p+2}(\tau_hu)(y,t)\eta^p(y)\xi(t)}{|x-y|^{N+s_1}}\nonumber\\
		&\qquad\qquad\qquad\qquad\qquad\times(Z_h-Z)\,\mathrm{d}x\mathrm{d}y\mathrm{d}t\\
		&\quad+2\int_{T_0}^{T_1}\int_{B_{\frac{1}{2}(R+r)}(x_0)}\int_{\mathbb{R}^N\backslash B_{R}(x_0)}(Z_h-Z)\frac{J_{q-p+2}(\tau_hu)(x,t)\eta^p(x)\xi(t)}{|x-y|^{N+s_1}}\,\mathrm{d}x\mathrm{d}y\mathrm{d}t\\
		&=:I_1+2I_2
	\end{align*}
	and
	\begin{align}
		\label{4.01}
		&\quad\int_{T_0}^{T_1}\int_{\mathbb{R}^N}\int_{\mathbb{R}^N}\frac{\left[ J_{p}\left(u_h(x,t)-u_h(y,t)\right)-J_{p}\left(u(x,t)-u(y,t)\right)\right] }{|x-y|^{N+s_pp}}\nonumber\\
		&\qquad\qquad\qquad\quad\times\left[ J_{q-p+2}(\tau_hu)(x,t)\eta^p(x)\xi(t)-J_{q-p+2}(\tau_hu)(y,t)\eta^p(y)\xi(t)\right]\mathrm{d}x\mathrm{d}y\mathrm{d}t\nonumber\\
		&=\int_{T_0}^{T_1}\int_{B_R(x_0)}\int_{B_R(x_0)}\frac{\left[ J_{p}\left(u_h(x,t)-u_h(y,t)\right)-J_{p}\left(u(x,t)-u(y,t)\right)\right] }{|x-y|^{N+s_pp}}\nonumber\\
		&\qquad\qquad\qquad\qquad\qquad\times\left[ J_{q-p+2}(\tau_hu)(x,t)\eta^p(x)\xi(t)-J_{q-p+2}(\tau_hu)(y,t)\eta^p(y)\xi(t)\right]\mathrm{d}x\mathrm{d}y\mathrm{d}t\nonumber\\
		&\quad+2\int_{T_0}^{T_1}\int_{B_{\frac{1}{2}(R+r)}(x_0)}\int_{\mathbb{R}^N\backslash B_{R}(x_0)}\frac{\left[ J_{p}\left(u_h(x,t)-u_h(y,t)\right)-J_{p}\left(u(x,t)-u(y,t)\right)\right] }{|x-y|^{N+s_pp}}\nonumber\\
		&\qquad\qquad\qquad\qquad\qquad\qquad\qquad\quad\ \times J_{q-p+2}(\tau_hu)(x,t)\eta^p(x)\xi(t)\mathrm{d}x\mathrm{d}y\mathrm{d}t\nonumber\\
		&=:P_1+2P_2.
	\end{align}
	
	Following the approach devised in \cite[Proof of Proposition 4.2]{LZ25}, and upon integrating identities (4.5)–(4.9) from \cite{LZ25} over the time interval $(T_0, T_1)$, we derive the following estimates:
	
	\begin{align}
		\label{4.1}
		I_{1,1}&:=\int_{T_0}^{T_1}\int_{B_R(x_0)}\int_{B_R(x_0)}(Z_h-Z)
		\frac{J_{q-p+2}(\tau_hu)(x,t)-J_{q-p+2}(\tau_hu)(y,t)}{|x-y|^{N+s_1}} \nonumber\\
		&\qquad\qquad\qquad\qquad\qquad\times 
		\frac{\eta^p(x)\xi(t)+\eta^p(y)\xi(t)}{2}
		\,\mathrm{d}x\mathrm{d}y\mathrm{d}t\ge0,
		\\[6pt]
		\label{4.2}
		I_{1,2}&:=\int_{T_0}^{T_1}\int_{B_R(x_0)}\int_{B_R(x_0)}(Z_h-Z)
		\frac{\eta^p(x)\xi(t)-\eta^p(y)\xi(t)}{|x-y|^{N+s_1}} \nonumber\\
		&\qquad\qquad\qquad\qquad\qquad\times
		\frac{J_{q-p+2}(\tau_hu)(x,t)+J_{q-p+2}(\tau_hu)(y,t)}{2}
		\,\mathrm{d}x\mathrm{d}y\mathrm{d}t\nonumber\\
		&\le
		\frac{C(N,p)}{(1-s_1)(R-r)}
		\int_{T_0}^{T_1}\int_{B_R(x_0)}
		|\tau_hu|^{q-p+1}\,\mathrm{d}x\mathrm{d}t,
		\\[6pt]
		\label{4.3}
		I_2
		&\le
		\frac{C(N)}{s_1(R-r)^{s_1}}
		\int_{T_0}^{T_1}\int_{B_{\frac12(R+r)}(x_0)}
		|\tau_hu|^{q-p+1}\,\mathrm{d}x\mathrm{d}t,
		\\[6pt]
		\label{4.4}
		P_1
		&\ge
		\int_{T_0}^{T_1}\int_{B_R(x_0)}\int_{B_R(x_0)}
		\frac{|J_{\frac{q}{p}+1}(\tau_hu)(x,t)-J_{\frac{q}{p}+1}(\tau_hu)(y,t)|^p
			(\eta^p(x)+\eta^p(y))\xi(t)}
		{|x-y|^{N+s_pp}}
		\,\mathrm{d}x\mathrm{d}y\mathrm{d}t
		\nonumber\\
		&\quad-
		C\int_{T_0}^{T_1}\int_{B_R(x_0)}\int_{B_R(x_0)}
		\frac{(|\tau_hu(x,t)|+|\tau_hu(y,t)|)^{q-p+2}
			|\eta^{\frac p2}(x)-\eta^{\frac p2}(y)|^2\xi(t)}
		{|x-y|^{N+s_pp}}
		\nonumber\\
		&\qquad\qquad\qquad\qquad\qquad\qquad\times
		\left(|u_h(x,t)-u_h(y,t)|+|u(x,t)-u(y,t)|^{p-2}\right)
		\,\mathrm{d}x\mathrm{d}y\mathrm{d}t,
		\\[6pt]
		\label{4.5}
		P_{1,2}
		&:=
		\int_{T_0}^{T_1}\int_{B_R(x_0)}\int_{B_R(x_0)}
		\frac{(|\tau_hu(x,t)|+|\tau_hu(y,t)|)^{q-p+2}
			|\eta^{\frac p2}(x)-\eta^{\frac p2}(y)|^2\xi(t)}
		{|x-y|^{N+s_pp}}
		\nonumber\\
		&\qquad\qquad\qquad\qquad\qquad\times
		\left(|u_h(x,t)-u_h(y,t)|+|u(x,t)-u(y,t)|^{p-2}\right)
		\,\mathrm{d}x\mathrm{d}y\mathrm{d}t\nonumber\\
		&\le
		\frac{C(p,q)}{(R-r)^2}
		\left(
		\int_{T_0}^{T_1}[u(t)]^{q}_{W^{\sigma,q}(B_R(x_0))}
		\,\mathrm{d}t
		\right)^{\frac{p-2}{p}}
		\left(\int_{T_0}^{T_1}\int_{B_R(x_0)}|\tau_hu|^{q}\,\mathrm{d}x\mathrm{d}t\right)^{\frac{q-p+2}{q}},
		\\[6pt]
		\label{4.6}
		|P_2|
		&\le
		\frac{C\mathcal T^{p-2}}{R^{s_pp}}
		\left(\frac{R}{R-r}\right)^{N+s_pp}
		\int_{T_0}^{T_1}\int_{B_R(x_0)}
		|\tau_hu|^{q-p+2}\,\mathrm{d}x\mathrm{d}t
		\nonumber\\
		&\quad+
		\frac{C\mathcal T^{p-1}|h|}{(R-r)^{N+s_pp+1}}
		\left(\frac{R}{R-r}\right)^{N+s_pp+1}
		\int_{T_0}^{T_1}\int_{B_R(x_0)}
		|\tau_hu|^{q-p+1}\,\mathrm{d}x\mathrm{d}t .
	\end{align}
	
	Now, we bring the previous estimates into the inequality \eqref{3.0} to deduce
	\begin{align*}
		&\quad\int_{T_0+S}^{T_1}\int_{B_r(x_0)}\int_{B_r(x_0)}\frac{\left| J_{\frac{q}{p}+1}(\tau_hu)(x,t)-J_{\frac{q}{p}+1}(\tau_hu)(y,t)\right|^p}{|x-y|^{N+s_pp}}\,\mathrm{d}x\mathrm{d}y\mathrm{d}t\\
		&\le \frac{CR^{1-s_1}}{R-r}\int_{T_0}^{T_1}\int_{B_R(x_0)}\left|\tau_hu\right|^{q-p+1}\,\mathrm{d}x\mathrm{d}t+\frac{C}{(R-r)^{s_1}}\int_{T_0}^{T_1}\int_{B_R(x_0)}\left|\tau_hu\right|^{q-p+1}\,\mathrm{d}x\mathrm{d}t\\
		&\quad+\frac{CR^{(N-\varepsilon)\frac{q-p+2}{q}}}{(R-r)^2}\left( \int_{T_0}^{T_1}[u(\cdot,t)]^{q}_{W^{\sigma,q}(B_R(x_0))}\,\mathrm{d}t\right)^\frac{p-2}{p} \left(\int_{T_0}^{T_1}\int_{B_{R}(x_0)}\left|\tau_hu\right|^{q}\,\mathrm{d}x\mathrm{d}t\right)^{\frac{q-p+2}{q}}\\
		&\quad+\frac{C\mathcal{T}^{p-2}}{(R-r)^{N+s_pp}}\int_{T_0}^{T_1}\int_{B_R(x_0)}\left|\tau_hu\right|^{q-p+2}\,\mathrm{d}x\mathrm{d}t\\
		&\quad+\frac{C\mathcal{T}^{p-1}|h|}{(R-r)^{N+s_pp+1}}\int_{T_0}^{T_1}\int_{B_R(x_0)}\left|\tau_hu\right|^{q-p+1}\,\mathrm{d}x\mathrm{d}t\\
		&\quad+C(p,q)\int_{T_0}^{T_1}\int_{B_R(x_0)}\left|\tau_hu\right|^{q-p+2}\eta(x)\xi'(t)\,\mathrm{d}x\mathrm{d}t.
	\end{align*}
	Recalling that $|\xi'(t)|\le \frac{C}{S}$ and $|\eta|\le1$, we get the desired result.
\end{proof}

\begin{remark}
	\label{rem5}
	We temporarily neglect the term
	\begin{align*}
		\int_{B_R(x_0)}\frac{|\tau_hu|^{\delta+1}(x,T_1)}{\delta+1}\,\mathrm{d}x
	\end{align*}
	in the inequality \eqref{3.0}, as it is non-negative and will not play a role in the subsequent iteration. We note that the property $Z\in\mathrm{sgn}(u(x,t)-u(y,t))$ is used when proving the non-negativity of $I_{1,1}$.
	In addition, H\"{o}lder's inequality is applied to estimate the term $P_{1,2}$. This step requires the restriction $\sigma>\max\left\{0,\mu^{s_p}_p\right\}$. We also impose an upper bound on $\sigma$, because once $\sigma$ exceeds this range, the difference quotient technique can no longer improve the differentiability of weak solutions.
\end{remark}

In Lemmas \ref{lem6} and \ref{lem12} below, we establish the improved regularity results in three different cases.

\begin{lemma}
	\label{lem6}
	Let $p\ge2$, $s_p\in \left( 0,\frac{p-1}{p}\right] $, $q\ge p$ and suppose that $u\in L^q_{\mathrm{loc}}\bigl(I;W^{\sigma,q}_{\mathrm{loc}}(\Omega)\bigl) $ is a locally bounded weak solution to problem \eqref{1.1} in the sense of Definition \ref{def1} with
	\begin{align*}
		\sigma\in\left(\max\left\lbrace0,\mu^{s_p}_p\right\rbrace,\frac{s_pp}{p-1}\right). 
	\end{align*}
	Then, we have
	\begin{align*}
		u\in L^q_{\mathrm{loc}}\bigl(I;W^{\alpha,q}_{\mathrm{loc}}(\Omega)\bigl) 
	\end{align*}
	for any $\alpha\in(\sigma,\beta)$ with
	\begin{align*}
		\beta:=\left( 1-\frac{p-1}{q}\right)\sigma+\frac{s_pp}{q}. 
	\end{align*}
	Moreover, there exists a constant $C=C(N,p,q,s_1,s_p,\sigma,\alpha)$ such that for any ball $B_R(x_0)\times(T_0,T_1)\subset\subset\Omega\times I$ with $R\in(0,1)$, $r\in(0,R)$ and $S\in (0,\min\left\lbrace 1,T_1-T_0\right\rbrace )$,
	\begin{align*}
		\int_{T_0+S}^{T_1}[u(\cdot,t)]^q_{W^{\alpha,q}\left( B_r(x_0)\right)}\mathrm{d}t\le\frac{C\left[\left(\mathcal{T}+1\right)^q+\int_{T_0}^{T_1}[u(\cdot,t)]^q_{W^{\sigma,q}\left( B_R(x_0)\right)}\mathrm{d}t\right]}{S(R-r)^{N+2q+2}}.
	\end{align*}
\end{lemma}
\begin{proof}
	Since the difference quotient technique is applied only with respect to the spatial variables and we have assumed $u\in L^q_{\mathrm{loc}}\bigl( I;W^{\sigma,q}_{\mathrm{loc}}(\Omega)\bigl)$. By H\"{o}lder's inequality and Lemma \ref{lem7}, we obtain
	\begin{align}
		\label{6.0}
		&\quad\int_{T_0}^{T_1}\int_{B_{R-d}(x_0)}|\tau_hu|^r(x,t)\,\mathrm{d}x\mathrm{d}t\nonumber\\
		&\le \left[C|h|^{\sigma q}\left((1-\sigma)\int_{T_0}^{T_1}[u]^q_{W^{\sigma,q}(B_R(x_0))}\,\mathrm{d}t+\frac{7R^N}{\sigma(R-r)^q}\|u\|^q_{L^\infty\left( B_{R-d}(x_0)\times(T_0,T_1)\right) }\right)\right]^\frac{r}{q}, 
	\end{align}
	where $r\in\left\lbrace q-p+1,q-p+2\right\rbrace $.
	
	In view of Lemma \ref{lem8}, setting $d=:\frac{1}{7}(R-r)$, for any $(T_0,T_1)\subset\subset I$, we have
	\begin{align*}
		&\quad\|u\|_{L^\infty(B_{R-d}(x_0)\times(T_0,T_1))}+\operatorname*{ess\,sup}_{t\in(T_0,T_1)}\mathrm{Tail}(u;x_0,R-d,t)\\
		&\le C(N,p,s_p)\left( \frac{R}{R-r}\right)^N\left(\|u\|_{L^\infty(B_R(x_0)\times(T_0,T_1))}+\operatorname*{ess\,sup}_{t\in(T_0,T_1)}\mathrm{Tail}(u;x_0,R,t)\right)\\
		&\le C(N,p,s_p)\left( \frac{7}{6}\right)^N\left(\|u\|_{L^\infty(B_R(x_0)\times(T_0,T_1))}+\operatorname*{ess\,sup}_{t\in(T_0,T_1)}\mathrm{Tail}(u;x_0,R,t)\right).
	\end{align*}
	By substituting the above two inequalities into \eqref{4.0} with $(r, R)$ replaced by $(r+d, R-d)$, we derive
	\begin{align}
		\label{6.1}
		&\quad\int_{T_0+S}^{T_1}\int_{B_{r+d}(x_0)}\int_{B_{r+d}(x_0)}\frac{\left| J_{\frac{q}{p}+1}(\tau_hu)(x,t)-J_{\frac{q}{p}+1}(\tau_hu)(y,t)\right|^p}{|x-y|^{N+s_pp}}\,\mathrm{d}x\mathrm{d}y\mathrm{d}t\nonumber\\
		&\le \frac{C}{(R-r)^{q+2}}\int_{T_0}^{T_1}\left[|h|^{\sigma q}\left([u(\cdot,t)]^q_{W^{\sigma,q}(B_R(x_0))}+\|u\|^q_{L^\infty\left( B_{R}(x_0)\times(T_0,T_1)\right) }\right)\right]^{\frac{q-p+1}{q}}\mathrm{d}t\nonumber\\
		&\quad+\frac{C}{(R-r)^{q+1}}\int_{T_0}^{T_1}\left[|h|^{\sigma q}\left([u(\cdot,t)]^q_{W^{\sigma,q}(B_R(x_0))}+\|u\|^q_{L^\infty\left( B_{R}(x_0)\times(T_0,T_1)\right) }\right)\right]^{\frac{q-p+1}{q}}\mathrm{d}t\nonumber\\
		&\quad+\frac{C}{(R-r)^{q+4}}\left( \int_{T_0}^{T_1}[u(\cdot,t)]^{q}_{W^{\sigma,q}(B_R(x_0))}\,\mathrm{d}t\right)^\frac{p-2}{p}\nonumber\\
		&\quad\times\left(\int_{T_0}^{T_1}|h|^{\sigma q}\left([u(\cdot,t)]^q_{W^{\sigma,q}(B_R(x_0))}+\|u\|^q_{L^\infty\left( B_{R}(x_0)\times(T_0,T_1)\right) }\right)\mathrm{d}t\right)^{\frac{q-p+2}{q}}\nonumber\\
		&\quad+\frac{C}{S(R-r)^{q+2}}\int_{T_0}^{T_1}\left[|h|^{\sigma q}\left([u(\cdot,t)]^q_{W^{\sigma,q}(B_R(x_0))}+\|u\|^q_{L^\infty\left( B_{R}(x_0)\times(T_0,T_1)\right) }\right)\right]^{\frac{q-p+2}{q}}\mathrm{d}t\nonumber\\
		&\quad+\frac{C\mathcal{T}^{p-2}}{(R-r)^{N+q+2}}\int_{T_0}^{T_1}\left[|h|^{\sigma q}\left([u(\cdot,t)]^q_{W^{\sigma,q}(B_R(x_0))}+\|u\|^q_{L^\infty\left( B_{R}(x_0)\times(T_0,T_1)\right) }\right)\right]^{\frac{q-p+2}{q}}\mathrm{d}t\nonumber\\
		&\quad+\frac{C\mathcal{T}^{p-1}|h|}{(R-r)^{N+q+2}}\int_{T_0}^{T_1}\left[|h|^{\sigma q}\left([u(\cdot,t)]^q_{W^{\sigma,q}(B_R(x_0))}+\|u\|^q_{L^\infty\left( B_{R}(x_0)\times(T_0,T_1)\right) }\right)\right]^{\frac{q-p+1}{q}}\,\mathrm{d}t.
	\end{align}
	
	Next, we make use of Lemma \ref{lem7} to estimate the term on the left-hand side of \eqref{6.1} from blow to get
	\begin{align}
		\label{6.2}
		&\quad\int_{T_0+S}^{T_1}\int_{B_r(x_0)}\left|\tau_\lambda\left(J_{\frac{q}{p}+1}\left(\tau_hu\right) \right) \right|^p\,\mathrm{d}x\mathrm{d}t\nonumber\\
		&\le C|\lambda|^{s_pp}\int_{T_0+S}^{T_1}\int_{B_{r+d}(x_0)}\int_{B_{r+d}(x_0)}\frac{\left| J_{\frac{q}{p}+1}(\tau_hu)(x,t)-J_{\frac{q}{p}+1}(\tau_hu)(y,t)\right|^p}{|x-y|^{N+s_pp}}\,\mathrm{d}x\mathrm{d}y\mathrm{d}t\nonumber\\
		&\quad+\frac{C|\lambda|^{s_pp}}{s_p(R-r)^p}\int_{T_0+S}^{T_1}\int_{B_{r+d}(x_0)}|\tau_hu|^q\,\mathrm{d}x\mathrm{d}t
	\end{align}
	holds true for any $\lambda\in B_d\backslash\left\lbrace0\right\rbrace $. Combining \eqref{6.1} and \eqref{6.2}, we infer
	\begin{align}
		\label{6.3}
		&\quad\int_{T_0+S}^{T_1}\int_{B_r(x_0)}\left|\tau_\lambda\left(J_{\frac{q}{p}+1}\left(\tau_hu\right) \right) \right|^p\,\mathrm{d}x\mathrm{d}t\nonumber\\
		&\le \frac{C|\lambda|^{s_pp}}{(R-r)^{q+2}}\int_{T_0}^{T_1}\left[|h|^{\sigma q}\left([u(\cdot,t)]^q_{W^{\sigma,q}(B_R(x_0))}+\|u\|^q_{L^\infty\left( B_{R}(x_0)\times(T_0,T_1)\right) }\right)\right]^{\frac{q-p+1}{q}}\mathrm{d}t\nonumber\\
		&\quad+\frac{C|\lambda|^{s_pp}}{(R-r)^{q+4}}\left( \int_{T_0}^{T_1}[u(\cdot,t)]^{q}_{W^{\sigma,q}(B_R(x_0))}\,\mathrm{d}t\right)^\frac{p-2}{p}\nonumber\\
		&\quad\times\left(\int_{T_0}^{T_1}|h|^{\sigma q}\left([u(\cdot,t)]^q_{W^{\sigma,q}(B_R(x_0))}+\|u\|^q_{L^\infty\left( B_{R}(x_0)\times(T_0,T_1)\right) }\right)\mathrm{d}t\right)^{\frac{q-p+2}{q}}\nonumber\\
		&\quad+\frac{C|\lambda|^{s_pp}}{S(R-r)^{q+2}}\int_{T_0}^{T_1}\left[|h|^{\sigma q}\left([u(\cdot,t)]^q_{W^{\sigma,q}(B_R(x_0))}+\|u\|^q_{L^\infty\left( B_{R}(x_0)\times(T_0,T_1)\right) }\right)\right]^{\frac{q-p+2}{q}}\mathrm{d}t\nonumber\\
		&\quad+\frac{C\mathcal{T}^{p-2}|\lambda|^{s_pp}}{(R-r)^{N+q+2}}\int_{T_0}^{T_1}\left[|h|^{\sigma q}\left([u(\cdot,t)]^q_{W^{\sigma,q}(B_R(x_0))}+\|u\|^q_{L^\infty\left( B_{R}(x_0)\times(T_0,T_1)\right) }\right)\right]^{\frac{q-p+2}{q}}\mathrm{d}t\nonumber\\
		&\quad+\frac{C\mathcal{T}^{p-1}|\lambda|^{s_pp}|h|}{(R-r)^{N+q+2}}\int_{T_0}^{T_1}\left[|h|^{\sigma q}\left([u(\cdot,t)]^q_{W^{\sigma,q}(B_R(x_0))}+\|u\|^q_{L^\infty\left( B_{R}(x_0)\times(T_0,T_1)\right) }\right)\right]^{\frac{q-p+1}{q}}\,\mathrm{d}t\nonumber\\
		&\quad+\frac{C|\lambda|^{s_pp}}{(R-r)^{(R-r)^{q+p}}}\int_{T_0+S}^{T_1}\left[|h|^{\sigma q}\left([u(\cdot,t)]^q_{W^{\sigma,q}(B_R(x_0))}+\|u\|^q_{L^\infty\left( B_{R}(x_0)\times(T_0,T_1)\right) }\right)\right]\,\mathrm{d}t.
	\end{align}
	Setting $\lambda=h$ and utilizing Lemma \ref{lem9}, one can use the fact that $0<|h|<d<R<1$ to simplify \eqref{6.3} and deduce
	\begin{align}
		\label{6.4}
	&\quad	\int_{T_0+S}^{T_1}\int_{B_r(x_0)}\left|\tau_h(\tau_hu)\right|^q\,\mathrm{d}x\mathrm{d}t\nonumber\\
		&\le \frac{C\left[\left( \mathcal{T}+1\right)^q+\int_{T_0}^{T_1}[u(\cdot,t)]^q_{W^{\sigma,q}\left(B_R(x_0)\right)}\,\mathrm{d}t\right]}{S(R-r)^{N+2q+2}}|h|^{s_pp+\sigma(q-p+1)}.
	\end{align}
	
	Applying Lemma \ref{lem10}, we get the inequality regarding the first order difference $\tau_hu$ and step size $|h|$,
	\begin{align*}
		&\quad\int_{T_0+S}^{T_1}\int_{B_r(x_0)}|\tau_hu|^q\,\mathrm{d}x\mathrm{d}t\\
		&\le C\left[\frac{\left( \mathcal{T}+1\right)^q+\int_{T_0}^{T_1}[u(\cdot,t)]^q_{W^{\sigma,q}\left(B_R(x_0)\right)}\,\mathrm{d}t}{S(R-r)^{N+2q+2}}+\frac{\|u\|^q_{L^\infty\left(B_R(x_0)\times(T_0,T_1)\right) }}{(R-r)^{s_pp+\sigma(q-p+1)}}\right]|h|^{s_pp+\sigma(q-p+1)}\\
		&\le\frac{C\left[\left( \mathcal{T}+1\right)^q+\int_{T_0}^{T_1}[u(\cdot,t)]^q_{W^{\sigma,q}\left(B_R(x_0)\right)}\,\mathrm{d}t\right] }{S(R-r)^{N+2q+2}}|h|^{s_pp+\sigma(q-p+1)}
	\end{align*}
	for any $h\in B_{\frac{1}{2}d}\backslash\left\lbrace 0\right\rbrace$. A direct application of Lemma \ref{lem11} yields the desired estimate
	\begin{align*}
		&\quad\int_{T_0+S}^{T_1}[u(\cdot,t)]^q_{W^{\alpha,q}(B_r(x_0))}\,\mathrm{d}t\\
		&\le C\left[ \frac{\left( \mathcal{T}+1\right)^q+\int_{T_0}^{T_1}[u(\cdot,t)]^q_{W^{\sigma,q}\left(B_R(x_0)\right)}\,\mathrm{d}t}{(\beta-\alpha)S(R-r)^{N+2q+2}}+\frac{\|u\|^q_{L^\infty\left(B_R(x_0)\times(T_0,T_1)\right) }}{\alpha(R-r)^\alpha q}\right]\\
		&\le \frac{C}{S(R-r)^{N+2q+2}}\left[\left(\mathcal{T}+1\right)^q+\int_{T_0+S}^{T_1}[u(\cdot,t)]^q_{W^{\sigma,q}(B_R(x_0))}\,\mathrm{d}t\right].
	\end{align*}
\end{proof}

\begin{lemma}
	\label{lem12}
	Let $p\ge2$, $s_p\in\left( \frac{p-1}{p},1\right)$, $q\ge p$ and suppose that $u\in L^q_{\mathrm{loc}}\bigl(I;W^{\sigma,q}_{\mathrm{loc}}(\Omega)\bigl)$ is locally bounded weak solution to problem \eqref{1.1} in the sense of Definition \ref{def1} with $\sigma\in\left(\max\left\lbrace\mu^{s_p}_p,0\right\rbrace,1  \right) $. Then the following conclusions hold:
	\begin{itemize}
		\item[(i)] If $\sigma\in \left(\max\left\lbrace\mu^{s_p}_p,0\right\rbrace,\frac{q-s_pp}{q-p+1}\right]$, we have
		\begin{align*}
			u\in L^q_{\mathrm{loc}}\bigl(I;W^{\alpha,q}_{\mathrm{loc}}(\Omega)\bigl)
		\end{align*}
		for any $\alpha\in(\sigma,\beta)$ with
		\begin{align*}
			\beta:=\left(1-\frac{p-1}{q}\right)\sigma+\frac{s_pp}{q}. 
		\end{align*}
		Moreover, there exists a constant $C_1$ depending on $N,p,q,s_1,s_p,\alpha,\sigma$ such that for any ball $B_R(x_0)\subset\subset\Omega$ with $R\in(0,1)$, $r\in (0,R)$, $(T_0,T_1)\subset\subset I$ and $S\in \left( 0,\min\left\lbrace T_1-T_0,1\right\rbrace \right) $,
		\begin{align*}
			\int_{T_0+S}^{T_1}[u(\cdot,t)]^q_{W^{\alpha,q}\left( B_r(x_0)\right)}\mathrm{d}t\le\frac{C_1\left[\left(\mathcal{T}+1\right)^q+\int_{T_0}^{T_1}[u(\cdot,t)]^q_{W^{\sigma,q}\left( B_R(x_0)\right)}\mathrm{d}t\right]}{S(R-r)^{N+2q+2}};
		\end{align*}
		\item[(ii)] If $\sigma\in\left(\frac{q-s_pp}{q-p+1},1\right)$, we have
		\begin{align*}
			u\in L^q_{\mathrm{loc}}\bigl(I;W^{1,q}_{\mathrm{loc}}(\Omega)\bigl).
		\end{align*}
		Moreover, there exists a constant $C_2$ depending on $N,p,q,s_1,s_p,\sigma$ such that for any ball $B_R(x_0)\subset\subset\Omega$ with $R\in(0,1)$, $r\in (0,R)$, $(T_0,T_1)\subset\subset I$ and $S\in \left( 0,\min\left\lbrace T_1-T_0,1\right\rbrace \right) $,
		\begin{align*}
			\int_{T_0+S}^{T_1}\int_{B_r(x_0)}|\nabla u|^q\,\mathrm{d}x\mathrm{d}t\le\frac{C_2\left[\left(\mathcal{T}+1\right)^q+\int_{T_0}^{T_1}[u(\cdot,t)]^q_{W^{\sigma,q}\left( B_R(x_0)\right)}\mathrm{d}t\right]}{S(R-r)^{N+2q+2}}.
		\end{align*}
	\end{itemize}
\end{lemma}

\begin{proof}
	For the case $\sigma\in \left(\max\left\lbrace\mu^{s_p}_p, 0\right\rbrace, \frac{q-s_pp}{q-p+1}\right]$, from the proof of Lemma \ref{lem6}, we can get the desired estimate directly.
	
	For the case $\sigma=\frac{q-s_pp}{q-p+1}$, note that $s_pp+\sigma(q-p+1)=q$, we apply Lemma \ref{lem10} on \eqref{6.4} to have
	\begin{align*}
		&\quad\int_{T_0+S}^{T_1}\int_{B_r(x_0)}|\tau_hu|^q\,\mathrm{d}x\mathrm{d}t\\
		&\le C\left[\frac{\left( \mathcal{T}+1\right)^q+\int_{T_0}^{T_1}[u(\cdot,t)]^q_{W^{\sigma,q}\left(B_R(x_0)\right)}\,\mathrm{d}t}{(1-\alpha)^qS(R-r)^{N+2q+2}}+\frac{\|u\|^q_{L^\infty\left(B_R(x_0)\times(T_0,T_1)\right) }}{(R-r)^{q}}\right]|h|^{\alpha q}\\
		&\le\frac{C\left[\left( \mathcal{T}+1\right)^q+\int_{T_0}^{T_1}[u(\cdot,t)]^q_{W^{\sigma,q}\left(B_R(x_0)\right)}\,\mathrm{d}t\right]}{S(R-r)^{N+2q+2}}|h|^{\alpha q}.
	\end{align*}
	The above estimate and Lemma \ref{lem11} enable us to obtain
	\begin{align*}
		\int_{T_0+S}^{T_1}[u(\cdot,t)]^q_{W^{\alpha,q}(B_r(x_0))}\,\mathrm{d}t\le \frac{C\left[\left( \mathcal{T}+1\right)^q+\int_{T_0}^{T_1}[u(\cdot,t)]^q_{W^{\sigma,q}\left(B_R(x_0)\right)}\,\mathrm{d}t\right]}{S(R-r)^{N+2q+2}}.
	\end{align*} 
	
	Similarly, for the case $\sigma\in \left(\frac{q-s_pp}{q-p+1},1\right)$, Lemma \ref{lem10} implies
	\begin{align*}
		\int_{T_0+S}^{T_1}\int_{B_r(x_0)}|\tau_hu|^q\,\mathrm{d}x\mathrm{d}t\le \frac{C\left[\left( \mathcal{T}+1\right)^q+\int_{T_0}^{T_1}[u(\cdot,t)]^q_{W^{\sigma,q}\left(B_R(x_0)\right)}\,\mathrm{d}t\right]}{S(R-r)^{N+2q+2}}|h|^{q}.
	\end{align*}
	We conclude from Lemma \ref{lem12.1} that
	\begin{align*}
		\int_{T_0+S}^{T_1}\int_{B_r(x_0)}|\nabla u|^q\,\mathrm{d}x\mathrm{d}t\le \frac{C\left[\left( \mathcal{T}+1\right)^q+\int_{T_0}^{T_1}[u(\cdot,t)]^q_{W^{\sigma,q}\left(B_R(x_0)\right)}\,\mathrm{d}t\right]}{S(R-r)^{N+2q+2}}.
	\end{align*}
This complets the proof.
\end{proof}

The preceding estimates provide the fundamental energy inequality that allows us to improve the Sobolev regularity of weak solutions via an iterative argument. In the lemma below, under an appropriate integrability condition, we establish the desired weak differentiability by means of repeated applications of Lemmas \ref{lem6} and \ref{lem12}. Finally, through a further iteration step, we improve the integrability exponent and thus remove the integrability assumptions imposed in the previous lemmas.

\begin{lemma}
	\label{lem13}
	Let $p\ge2$, $q\ge p$ and suppose that $u\in L^q_{\mathrm{loc}}\bigl(0,T;W^{\sigma,q}_{\mathrm{loc}}(\Omega)\bigl)$ is a locally bounded weak solution to problem \eqref{1.1} in the sense of Definition \ref{def1} with $\sigma\in \left( \max\left\lbrace\mu^{s_p}_p,0\right\rbrace,s_p\right]$. Then the following conclusions hold:
	\begin{itemize}
		\item[(i)] If $s_p\in\left(0,\frac{p-1}{p}\right]$, we have
		\begin{align*}
			u\in L^q_{\mathrm{loc}}\bigl(I;W^{\alpha,q}_{\mathrm{loc}}(\Omega)\bigl)
		\end{align*}
		for any $\alpha\in\left( \sigma,\frac{s_pp}{p-1}\right) $. Moreover, there exist constants $C_1,\kappa_1,\kappa_2$ depending on $N,p,q,s_1,s_p,\alpha$ and $\sigma$ such that for any ball $B_R(x_0)\subset\subset\Omega$ with $R\in(0,1)$, $r\in (0,R)$, $(T_0,T_1)\subset\subset(0,T)$ and $S\in \left( 0,\min\left\lbrace T_1-T_0,1\right\rbrace \right) $,
		\begin{align*}
			\int_{T_0+S}^{T_1}[u(\cdot,t)]^q_{W^{\alpha,q}\left( B_r(x_0)\right) }\mathrm{d}t\le\frac{C_1\left[\left(\mathcal{T}+1\right)^q+\int_{T_0}^{T_1}[u(\cdot,t)]^q_{W^{\sigma,q}\left( B_R(x_0)\right)}\mathrm{d}t\right]}{S^{\kappa_2}(R-r)^{\kappa_1}};
		\end{align*}
		\item[(ii)] If $s_p\in \left(\frac{p-1}{p},1\right)$, we have
		\begin{align*}
			u\in L^q_{\mathrm{loc}}\bigl(I;W^{1,q}_{\mathrm{loc}}(\Omega)\bigl).
		\end{align*}
		Moreover, there exist constants $C_2,\kappa_3,\kappa_4$ depending on $N,p,q,s_1,s_p$ and $\sigma$ such that for any ball $B_R(x_0)\subset\subset\Omega$ with $R\in(0,1)$, $r\in (0,R)$, $(T_0,T_1)\subset\subset(0,T)$ and $S\in \left( 0,\min\left\lbrace T_1-T_0,1\right\rbrace \right) $,
		\begin{align*}
			\int_{T_0+S}^{T_1}\int_{B_r(x_0)}|\nabla u|^q\,\mathrm{d}x\mathrm{d}t\le\frac{C_2\left[\left(\mathcal{T}+1\right)^q+\int_{T_0}^{T_1}[u(\cdot,t)]^q_{W^{\sigma,q}\left( B_R(x_0)\right)}\mathrm{d}t\right]}{S^{\kappa_4}(R-r)^{\kappa_3}}.
		\end{align*}
	\end{itemize}
\end{lemma}

\begin{proof}
	We first consider the case $s_p\in \left( 0,\frac{p-1}{p}\right] $. For fixed $\alpha\in \left(\sigma,\frac{s_pp}{p-1} \right)$, by setting $\tilde{\alpha}:=\frac{1}{2}\left( \alpha+\frac{s_pp}{p-1}\right) $ and defining the sequences $\left\lbrace\alpha_i\right\rbrace $, $\left\lbrace\rho_i\right\rbrace $, $\left\lbrace\beta_i\right\rbrace $, $\left\lbrace S_i\right\rbrace $ and $\left\lbrace\mathcal{T}_i\right\rbrace$ as follows
	\begin{align*}
		\begin{cases}
			\alpha_0:=\sigma,\,\alpha_{i+1}:=\left( 1-\frac{p-1}{q}\right)\alpha_i+\frac{\tilde{\alpha}(p-1)}{q},\\
			\rho_i:=r+\frac{R-r}{2^i},\\
			\beta_i:=\left( 1-\frac{p-1}{q}\right)\sigma_{i-1}+\frac{s_pp}{q},\\
			S_i:=S\left(1-\frac{1}{2^i}\right),\\
			\mathcal{T}_i:=\|u\|_{L^\infty\left(B_{\rho_i}(x_0)\times(T_0,T_1)\right)}+\operatorname*{ess\,sup}\limits_{t\in(T_0,T_1)}\mathrm{Tail}\left(u;x_0,\rho_i,t\right).   
		\end{cases}
	\end{align*}
	One can verify that $\alpha_i$ goes to $\tilde{\alpha}$ as $i$ goes to $+\infty$. This ensures that we can improve differentiability of weak solution to $\alpha$ with a finite number of iterations. For any $i\in\mathbb{N}^+$, by Lemma \ref{lem6}, we have the estimate
	\begin{align}
		\label{13.1}
	&\quad	\int_{T_0+S_i}^{T_1}[u(\cdot,t)]^q_{W^{\alpha_i,q}(B_{\rho_i}(x_0))}\,\mathrm{d}t \nonumber\\
		&\le\frac{C_i\left[\left(\mathcal{T}_{i-1}+1\right)^q+\int_{T_0+S_{i-1}}^{T_1}[u(\cdot,t)]^q_{W^{\alpha_{i-1},q}\left( B_{\rho_{i-1}}(x_0)\right)}\mathrm{d}t\right]}{\frac{S}{2^i}(\rho_{i-1}-\rho_i)^{N+2q+2}}.
	\end{align}
	From Lemma \ref{lem8}, we get the facts
	\begin{align}
		\label{13.2}
		\begin{cases}
			\mathcal{T}_i\le C(N)2^{iN}\mathcal{T},\\
			\frac{1}{\rho_{i-1}-\rho_i}\le \frac{2^i}{R-r}.
		\end{cases}
	\end{align}
	Bringing \eqref{13.2} into \eqref{13.1} we obtain
	\begin{align}
		\label{13.3}
		&\quad\int_{T_0+S_i}^{T_1}[u(\cdot,t)]^q_{W^{\alpha_i,q}(B_{\rho_i}(x_0))}\,\mathrm{d}t\nonumber\\
		&\le2^{i(2N+2q+3)}\frac{C_i\left[\left(\mathcal{T}+1\right)^q+\int_{T_0+S_{i-1}}^{T_1}[u(\cdot,t)]^q_{W^{\alpha_{i-1},q}\left( B_{\rho_{i-1}}(x_0)\right)}\mathrm{d}t\right]}{S(R-r)^{N+2q+2}}.
	\end{align}
	We now can show that $C_i$ here is independent of $i$. Indeed, this constant can be estimated as
	\begin{align*}
		C_i\le C_*:=\frac{C(N,p,q)}{s_1s_p(1-s_1)\left[ N-\frac{q}{q-p+2}(N+s_pp-2-\sigma p+2\sigma)\right]^\frac{q-p+2}{q}\sigma^2\left(\frac{s_pp}{p-1}-\alpha\right) }.
	\end{align*}
	Iterating \eqref{13.3} with the index $i$ from $1$ to $i_0\in\mathbb{N}^+$, we deduce
	\begin{align}
		\label{13.4}
		&\quad\int_{T_0+S_{i_0}}^{T_1}[u(\cdot,t)]^q_{W^{\alpha_{i_0},q}(B_{\rho_{i_0}}(x_0))}\,\mathrm{d}t\nonumber\\
		&\le2^{i_0+(2N+2q+3)\frac{(1+i_0)i_0}{2} }\frac{C_{*}^{i_0}\left[\left(\mathcal{T}+1\right)^q+\int_{T_0}^{T_1}[u(\cdot,t)]^q_{W^{\sigma,q}\left( B_{R}(x_0)\right)}\mathrm{d}t\right]}{S^{i_0}(R-r)^{(N+2q+2)i_0}}.
	\end{align}
	In this step, we set $\alpha_{i_0}>\alpha$. Since $\alpha_i=\left(\frac{q-p+1}{q}\right)^i\left( \sigma-\tilde{\alpha}\right)+\tilde{\alpha}$, one can choose
	\begin{align*}
		i_0:=\left[\frac{\ln \frac{\tilde{\alpha}-\alpha}{\tilde{\alpha}-\sigma}}{\ln\frac{q-p+1}{q}}\right]+1. 
	\end{align*}
	Therefore, from \eqref{13.4}, there holds
	\begin{align}
		\label{13.5}
		&\quad\int_{T_0+S}^{T_1}[u(\cdot,t)]^q_{W^{\alpha,q}(B_{r}(x_0))}\,\mathrm{d}t\nonumber\\
		&=\int_{T_0+S}^{T_1}\int_{B_{r}(x_0)}\int_{B_{r}(x_0)}\frac{|u(x,t)-u(y,t)|^q}{|x-y|^{N+\alpha_{i_0}q+(\alpha q-\alpha_{i_0}q)}}\,\mathrm{d}x\mathrm{d}y\mathrm{d}t\nonumber\\
		&\le2^q \int_{T_0+S_{i_0}}^{T_1}\int_{B_{\rho_{i_0}}(x_0)}\int_{B_{\rho_{i_0}}(x_0)}\frac{|u(x,t)-u(y,t)|^q}{|x-y|^{N+\alpha_{i_0}q}}\,\mathrm{d}x\mathrm{d}y\mathrm{d}t\nonumber\\
		&\le2^{q+i_0+(2N+2q+3)\frac{(1+i_0)i_0}{2} }\frac{C_{*}^{i_0}\left[\left(\mathcal{T}+1\right)^q+\int_{T_0}^{T_1}[u(\cdot,t)]^q_{W^{\sigma,q}\left( B_{R}(x_0)\right)}\mathrm{d}t\right]}{S^{i_0}(R-r)^{(N+2q+2)i_0}}.
	\end{align}
	Letting $C:=2^{q+i_0+(2N+2q+3)\frac{(1+i_0)i_0}{2} }C_{*}^{i_0}$, $\kappa_1:=(N+2q+2)i_0$ and $\kappa_2:=i_0$, we get the desired result.
	
	Next, we prove the result in another case. To this end, we set
	\begin{align*}
		\gamma:=\frac{2q-p+1-s_pp}{2(q-p+1)},\quad \tilde{\gamma}:=\frac{1+\gamma}{2}
	\end{align*}
	and define the sequence $\left\lbrace\gamma_i\right\rbrace$ as
	\begin{align*}
		\gamma_0:=\sigma,\quad\gamma_{i+1}:=\left(1-\frac{p-1}{q}\right)\gamma_i+\frac{\tilde{\gamma}(p-1)}{q}. 
	\end{align*}
	Applying the method above, one can obtain an estimate similar to \eqref{13.5}, that is, for $k_0\in \mathbb{N}^+$,
	\begin{align*}
		&\quad\int_{T_0+S_{k_0}}^{T_1}[u(\cdot,t)]^q_{W^{\gamma,q}(B_{k_0}(x_0))}\,\mathrm{d}t\nonumber\\
		&\le2^{q+k_0+(2N+2q+3)\frac{(1+k_0)k_0}{2} }\frac{C_{*}^{k_0}\left[\left(\mathcal{T}+1\right)^q+\int_{T_0}^{T_1}[u(\cdot,t)]^q_{W^{\sigma,q}\left( B_{R}(x_0)\right)}\mathrm{d}t\right]}{S^{i_0}(R-r)^{(N+2q+2)i_0}},
	\end{align*}
	where we let
	\begin{align*}
		k_0:=\left[\frac{\ln\frac{\tilde{\gamma}-\gamma}{\tilde{\gamma}-\sigma}}{\ln\frac{q-p+1}{q}}\right]+1 
	\end{align*}
	so that $\gamma_{k_0}>\gamma$. Utilizing (ii) of Lemma \ref{lem12} with parameters $R,\sigma,T_0,S$ replaced by $\rho_{k_0},\gamma,T_0+S_{k_0},S-S_{k_0}$ respectively, we arrive at
	\begin{align*}
		&\quad\int_{T_0+S}^{T_1}\int_{B_r(x_0)}|\nabla u|^q\,\mathrm{d}x\mathrm{d}t\\
		&\le\frac{4^{k_0}C_L\left[\left(\mathcal{T}+1\right)^q+\int_{T_0+S_{k_0}}^{T_1}[u(\cdot,t)]^q_{W^{\gamma,q}\left( B_{\rho_{k_0}}(x_0)\right)}\mathrm{d}t\right]}{S(R-r)^{N+2q+2}}\\
		&\le 2^{q+3k_0+1+(2N+2q+3)\frac{(1+k_0)k_0}{2}}\frac{C_*^{k_0}C_L\left[\left(\mathcal{T}+1\right)^q+\int_{T_0+S_{k_0}}^{T_1}[u(\cdot,t)]^q_{W^{\gamma,q}\left( B_{\rho_{k_0}}(x_0)\right)}\mathrm{d}t\right]}{S^{k_0+1}(R-r)^{(N+2q+2)(k_0+1)}}.
	\end{align*}
	Setting $C_2:=2^{q+3k_0+1+(2N+2q+3)\frac{(1+k_0)k_0}{2}}C_*^{k_0}C_L$, $\kappa_3:=(N+2q+2)(k_0+1)$ and $\kappa_4:=k_0+1$, we complete the proof.
\end{proof}

Combining the estimates obtained above, we are now in a position to establish the Sobolev regularity of weak solutions with respect to spatial variable in the super-quadratic case.

\begin{theorem}
	\label{th14}
	Let $p\ge2$ and suppose $u$ is a locally bounded weak solution to problem \eqref{1.1} in the sense of Definition \ref{def1}. Then, for any $q\ge p$, the following conclusions hold:
	\begin{itemize}
		\item[(i)] If $s_p\in \left( 0,\frac{p-1}{p}\right]$, we have
		\begin{align*}
			u\in L^q_{\mathrm{loc}}\bigl(I;W^{\alpha,q}_{\mathrm{loc}}(\Omega)\bigl)
		\end{align*}
		for any $\alpha\in\left(s_p,\frac{s_pp}{p-1}\right)$. Moreover, there exist constants $C_1,\lambda_1,\lambda_2$ depending on $N,p,q,s_1,s_p$ and $\alpha$ such that for any ball $B_R(x_0)\subset\subset\Omega$ with $R\in(0,1)$, $r\in (0,R)$, $(T_0,T_1)\subset\subset I$ and $S\in \left( 0,\min\left\lbrace T_1-T_0,1\right\rbrace \right) $,
		\begin{align*}
			\int_{T_0+S}^{T_1}[u(\cdot,t)]^q_{W^{\alpha,q}\left( B_r(x_0)\right) }\mathrm{d}t\le\frac{C_1\left[\left(\mathcal{T}+1\right)^q+\left( \int_{T_0}^{T_1}[u(\cdot,t)]^p_{W^{s_p,p}\left( B_R(x_0)\right)}\mathrm{d}t\right)^{\frac{q}{p}} \right]}{S^{\lambda_2}(R-r)^{\lambda_1}};
		\end{align*}
		\item[(ii)] If $s_p\in \left( \frac{p-1}{p},1\right)$, we have
		\begin{align*}
			u\in L^q_{\mathrm{loc}}\bigl(I;W^{1,q}_{\mathrm{loc}}(\Omega)\bigl).
		\end{align*}
		Moreover, there exist constants $C_2,\theta_1,\theta_2$ depending on $N,p,q,s_1$ and $s_p$ such that for any ball $B_R(x_0)\subset\subset\Omega$ with $R\in(0,1)$, $r\in (0,R)$, $(T_0,T_1)\subset\subset(0,T)$ and $S\in \left( 0,\min\left\lbrace T_1-T_0,1\right\rbrace \right) $,
		\begin{align*}
			\int_{T_0+S}^{T_1}\int_{B_r(x_0)}|\nabla u|^q\,\mathrm{d}x\mathrm{d}t\le\frac{C_2\left[\left(\mathcal{T}+1\right)^q+\left( \int_{T_0}^{T_1}[u(\cdot,t)]^p_{W^{s_p,p}\left( B_R(x_0)\right)}\mathrm{d}t\right)^{\frac{q}{p}} \right]}{S^{\theta_2}(R-r)^{\theta_1}}.
		\end{align*}
	\end{itemize}
\end{theorem}

\begin{proof}
	Since the weak solution $u$ is locally bounded in $\Omega_I$, for $\Omega_0\times(t_0,t_1)\subset\subset\Omega\times I$, we have the following estimate
	\begin{align}
		\label{14.1}
		\int_{t_0}^{t_1}[u(\cdot,t)]^{q_1}_{W^{\frac{\alpha q_2}{q_1},q_1}(\Omega_0)}\,\mathrm{d}t&\le C(q_1,q_2)\|u\|^{q_1-q_2}_{L^\infty\left(\Omega_0\times(t_0,t_1)\right)}\int_{t_0}^{t_1}[u(\cdot,t)]^{q_2}_{W^{\alpha,q_2}(\Omega_0)}\,\mathrm{d}t\nonumber\\
		&\le C(q_1,q_2)\left[\|u\|_{L^\infty\left(\Omega_0\times(t_0,t_1)\right)}+\left( \int_{t_0}^{t_1}[u(\cdot,t)]^{q_2}_{W^{\alpha,q_2}(\Omega_0)}\,\mathrm{d}t\right)^{\frac{1}{q_2}}\right]^{q_1},
	\end{align}
	where we use Young's inequality in the last line.
	
	We first consider the case $s_p\in\left(0,\min\left\lbrace\frac{p-1}{p},\frac{2}{p}\right\rbrace \right]$. From Lemma \ref{lem13} with $\sigma=\frac{s_pp}{q}$ and \eqref{14.1}, we can directly get, for any $\alpha\in \left(s_p,\frac{s_pp}{p-1}\right)$,
	\begin{align*}
		\int_{T_0+S}^{T_1}[u(\cdot,t)]^q_{W^{\alpha,q}\left( B_r(x_0)\right) }\mathrm{d}t&\le\frac{C\left[\left(\mathcal{T}+1\right)^q+\int_{T_0}^{T_1}[u(\cdot,t)]^q_{W^{\frac{s_pp}{q},q}\left( B_R(x_0)\right)}\mathrm{d}t\right]}{S^{\kappa_2}(R-r)^{\kappa_1}}\\
		&\le \frac{C_1\left[\left(\mathcal{T}+1\right)^q+\left( \int_{T_0}^{T_1}[u(\cdot,t)]^p_{W^{s_p,p}\left( B_R(x_0)\right)}\mathrm{d}t\right) ^{\frac{q}{p}}\right]}{S^{\kappa_2}(R-r)^{\kappa_1}}.
	\end{align*}
	
	For the case $s_p\in\left(0,\frac{p-1}{p}\right]\backslash\left(0,\frac{2}{p}\right]$ (in this case, $p>3$), we utilize the sequences $\left\lbrace \rho_i\right\rbrace $, $\left\lbrace \mathcal{T}_i\right\rbrace $, $\left\lbrace S_i\right\rbrace $ defined in proof of Lemma \ref{lem13} and define $\left\lbrace q_i\right\rbrace$ as 
	\begin{align*}
		q_i:=\mu^ip,\quad i\in\mathbb{N}^+
	\end{align*}
	with $\mu\in\left(1,\frac{s_p(p-2)}{s_pp-2}\right)$ being determined later.
	
	We will show the desired result by an induction argument. For $i=1$, we apply (i) in Lemma \ref{lem13} and \eqref{14.1} to deduce
	\begin{align*}
		\int_{T_0+S_1}^{T_1}[u(\cdot,t)]^{q_1}_{W^{\alpha,q_1}\left( B_{\rho_1}(x_0)\right) }\mathrm{d}t&\le\frac{C\left[\left(\mathcal{T}+1\right)^{q_1}+\left( \int_{T_0}^{T_1}[u(\cdot,t)]^p_{W^{s_p,p}\left( B_R(x_0)\right)}\mathrm{d}t\right)^{\frac{q_1}{p}}\right]}{\left( \frac{S}{2}\right) ^{\kappa_2}(\rho_0-\rho_1)^{\kappa_1}}\nonumber\\
		&\le \frac{2^{\kappa_1+\kappa_2}C\left[\left(\mathcal{T}+1\right)^{q_1}+\left( \int_{T_0}^{T_1}[u(\cdot,t)]^p_{W^{s_p,p}\left( B_R(x_0)\right)}\mathrm{d}t\right)^{\frac{q_1}{p}}\right]}{S ^{\kappa_2}(R-r)^{\kappa_1}}.
	\end{align*}
	For any fixed $i\in\mathbb{N}^+$, we assume
	\begin{align}
		\label{14.2}
		\int_{T_0+S_i}^{T_1}[u(\cdot,t)]^{q_i}_{W^{\alpha,q_i}\left( B_{\rho_i}(x_0)\right) }\mathrm{d}t
		\le \frac{C_i\left[\left(\mathcal{T}+1\right)^{q_i}+\left( \int_{T_0}^{T_1}[u(\cdot,t)]^p_{W^{s_p,p}\left( B_R(x_0)\right)}\mathrm{d}t\right)^{\frac{q_i}{p}}\right]}{S ^{\tilde{\lambda}_i}(R-r)^{\lambda'_i}}
	\end{align}
	holds true. The above inequality implies
	\begin{align*}
		u\in L^{q_i}\bigl(T_0+S_i,T_1;W^{\alpha,q_i}\left(B_{\rho_i}(x_0)\right)\bigl).
	\end{align*}
	Observe the fact that $\frac{\alpha}{\mu}>\mu^{s_p}_p$, one can apply (i) in Lemma \ref{lem13} and \eqref{14.1} to get
	\begin{align*}
		&\quad\int_{T_0+S_{i+1}}^{T_1}[u(\cdot,t)]^{q_{i+1}}_{W^{\alpha,q_{i+1}}\left( B_{\rho_{i+1}}(x_0)\right) }\mathrm{d}t\\
		&\le \frac{\widetilde{C}_{i+1}\left[\left(\mathcal{T}+1\right)^{q_{i+1}}+\left( \int_{T_0+S_{i}}^{T_1}[u(\cdot,t)]^{q_i}_{W^{\alpha,q_i}\left( B_{\rho_i}(x_0)\right)}\mathrm{d}t\right)^{\frac{q_{i+1}}{q_i}}\right]}{\left(\frac{S}{2^{i+1}}\right)  ^{\tilde{\kappa}_i}(\rho_{i}-\rho_{i+1})^{\kappa'_i}}\\
		&\le 2^{i(\tilde{\kappa}_i+\kappa'_i)}\frac{\widetilde{C}_{i+1}\left[\left(\mathcal{T}+1\right)^{q_{i+1}}+\left( \int_{T_0+S_{i}}^{T_1}[u(\cdot,t)]^{q_i}_{W^{\alpha,q_i}\left( B_{\rho_i}(x_0)\right)}\mathrm{d}t\right)^{\mu}\right]}{S  ^{\tilde{\kappa}_i}(R-r)^{\kappa'_i}}\\
		&\le 2^{i(\tilde{\kappa}_i+\kappa'_i)}\frac{\widetilde{C}_{i+1}\left[\left(\mathcal{T}+1\right)^{q_{i+1}}+\left( \frac{C_i\left[\left(\mathcal{T}+1\right)^{q_i}+\left( \int_{T_0}^{T_1}[u(\cdot,t)]^p_{W^{s_p,p}\left( B_R(x_0)\right)}\mathrm{d}t\right)^{\frac{q_i}{p}}\right]}{S ^{\tilde{\lambda}_i}(R-r)^{\lambda'_i}}\right)^{\mu}\right]}{S  ^{\tilde{\kappa}_i}(R-r)^{\kappa'_i}}\\
		&\le 2^{\mu+i(\tilde{\kappa}_i+\kappa'_i)+1}\frac{\widetilde{C}_{i+1}C_i\left[\left(\mathcal{T}+1\right)^{q_{i+1}}+\left( \int_{T_0}^{T_1}[u(\cdot,t)]^{p}_{W^{s_p,p}\left( B_{R}(x_0)\right)}\mathrm{d}t\right)^{\frac{q_{i+1}}{p}}\right]}{S  ^{\tilde{\kappa}_i+\tilde{\lambda}\mu}(R-r)^{\kappa'_i+\lambda'_i\mu}}.
	\end{align*}
	Hence, we prove $\text{\eqref{14.2}}_{i+1}$ by set $C_{i+1}:=2^{\mu+i(\tilde{\kappa}_i+\kappa'_i)+1}\widetilde{C}_{i+1}C_i$, $\tilde{\lambda}_{i+1}:=\tilde{\kappa}_i+\tilde{\lambda}\mu$ and $\lambda'_{i+1}:=\kappa'_i+\lambda'_i\mu$ and finish our induction argument.
	
	For any $q\in[p,+\infty)$, we choose $i_0\in \mathbb{N}^+$ as the integer that satisfies
	\begin{align*}
		\left( \frac{s_p(p-2)}{s_pp-2}\right)^{i_0}>\frac{q}{p}. 
	\end{align*}
	That is,
	\begin{align*}
		i_0:=\left[ \frac{\ln\frac{q}{p}}{\ln\frac{s_p(p-2)}{s_pp-2}}\right]+1.
	\end{align*}
	We can set $\mu=\left( \frac{q}{p}\right)^\frac{1}{i_0}$ to obtain the desired estimate.
	
	Next, we consider the case $s_p\in \left( \frac{p-1}{p},1\right)$. When $p\in[2,3]$, we can increase the index of differentiability and integrability by applying (ii) in Lemma \ref{lem13} and \eqref{14.1} to get
	\begin{align*}
		\int_{T_0+S}^{T_1}\int_{B_r(x_0)}|\nabla u|^q\,\mathrm{d}x\mathrm{d}t\le \frac{C\left[\left(\mathcal{T}+1\right)^q+\left( \int_{T_0}^{T_1}[u(\cdot,t)]^p_{W^{s_p,p}\left( B_R(x_0)\right)}\mathrm{d}t\right)^\frac{q}{p}\right]}{S^{\theta_2}(R-r)^{\theta_1}}.
	\end{align*}
	
	Finally, we deal with the case $p>3$. We set
	\begin{align*}
		\mu:=\frac{s_pp-s_p-1}{s_p(p-2)}
	\end{align*}
	and apply the inductive method described earlier. For $i=1$, since $\frac{s_pp}{q_1}>\mu^{s_p}_p$, we have
	\begin{align*}
		\int_{T_0+S_1}^{T_1}\int_{B_{\rho_1}(x_0)}|\nabla u|^{q_1}\,\mathrm{d}x\mathrm{d}t&\le \frac{\widetilde{C}_1\left[\left(\mathcal{T}+1\right)^{q_1}+ \int_{T_0}^{T_1}[u(\cdot,t)]^{q_1}_{W^{\frac{s_pp}{q_1},q_1}\left( B_R(x_0)\right)}\mathrm{d}t\right]}{S^{\kappa_4}(R-r)^{\kappa_3}}\\
		&\le \frac{C_1\left[\left(\mathcal{T}+1\right)^{q_1}+\left( \int_{T_0}^{T_1}[u(\cdot,t)]^{p}_{W^{s_p,p}\left( B_R(x_0)\right)}\mathrm{d}t\right)^\frac{q_1}{p}\right]}{S^{\kappa_4}(R-r)^{\kappa_3}}.
	\end{align*}
	
	For $i\in\mathbb{N}^+$, we suppose the following estimate 
	\begin{align*}
		\int_{T_0+S_i}^{T_1}\int_{B_{\rho_i}(x_0)}|\nabla u|^{q_i}\,\mathrm{d}x\mathrm{d}t\le \frac{C_i\left[\left(\mathcal{T}+1\right)^{q_i}+\left( \int_{T_0}^{T_1}[u(\cdot,t)]^{p}_{W^{s_p,p}\left( B_R(x_0)\right)}\mathrm{d}t\right)^\frac{q_i}{p}\right]}{S^{\tilde{\theta}_i}(R-r)^{\theta'_i}}
	\end{align*}
holds. Then, one can utilize Lemma \ref{lem14.1} and (ii) of Lemma \ref{lem13} and \eqref{14.1} to obtain
	\begin{align*}
		&\quad\int_{T_0+S_{i+1}}^{T_1}\int_{B_{\rho_{i+1}}(x_0)}|\nabla u|^{q_{i+1}}\,\mathrm{d}x\mathrm{d}t\\
		&\le \frac{\widetilde{C}_i\left[\left(\mathcal{T}+1\right)^{q_{i+1}}+ \int_{T_0+S_i}^{T_1}[u(\cdot,t)]^{q_{i+1}}_{W^{\frac{s_pq_i}{q_{i+1}},q_{i+1}}\left( B_{\rho_i}(x_0)\right)}\mathrm{d}t\right] }{S^{\tilde{\kappa}_{i+1}}(R-r)^{\kappa'_{i+1}}}\\
		&\le \frac{\widetilde{C}_i\left[\left(\mathcal{T}+1\right)^{q_{i+1}}+ \left( \int_{T_0+S_i}^{T_1}[u(\cdot,t)]^{q_{i}}_{W^{s_p,q_i}\left( B_{\rho_i}(x_0)\right)}\mathrm{d}t\right)^\mu \right] }{S^{\tilde{\kappa}_{i+1}}(R-r)^{\kappa'_{i+1}}}\\
		&\le \frac{\widetilde{C}_i\left[\left(\mathcal{T}+1\right)^{q_{i+1}}+ C(N,s_p,p)\left( \int_{T_0+S_i}^{T_1}\int_{B_{\rho_i}(x_0)}|\nabla u|^{q_i}\,\mathrm{d}x\mathrm{d}t\right)^\mu \right] }{S^{\tilde{\kappa}_{i+1}}(R-r)^{\kappa'_{i+1}}}\\
		&\le \frac{4\widetilde{C}_iC(N,s_p,p)C_i^\mu\left[\left(\mathcal{T}+1\right)^{q_{i+1}}+ \left( \int_{T_0}^{T_1}[u(\cdot,t)]^{p}_{W^{s_p,p}\left( B_{R}(x_0)\right)}\mathrm{d}t\right)^\frac{q_{i+1}}{p} \right] }{S^{\tilde{\kappa}_{i+1}+\tilde{\theta}_i\mu}(R-r)^{\kappa'_{i+1}+\theta'_{i}\mu}},
	\end{align*}
	where we set $C_{i+1}:=4\widetilde{C}_iC(N,s_p,p)C_i^\mu$, $\tilde{\theta}_{i+1}:=\tilde{\kappa}_{i+1}+\tilde{\theta}_i\mu$ and $\theta'_{i+1}:=\kappa'_{i+1}+\theta'_{i}\mu$ to get the desired estimate and finish the induction.
	
	For any $q\ge p$, we choose $i_0\in\mathbb{N}^+$ that satisfies $q_{i_0}\ge q$. Indeed, we can choose
	\begin{align*}
		i_0:=\left[ \frac{\ln\frac{q}{p}}{\ln\left( \frac{s_pp-s_p-1}{s_p(p-2)}\right)}\right]+1. 
	\end{align*}
	Hence, it follows from H\"{o}lder's inequality that
	\begin{align*}
		&\quad\int_{T_0+S}^{T_1}\int_{B_r(x_0)}|\nabla u|^{q}\,\mathrm{d}x\mathrm{d}t\\
		&\le\left( \int_{T_0+S_{i_0}}^{T_1}\int_{B_{\rho_{i_0}}(x_0)}|\nabla u|^{q_{i_0}}\,\mathrm{d}x\mathrm{d}t\right)^\frac{q}{q_{i_0}}\\
		&\le\frac{C_{i_0}^\mu C(N,\mu)\left[\left(\mathcal{T}+1\right)^q+\left( \int_{T_0}^{T_1}[u(\cdot,t)]^p_{W^{s_p,p}\left( B_R(x_0)\right)}\mathrm{d}t\right)^{\frac{q}{p}} \right]}{S^{\tilde{\theta}_{i_0}}(R-r)^{\theta'_{i_0}}}. 
	\end{align*}
This complets the proof.
\end{proof}

\section{Sobolev regularity  for the sub-quadratic case}
\label{sec4}
In this section, we focus on the sub-quadratic case of equation \eqref{1.1}. Due to the structural differences, a direct estimate of the $W^{s_p,p}$-energy of $J_{\frac{q}{p}+1}(\tau_hu)$ is not feasible. Nevertheless, by choosing the exponent $\gamma$ in \eqref{15.0} as $\sigma p$, the resulting energy aligns with that of the super-quadratic case, allowing the estimates to approach the desired level arbitrarily closely in terms of the exponents.

To begin with, we give a lemma adapted from \cite[Lemma 3.2]{BDL2401}, which provides a crucial control over the nonlocal term. However, it should be noted that the information on the step size revealed by this lemma will be masked by the $1$-growth term.

\begin{lemma}
	\label{lem17}
	Let $p\in(1, 2)$ and $s_p\in(0, 1)$. There exists a constant $C:=\frac{\widetilde{C}(N,p)}{s_p}$ such that whenever $u\in L^\infty\bigl( T_0,T_1;L^{p-1}_{s_pp}(\mathbb{R}^N)\bigl)$, $R>0$, $r\in(0,R)$ and $d\in \left( 0,\frac{1}{4}(R-r)\right] $, for any $x\in B_{\frac{1}{2}(R+r)}(x_0)$, $h\in B_d\backslash \left\lbrace 0\right\rbrace $ and $T_1-T_0\le1$, we have
	\begin{align*}
		&\quad\left| \int_{T_0}^{T_1}\int_{\mathbb{R}^N\backslash B_R(x_0)}\frac{J_p\left( u_h(x)-u_h(y)\right)-J_p\left( u(x)-u(y)\right) }{|x-y|^{N+s_pp}}\,\mathrm{d}y\mathrm{d}t\right|\\
		&\le \frac{C}{R^{s_pp}}\left( \frac{R}{R-r}\right)^{N+s_pp+1}\left( \int_{T_0}^{T_1}|\tau_hu(x)|^{p-1}\,\mathrm{d}t+\frac{|h|}{R}\mathcal{T}^{p-1}\right).
	\end{align*}
\end{lemma}

The overall proof strategy parallels that employed in Section \ref{sec3}, but the presence of additional parameters modifies the form of the resulting constants. Importantly, this does not change the dependence of these constants on the exponents in the equation.

\begin{proposition}
	\label{pro15}
	Let $p\in (1, 2)$, $s_1,s_p\in(0, 1)$, $q\ge p$ and $\sigma\in(0,1)$. Suppose $u\in L^q_{\mathrm{loc}}\bigl(I;W^{\sigma,q}_{\mathrm{loc}}(\Omega)\bigl)$ is a locally bounded weak solution to problem \eqref{1.1} in the sense of Definition \ref{def1}. Then, For any $r\in(0,R)$ with $R\in(0,1)$, $d\in\left(0,\frac{1}{4}(R-r)\right]$, $B_{R+d}(x_0)\times(T_0,T_1)\subset\subset\Omega\times I$ with $S\in\left( 0,\min\left\lbrace T_1-T_0,1\right\rbrace \right)$ and $\gamma\in(0,\sigma p)$, there exists a constant $C:=C(N,p,q,s_1,s_p,\sigma,\gamma)$ such that
	\begin{align}
		\label{15.0}
		&\quad\int_{T_0+S}^{T_1}\int_{B_r(x_0)}\int_{B_r(x_0)}\frac{\left|J_{\frac{q}{p}+1}(\tau_hu)(x,t)-J_{\frac{q}{p}+1}(\tau_hu)(y,t)\right|^p}{|x-y|^{N+\frac{s_pp^2}{2}+\gamma\left(1-\frac{p}{2}\right)}}\,\mathrm{d}x\mathrm{d}y\mathrm{d}t\nonumber\\
		&\le \left(\frac{C}{S(R-r)^{N+s_pp+1}}\right)^\frac{p}{2}\left[ \left( \int_{T_0}^{T_1}[u(\cdot,t)]^q_{W^{\sigma,q}(B_R(x_0))}\,\mathrm{d}t\right)^{\frac{p}{q}} \left(\int_{T_0}^{T_1}\int_{B_r(x_0)}|\tau_hu|^q\,\mathrm{d}x\mathrm{d}t \right)^{\frac{q-p}{q}}\right]^{1-\frac{p}{2}}\nonumber\\
		&\quad\times\left(\int_{T_0}^{T_1}\int_{B_R(x_0)}|\tau_hu|^{q-p+1}\,\mathrm{d}x\mathrm{d}t+\int_{T_0}^{T_1}\int_{B_R(x_0)}|\tau_hu|^{q}\,\mathrm{d}x\mathrm{d}t\right.\nonumber\\
		&\quad\left.+|h|\mathcal{T}^{p-1}\int_{T_0}^{T_1}\int_{B_R(x_0)}|\tau_hu|^{q}\,\mathrm{d}x\mathrm{d}t+\int_{T_0}^{T_1}\int_{B_R(x_0)}|\tau_hu|^{q-p+2}\,\mathrm{d}x\mathrm{d}t\right)^\frac{p}{2}.
	\end{align}
\end{proposition}

\begin{proof}
	Thanks to the proof of Proposition \ref{pro4}, we begin with \eqref{3.0} and utilize Lemma \ref{lem16} to deal with $p$-structure as follows
	\begin{align}
		\label{15.1}
		P_1\ge \frac{p-1}{2^{q-p+2}}P_3-\left( \frac{2^{q-p+2}}{p-1}\right)^{p-1}P_4 
	\end{align}
	and
	\begin{align*}
		|P_2|&\le \int_{T_0}^{T_1}\int_{B_\frac{R+r}{2}(x_0)}\left|\int_{\mathbb{R}^N\backslash B_R(x_0)}\frac{J_{p}(u_h(x,t)-u_h(y,t))-J_p(u(x,t)-u(y,t))}{|x-y|^{N+s_pp}}\,\mathrm{d}y\right|\\
		&\qquad\qquad\qquad\qquad\times\left|J_{q-p+2}(\tau_hu)(x,t)\right|\eta^p(x)\,\mathrm{d}x\mathrm{d}t, 
	\end{align*}
	where $P_1$ and $P_2$ are given in \eqref{4.01}, $P_3$ and $P_4$ are defined as
	\begin{align*}
		P_3&:=\int_{T_0}^{T_1}\int_{B_R(x_0)}\int_{B_R(x_0)}\frac{\left( |u_h(x,t)-u_h(y,t)|+|u(x,t)-u(y,t)|\right)^{p-2}\left( \eta^p(x)+\eta^p(y)\right)}{|x-y|^{N+s_pp}}\\
		&\qquad\qquad\qquad\qquad\qquad\times\left( |\tau_hu(x,t)|+|\tau_hu(y,t)|\right)^{q-p}|\tau_hu(x,t)-\tau_hu(y,t)|^2\xi(t)\,\mathrm{d}x\mathrm{d}y\mathrm{d}t,
	\end{align*}
	and
	\begin{align*}
		P_4:=\int_{T_0}^{T_1}\int_{B_R(x_0)}\int_{B_R(x_0)}\frac{\left( |\tau_hu(x,t)|+|\tau_hu(y,t)|\right)^{q}|\eta(x)-\eta(y)|^p}{|x-y|^{N+s_pp}}\,\mathrm{d}x\mathrm{d}y\mathrm{d}t.
	\end{align*}
	By applying Lemma \ref{lem9} and H\"{o}lder's inequality with exponents $ \frac{2}{p}$ and $\frac{2}{2-p} $, we have
	\begin{align}
		\label{15.2}
		&\quad\int_{T_0+S}^{T_1}\int_{B_r(x_0)}\int_{B_r(x_0)}\frac{\left|J_{\frac{q}{p}+1}(\tau_hu)(x,t)-J_{\frac{q}{p}+1}(\tau_hu)(y,t)\right|^p}{|x-y|^{N+\frac{s_pp^2}{2}+\gamma\left(1-\frac{p}{2}\right)}}\,\mathrm{d}x\mathrm{d}y\mathrm{d}t\nonumber\\
		&\le C(p,q)\int_{T_0+S}^{T_1}\int_{B_r(x_0)}\int_{B_r(x_0)}\frac{\left( |\tau_hu(x,t)|+|\tau_hu(y,t)|\right)^{q-p}|\tau_hu(x,t)-\tau_hu(y,t)|^p}{|x-y|^{N+\frac{s_pp^2}{2}+\gamma\left(1-\frac{p}{2}\right)}}\,\mathrm{d}x\mathrm{d}y\mathrm{d}t\nonumber\\
		&\le C(p,q)P_5^{1-\frac{p}{2}}P_3^{\frac{p}{2}},
	\end{align}
	where
	\begin{align*}
		P_5&:=\int_{T_0}^{T_1}\int_{B_r(x_0)}\int_{B_r(x_0)}\frac{\left( |u_h(x,t)-u_h(y,t)|+|u(x,t)-u(y,t)|\right)^p}{|x-y|^{N+\gamma}}\\
		&\qquad\qquad\qquad\qquad\qquad\times\left( |\tau_hu(x,t)|+|\tau_hu(y,t)|\right)^{q-p}\,\mathrm{d}x\mathrm{d}y\mathrm{d}t.
	\end{align*}
	
	We use H\"{o}lder's inequality again on $P_5$ with exponents $\frac{q}{p}$ and $\frac{q}{q-p}$ to obtain
	\begin{align}
		\label{15.3}
		P_5&\le\left( \int_{T_0}^{T_1}\int_{B_r(x_0)}\int_{B_r(x_0)}\frac{\left( |u_h(x,t)-u_h(y,t)|+|u(x,t)-u(y,t)|\right)^q}{|x-y|^{N+\sigma q}}\,\mathrm{d}x\mathrm{d}y\mathrm{d}t\right)^{\frac{p}{q}}\nonumber\\
		&\quad\times\left( \int_{T_0}^{T_1}\int_{B_r(x_0)}\int_{B_r(x_0)}\frac{\left( |\tau_hu(x,t)|+|\tau_hu(y,t)|\right)^{q}}{|x-y|^{N+(\gamma-\sigma q)\frac{q}{q-p}}}\,\mathrm{d}x\mathrm{d}y\mathrm{d}t\right)^\frac{q-p}{q}\nonumber\\
		&\le C(p,q)\left( \int_{T_0}^{T_1}\int_0^{2R}\frac{\mathrm{d}\rho}{\rho^{(\gamma-\sigma p)\frac{q}{q-p}+1}}\int_{B_r(x_0)}|\tau_hu(x,t)|^q\,\mathrm{d}x\mathrm{d}t\right)^\frac{q-p}{q}\nonumber\\
		&\qquad\qquad\times\left( \int_{T_0}^{T_1}[u(\cdot,t)]^q_{W^{\sigma,q}(B_R(x_0))}\,\mathrm{d}t\right)^{\frac{p}{q}}\nonumber\\
		&\le \frac{C(p,q)}{(\sigma p-\gamma)^\frac{q-p}{q}}\left( \int_{T_0}^{T_1}[u(\cdot,t)]^q_{W^{\sigma,q}(B_R(x_0))}\,\mathrm{d}t\right)^{\frac{p}{q}}\left(\int_{T_0}^{T_1}\int_{B_r(x_0)}|\tau_hu(x,t)|^q\,\mathrm{d}x\mathrm{d}t \right)^\frac{q-p}{q}.
	\end{align}
	For the nonlocal term $P_2$, Lemma \ref{lem17} implies
	\begin{align}
		\label{15.4}
		|P_2|\le \frac{C}{(R-r)^{N+s_pp+1}}\left( \int_{T_0}^{T_1}\int_{B_\frac{R+r}{2}(x_0)}|\tau_hu|^q\,\mathrm{d}x\mathrm{d}t+|h|\mathcal{T}\int_{T_0}^{T_1}\int_{B_R(x_0)}|\tau_hu|^{q-p+1}\,\mathrm{d}x\mathrm{d}t\right).
	\end{align}
	Moreover, one has the estimate
	\begin{align}
		\label{15.5}
		|P_4|&\le C(q)\int_{T_0}^{T_1}\int_{0}^{2R}\frac{\mathrm{d}\rho}{\rho^{-(1-s_p)p+1}}\int_{B_R(x_0)}|\tau_hu|^q\,\mathrm{d}x\mathrm{d}t\nonumber\\
		&\le \frac{C(N,p,q)}{(1-s_p)(R-r)^p}\int_{T_0}^{T_1}\int_{B_R(x_0)}|\tau_hu|^q\,\mathrm{d}x\mathrm{d}t,
	\end{align}
	and
	\begin{align}
		\label{15.6}
		\int_{T_0}^{T_1}\int_{B_R(x_0)}\frac{|\tau_hu|^{q-p+2}}{|q-p+2|}\eta(x)\xi'(t)\,\mathrm{d}x\mathrm{d}t\le\frac{C}{S}\int_{T_0}^{T_1}\int_{B_R(x_0)}|\tau_hu|^{q-p+2}\,\mathrm{d}x\mathrm{d}t.
	\end{align}
	
By substituting \eqref{15.1}–\eqref{15.6} and \eqref{4.1}–\eqref{4.3} into inequality \eqref{3.0}, we arrive at the desired result.
\end{proof}

\begin{remark}
	\label{rem15.1}
	The constant $C$ in \eqref{15.0} takes the form of
	\begin{align*}
		C=\frac{C(N,p,q,s_1,s_p)}{(\sigma p-\gamma)^\frac{q-p}{q}}.
	\end{align*}
\end{remark}

The following two lemmas will play the same role as Lemmas \ref{lem6} and \ref{lem12} in the iteration scheme.

\begin{lemma}
	\label{lem18}
	Let $p\in(1,2)$, $s_p\in\left( 0,\frac{p-1}{p}\right]$, $q\ge p$ and suppose that $u\in L^q_{\mathrm{loc}}\bigl(I;W^{\sigma,q}_{\mathrm{loc}}(\Omega)\bigl)$ is a locally bounded weak solution to problem \eqref{1.1} in the sense of Definition \ref{def1}. Then, we have
	\begin{align*}
		u\in L^q_{\mathrm{loc}}\bigl(I;W^{\alpha,q}_{\mathrm{loc}}(\Omega)\bigl)
	\end{align*}
	for any $\alpha\in(\sigma,\beta)$ with
	\begin{align*}
		\beta:=\sigma\left(1-\frac{p^2}{2q}+\frac{p}{2q}\right)+\frac{s_pp^2}{2q}. 
	\end{align*}
	Moreover, there exists a constant $C$ depending on $N,p,q,s_1,s_p,\sigma$ and $\alpha$ such that for any ball $B_R(x_0)\subset\subset\Omega$ with $R\in(0,1)$, $r\in (0,R)$, $(T_0,T_1)\subset\subset I$ and $S\in \left( 0,\min\left\lbrace T_1-T_0,1\right\rbrace \right) $,
	\begin{align*}
		\int_{T_0+S}^{T_1}[u(\cdot,t)]^q_{W^{\alpha,q}\left( B_r(x_0)\right)}\mathrm{d}t\le\frac{C\left[\left(\mathcal{T}+1\right)^q+\int_{T_0}^{T_1}[u(\cdot,t)]^q_{W^{\sigma,q}\left( B_R(x_0)\right)}\mathrm{d}t\right]}{S(R-r)^{N+2q+1}}.
	\end{align*}
\end{lemma}

\begin{proof}
	We first deal with the right-hand side of \eqref{15.0}. By applying Lemma \ref{lem7}, H\"{o}lder's inequality, Young's inequality and note that $u$ is locally bounded in $\Omega_I$, we have, for $r\in\left\lbrace q-p+1,q-p\right\rbrace$ and $m=q-p+2$,
	\begin{align}
		\label{18.1}
		&\quad\int_{T_0}^{T_1}\int_{B_{R-d}(x_0)}|\tau_hu|^r(x,t)\,\mathrm{d}x\mathrm{d}t\nonumber\\
		&\le \left( \int_{T_0}^{T_1}\int_{B_{R-d}(x_0)}|\tau_hu|^q(x,t)\,\mathrm{d}x\mathrm{d}t\right)^\frac{r}{q} \nonumber\\
		&\le \left[C|h|^{\sigma q}\left((1-\sigma)\int_{T_0}^{T_1}[u]^q_{W^{\sigma,q}(B_R(x_0))}\,\mathrm{d}t+\frac{7R^N}{\sigma(R-r)^q}\|u\|^q_{L^\infty\left( B_{R-d}(x_0)\times(T_0,T_1)\right) }\right)\right]^\frac{r}{q}
	\end{align}
	and
	\begin{align}
		\label{18.01}
		&\quad\int_{T_0}^{T_1}\int_{B_{R-d}(x_0)}|\tau_hu|^{m}(x,t)\,\mathrm{d}x\mathrm{d}t\nonumber\\
		&\le C|h|^{\sigma q}\left(\left(1-\frac{\sigma q}{m}\right)\int_{T_0}^{T_1}[u]^{m}_{W^{\frac{\sigma q}{m},m}(B_R(x_0))}\,\mathrm{d}t+\frac{7R^N}{\sigma(R-r)^m}\|u\|^m_{L^\infty\left( B_{R-d}(x_0)\times(T_0,T_1)\right) }\right)\nonumber\\
		&\le C|h|^{\sigma q}\left(\left(1-\frac{\sigma q}{m}\right)\left( \int_{T_0}^{T_1}[u(\cdot,t)]^{q}_{W^{\sigma,q}(\Omega_0)}\,\mathrm{d}t\right)^{\frac{m}{q}}+\frac{7R^N}{\sigma(R-r)^m}\|u\|^m_{L^\infty\left( B_{R-d}(x_0)\times(T_0,T_1)\right) }\right),
	\end{align}
	where $d:=\frac{1}{4}(R-r)$. Moreover, we replace $r,R$ by $r+d$ and $R-d$ respectively, and bring \eqref{18.1}, \eqref{18.01} into \eqref{15.0} to get
	\begin{align}
		\label{18.2}
		&\quad\int_{T_0+S}^{T_1}\int_{B_r(x_0)}\int_{B_r(x_0)}\frac{\left|J_{\frac{q}{p}+1}(\tau_hu)(x,t)-J_{\frac{q}{p}+1}(\tau_hu)(y,t)\right|^p}{|x-y|^{N+\frac{s_pp^2}{2}+\gamma\left(1-\frac{p}{2}\right)}}\,\mathrm{d}x\mathrm{d}y\mathrm{d}t\nonumber\\
		&\le \frac{C\left( \int_{T_0}^{T_1}[u(\cdot,t)]^q_{W^{\sigma,q}(B_R(x_0))}\,\mathrm{d}t\right)^{\frac{p\left(1-\frac{p}{2} \right) }{q}}}{S^{\frac{p}{2}}(R-r)^{(N+s_pp+1)\frac{p}{2}}}\mathcal{K}^{\frac{q-p}{p}\left(1-\frac{p}{2}\right)}|h|^{\sigma(q-p)\left( 1-\frac{p}{2}\right) }\nonumber\\
		&\quad\times\left[ \mathcal{K}^{\frac{q-p+1}{q}}|h|^{\sigma(q-p+1)}+\mathcal{K}|h|^{\sigma q}+\mathcal{T}^{p-1}\mathcal{K}^{\frac{q-p+1}{q}}|h|^{\sigma(q-p+1)+1}+\mathcal{K}^{\frac{q-p+2}{q}}|h|^{\sigma(q-p+2)}\right]^\frac{p}{2}\nonumber\\
		&\le \frac{C\left( \int_{T_0}^{T_1}[u(\cdot,t)]^q_{W^{\sigma,q}(B_R(x_0))}\,\mathrm{d}t\right)^{\frac{p\left(1-\frac{p}{2} \right) }{q}}}{S(R-r)^{N+s_pp+1}}\widetilde{\mathcal{K}}^{(q-p)\left(1-\frac{p}{2}\right)}|h|^{\sigma(q-p)\left(1-\frac{p}{2}\right)}\nonumber\\
		&\quad\times\left[\widetilde{\mathcal{K}}^{q-p+1}|h|^{\sigma(q-p+1)}+\widetilde{\mathcal{K}}^q|h|^{\sigma q}+\widetilde{\mathcal{K}}^q|h|^{\sigma(q-p+1)+1}+\widetilde{\mathcal{K}}^{q-p+2}|h|^{\sigma(q-p+2)}\right]^\frac{p}{2}\nonumber\\
		&\le\frac{C}{S(R-r)^{N+2q+1}}\widetilde{\mathcal{K}}^q|h|^{\sigma\left(q-\frac{p}{2}\right) },
	\end{align}
	where we set
	\begin{align*}
	&	\mathcal{K}:=\int_{T_0}^{T_1}[u(\cdot,t)]^q_{W^{\sigma,q}(B_R(x_0))}\,\mathrm{d}t+\|u\|^q_{L^\infty\left( B_{R}(x_0)\times(T_0,T_1)\right)},\\
	&	\widetilde{\mathcal{K}}:=\mathcal{T}+\mathcal{K}^\frac{1}{q}+1,
	\end{align*}
	and utilize H\"{o}lder's inequality in the second-to-last line.
	
	Next, one can use the method in the proof of Lemma \ref{lem6} to estimate the first term of \eqref{18.2} as 
	\begin{align}
		\label{18.3}
		&\quad\int_{T_0+S}^{T_1}\int_{B_r(x_0)}|\tau_hu(x,t)|^q\,\mathrm{d}x\mathrm{d}t\nonumber\\
		&\le\frac{C\left[ \left( \mathcal{T}+1\right)^q+\int_{T_0}^{T_1}[u]^q_{W^{\sigma,q}(B_R(x_0))} \,\mathrm{d}t\right] }{S(R-r)^{N+2q+1}}|h|^{\sigma\left( q-\frac{p}{2}\right)+\gamma\left(1-\frac{p}{2}\right)+\frac{s_pp^2}{2}}.
	\end{align}
	For any $\alpha\in (\sigma,\beta)$, we can take
	\begin{align*}
		\gamma:=\frac{(\alpha+\beta)q-s_pp^2-\sigma(2q-p)}{2-p}\in(0,\sigma p)
	\end{align*}
	so that we have
	\begin{align*}
		\tilde{\alpha}:=\frac{\sigma\left( q-\frac{p}{2}+\gamma\left( 1-\frac{p}{2}\right)+\frac{s_pp^2}{2}\right)}{q}>\alpha.
	\end{align*}
	Therefore, after applying of Lemma \ref{lem11}, we deduce
	\begin{align*}
		&\quad\int_{T_0+S}^{T_1}[u(\cdot,t)]^q_{W^{\alpha,q}\left( B_r(x_0)\right)}\mathrm{d}t\\
		&\le C\left[ \frac{(\mathcal{T}+1)^q+\int_{T_0}^{T_1}[u(\cdot,t)]^q_{W^{\sigma,q}(B_R(x_0))}\,\mathrm{d}t}{(\beta-\alpha)S(R-r)^{N+2q+1}}+\frac{\|u\|^q_{L^\infty\left(B_R(x_0)\times(T_0,T_1)\right)}}{\alpha(R-r)^{\alpha q}}\right]\\
		&\le\frac{C\left[\left(\mathcal{T}+1\right)^q+\int_{T_0}^{T_1}[u(\cdot,t)]^q_{W^{\sigma,q}\left( B_R(x_0)\right)}\mathrm{d}t\right]}{S(R-r)^{N+2q+1}}
	\end{align*}
	with the constant $C$ takes the form of
	\begin{align*}
		C=\frac{\widetilde{C}(N,p,q,s_1,s_p)}{(1-\alpha)^{q+1}\alpha(\beta-\alpha)}.
	\end{align*}
\end{proof}

\begin{lemma}
	\label{lem19}
	Let $p\in(1,2)$, $s_p\in\left( 0,\frac{p-1}{p}\right]$, $q\ge p$ and suppose $u\in L^q_{\mathrm{loc}}\bigl(I;W^{\sigma,q}_{\mathrm{loc}}(\Omega)\bigl)$ is a locally bounded weak solution to problem \eqref{1.1} in the sense of Definition \ref{def1}. Then the following conclusions hold:
	\begin{itemize}
		\item[(i)] If $\sigma\in \left(0,\frac{2q-s_pp^2}{2q-p^2+p}\right]$, we have
		\begin{align*}
			u\in L^q_{\mathrm{loc}}\bigl(I;W^{\alpha,q}_{\mathrm{loc}}(\Omega)\bigl)
		\end{align*}
		for any $\alpha\in(\sigma,\beta)$ with
		\begin{align*}
			\beta:=\left(1-\frac{p^2}{2q}+\frac{p}{2q}\right)\sigma+\frac{s_pp^2}{2q}. 
		\end{align*}
		Moreover, there exists a constant $C_1$ depending on $N,p,q,s_1,s_p,\alpha,\sigma$ such that for any ball $B_R(x_0)\subset\subset\Omega$ with $R\in(0,1)$, $r\in (0,R)$, $(T_0,T_1)\subset\subset(0,T)$ and $S\in \left( 0,\min\left\lbrace T_1-T_0,1\right\rbrace \right) $,
		\begin{align*}
			\int_{T_0+S}^{T_1}[u(\cdot,t)]^q_{W^{\alpha,q}\left( B_r(x_0)\right)}\mathrm{d}t\le\frac{C_1\left[\left(\mathcal{T}+1\right)^q+\int_{T_0}^{T_1}[u(\cdot,t)]^q_{W^{\sigma,q}\left( B_R(x_0)\right)}\mathrm{d}t\right]}{S(R-r)^{N+2q+1}};
		\end{align*}
		\item[(ii)] If $\sigma\in\left(\frac{2q-s_pp^2}{2q-p^2+p},1\right)$, we have
		\begin{align*}
			u\in L^q_{\mathrm{loc}}\bigl(I;W^{1,q}_{\mathrm{loc}}(\Omega)\bigl).
		\end{align*}
		Moreover, there exists a constant $C_2$ depending on $N,p,q,s_1,s_p,\sigma$ such that for any ball $B_R(x_0)\subset\subset\Omega$ with $R\in(0,1)$, $r\in (0,R)$, $(T_0,T_1)\subset\subset I$ and $S\in \left( 0,\min\left\lbrace T_1-T_0,1\right\rbrace \right) $,
		\begin{align*}
			\int_{T_0+S}^{T_1}\int_{B_r(x_0)}|\nabla u|^q\,\mathrm{d}x\mathrm{d}t\le\frac{C_2\left[\left(\mathcal{T}+1\right)^q+\int_{T_0}^{T_1}[u(\cdot,t)]^q_{W^{\sigma,q}\left( B_R(x_0)\right)}\mathrm{d}t\right]}{S(R-r)^{N+2q+1}}.
		\end{align*}
	\end{itemize}
\end{lemma}

\begin{proof}
	This proof is the version of the proof of Lemma \ref{lem12} in the sub-quadratic case. For $\sigma\in \left(0,\frac{2q-s_pp^2}{2q-p^2+p}\right]$, one can obtain the desired estimate by using the method in the proof of Lemma \ref{lem18} with the value of parameters $\gamma$ and $\tilde{\alpha}$ remain unchanged.
	
	For $\sigma\in\left(\frac{2q-s_pp^2}{2q-p^2+p},1\right)$, we handle the first term of \eqref{18.2} with Lemma \ref{lem7} to have, for any $\lambda\in B_d\backslash\left\lbrace0\right\rbrace$,
	\begin{align*}
		&\quad\int_{T_0+S}^{T_1}\int_{B_r(x_0)}\left|\tau_\lambda\left( J_{\frac{q}{p}+1}(\tau_hu)\right) \right|^p\,\mathrm{d}x\mathrm{d}t\\
		&\le C|\lambda|^{\frac{s_pp^2}{2}+\gamma\left( 1-\frac{p}{2}\right) } \int_{T_0}^{T_1}\int_{B_{r+d}(x_0)}\int_{B_{r+d}(x_0)}\frac{\left|J_{\frac{q}{p}+1}(\tau_hu)(x,t)-J_{\frac{q}{p}+1}(\tau_hu)(y,t)\right|^p}{|x-y|^{N+\frac{s_pp^2}{2}+\gamma\left(1-\frac{p}{2}\right)}}\,\mathrm{d}x\mathrm{d}y\mathrm{d}t\\
		&\quad+C\frac{|\lambda|^{\frac{s_pp^2}{2}+\gamma\left( 1-\frac{p}{2}\right) }}{\gamma(R-r)^{q}}\int_{T_0}^{T_1}\int_{B_{r+d}(x_0)}|\tau_hu|\,\mathrm{d}x\mathrm{d}t.
	\end{align*}
	Bringing the above inequality into \eqref{18.2} and letting $\lambda=h$, we deduce
	\begin{align}
		\label{19.1}
		&\quad\int_{T_0+S}^{T_1}\int_{B_r(x_0)}\left|\tau_h(\tau_hu)\right|^q\,\mathrm{d}x\mathrm{d}t\nonumber\\
		&\le\frac{C\left[\left(\mathcal{T}+1\right)^q+\int_{T_0}^{T_1}[u(\cdot,t)]^q_{W^{\sigma,q}\left( B_R(x_0)\right)}\mathrm{d}t\right]}{S(R-r)^{N+2q+1}}|h|^{\sigma\left(q-\frac{p}{2}\right)+\gamma\left(1-\frac{p}{2}\right)+\frac{s_pp^2}{2}}.
	\end{align}
	In order to ensure (ii) of Lemma \ref{lem10} can be applied, we set
	\begin{align}
		\label{19.2}
		\gamma:=\frac{\sigma\left(\frac{3}{2}p-\frac{1}{2}p^2-q\right)-\frac{s_pp^2}{2}+q}{2-p}
	\end{align}
	so that
	\begin{align*}
		\sigma\left(q-\frac{p}{2}\right)+\gamma\left(1-\frac{p}{2}\right)+\frac{s_pp^2}{2}>1.
	\end{align*}
	Thus, from (ii) of Lemma \ref{lem10} and  \eqref{19.1}, we obtain the following inequality:
	\begin{align*}
		\int_{T_0+S}^{T_1}\int_{B_r(x_0)}|\tau_hu|^q\,\mathrm{d}x\mathrm{d}t\le \frac{C\left[\left(\mathcal{T}+1\right)^q+\int_{T_0}^{T_1}[u(\cdot,t)]^q_{W^{\sigma,q}\left( B_R(x_0)\right)}\mathrm{d}t\right]}{S(R-r)^{N+2q+1}}|h|^q,
	\end{align*}
	where $C$ takes the form of
	\begin{align*}
		C=\frac{C(N,p,q,s_1,s_p)}{(\sigma p-\gamma)^\frac{q-p}{q}\left[\sigma\left(q+\frac{p}{2}-\frac{p^2}{2}\right)-\frac{s_pp^2}{2}\right]^q},
	\end{align*}
	and $\gamma$ is defined in \eqref{19.2}. By Lemma \ref{lem12}, we conclude the desired result.
\end{proof}

\begin{lemma}
	\label{lem20}
	Let $p\in(1,2)$, $q\ge p$ and suppose that $u\in L^q_{\mathrm{loc}}\bigl(I;W^{\sigma,q}_{\mathrm{loc}}(\Omega)\bigl)$ is a locally bounded weak solution to problem \eqref{1.1} in the sense of Definition \ref{def1}. Then the following conclusions hold:
	\begin{itemize}
		\item[(i)] If $s_p\in\left(0,\frac{p-1}{p}\right]$, we have
		\begin{align*}
			u\in L^q_{\mathrm{loc}}\bigl(I;W^{\alpha,q}_{\mathrm{loc}}(\Omega)\bigl)
		\end{align*}
		for any $\alpha\in\left( \sigma,\frac{s_pp}{p-1}\right) $. Moreover, there exist constants $C_1,\kappa_1,\kappa_2$ depending on $N,p,q,s_1,s_p,\alpha$ and $\sigma$ such that for any ball $B_R(x_0)\subset\subset\Omega$ with $R\in(0,1)$, $r\in (0,R)$, $(T_0,T_1)\subset\subset I$ and $S\in \left( 0,\min\left\lbrace T_1-T_0,1\right\rbrace \right) $,
		\begin{align*}
			\int_{T_0+S}^{T_1}[u(\cdot,t)]^q_{W^{\alpha,q}\left( B_r(x_0)\right) }\mathrm{d}t\le\frac{C_1\left[\left(\mathcal{T}+1\right)^q+\int_{T_0}^{T_1}[u(\cdot,t)]^q_{W^{\sigma,q}\left( B_R(x_0)\right)}\mathrm{d}t\right]}{S^{\kappa_2}(R-r)^{\kappa_1}};
		\end{align*}
		\item[(ii)] If $s_p\in \left(\frac{p-1}{p},1\right)$, we have
		\begin{align*}
			u\in L^q_{\mathrm{loc}}\bigl(I;W^{1,q}_{\mathrm{loc}}(\Omega)\bigl).
		\end{align*}
		Moreover, there exist constants $C_2,\kappa_3,\kappa_4$ depending on $N,p,q,s_1,s_p$ and $\sigma$ such that for any ball $B_R(x_0)\subset\subset\Omega$ with $R\in(0,1)$, $r\in (0,R)$, $(T_0,T_1)\subset\subset I$ and $S\in \left( 0,\min\left\lbrace T_1-T_0,1\right\rbrace \right) $,
		\begin{align*}
			\int_{T_0+S}^{T_1}\int_{B_r(x_0)}|\nabla u|^q\,\mathrm{d}x\mathrm{d}t\le\frac{C_2\left[\left(\mathcal{T}+1\right)^q+\int_{T_0}^{T_1}[u(\cdot,t)]^q_{W^{\sigma,q}\left( B_R(x_0)\right)}\mathrm{d}t\right]}{S^{\kappa_4}(R-r)^{\kappa_3}}.
		\end{align*}
	\end{itemize}
\end{lemma}

\begin{proof}
	We do not provide a detailed proof of this Lemma since it is almost exactly the same as the proof of Lemma \ref{lem13}. Indeed, an iteration is performed on Lemma \ref{lem18} and (i) of Lemma \ref{lem19} (when in the case of $s_p\in\left( 0,\frac{p-1}{p}\right)$, (ii) of Lemma \ref{lem19} also needs to be applied after the iteration to obtain gradient estimate of the weak solution). The difference lies in the definition of some parameters which we list below.
	
	To prove (i), for any $\alpha$, we set
	\begin{align*}
		\begin{cases}
			\tilde{\alpha}:=\frac{1}{2}\left(\alpha+\frac{s_pp}{p-1}\right),\\
			\alpha_0:=\sigma,\quad\alpha_{i+1}:=\left( 1-\frac{p^2}{2q}+\frac{p}{2q}\right)\alpha_i+\tilde{\alpha}\frac{p^2-p}{2q},\\
			\beta_i:=\left( 1-\frac{p^2}{2q}+\frac{p}{2q}\right)\alpha_i+\frac{s_pp^2}{2q}.
		\end{cases}
	\end{align*}
	After calculation, we request
	\begin{align*}
		i_0:=\left[\frac{\ln\frac{\tilde{\alpha}-\alpha}{\alpha-\sigma}}{\ln\left( 1-\frac{p^2}{2q}+\frac{p}{2q}\right) }\right]+2 
	\end{align*}
	such that $\alpha_{i_0}\ge\alpha$ and deduce the desired inequality with $C_1:=3^{i_0}2^{\left(N+2q+1+\frac{Nq}{p-1}\right)\frac{i_0(i_0+1)}{2} }C^{i_0}_*$, $\kappa_1:=(N+2q+1)i_0$ and $\kappa_2:=i_0$, here $C_*$ takes the form of
	\begin{align*}
		C_*:=\frac{C(N,p,q,s_1,s_p)}{\sigma(1-\alpha)^{q+1}\left[s_pp^2-(p-1)\tilde{\alpha}\right]^{\frac{q-p}{q}}}
	\end{align*}
	and it is independent of the index $i$.
	
	To prove (ii), we set
	\begin{align*}
		\begin{cases}
			\gamma:=\frac{4q-s_pp^2-p^2+p}{2q-p^2+p},\\
			\tilde{\gamma}:=\frac{1+\gamma}{2},\\
			\gamma_0:=\sigma,\quad \gamma_{i+1}:=\left( 1-\frac{p^2}{2q}+\frac{p}{2q}\right)\gamma_i+\tilde{\gamma}\frac{p^2-p}{2q}.
		\end{cases}
	\end{align*}
	After iterating (i) of Lemma \ref{lem19}, one can obtain
	\begin{align}
		\label{20.1}
		\int_{T_0+S_{i'_0}}^{T_1}[u(\cdot,t)]^q_{W^{\gamma_{i'_{0}},q}\left(B_{\rho_{i'_0}}(x_0)\right)}\,\mathrm{d}t\le \frac{\widetilde{C}_{i'_0}\left[\left(\mathcal{T}+1\right)^q+\int_{T_0}^{T_1}[u(\cdot,t)]^q_{W^{\sigma,q}\left( B_R(x_0)\right)}\mathrm{d}t\right]}{S^{i'_0}(R-r)^{(N+2q+1)i'_0}},
	\end{align}
	where $\widetilde{C}_{i'_0}:=3^{i'_0}2^{\left(N+2q+1+\frac{Nq}{p-1}\right)\frac{i'_0(i'_0+1)}{2} }\widetilde{C}^{i'_0}_*$, $\widetilde{C}_*$ takes the form of
	\begin{align*}
		\widetilde{C}_*:=\frac{C(N,p,q,s_1,s_p)}{\sigma(1-\alpha)^{q+2}}
	\end{align*}
	and $i'_0$ is given by
	\begin{align*}
		i'_0:=\left[\frac{\ln\frac{1-\alpha}{2(\alpha-\sigma)}}{\ln\left(1-\frac{p^2}{2q}+\frac{p}{2q}\right)}\right]+2. 
	\end{align*}
	
	Finally, we apply (ii) of Lemma \ref{lem19} providing
	\begin{align*}
		\int_{T_0+S_{i'_0}}^{T_1}[u(\cdot,t)]^q_{W^{\gamma_{i'_{0}},q}\left(B_{\rho_{i'_0}}(x_0)\right)}\,\mathrm{d}t<+\infty
	\end{align*}
	to conclude
	\begin{align*}
		\int_{T_0+S}^{T_1}\int_{B_r(x_0)}|\nabla u|^q\,\mathrm{d}x\mathrm{d}t&\le\frac{C\left[\left(\mathcal{T}+1\right)^q+\int_{T_0+S_{i'_0}}^{T_1}[u(\cdot,t)]^q_{W^{\gamma_{i'_{0}},q}\bigl(B_{\rho_{i'_0}}(x_0)\bigl)}\,\mathrm{d}t\right]}{(S-S_{i'_0})(\rho_{i'_0}-r)^{N+2q+1}}\\
		&\le \frac{4^{i'_0+1}C\widetilde{C}_{i'_0}\left[\left(\mathcal{T}+1\right)^q+\int_{T_0+S_{i'_0}}^{T_1}[u(\cdot,t)]^q_{W^{\gamma_{i'_{0}},q}\bigl(B_{\rho_{i'_0}}(x_0)\bigl)}\,\mathrm{d}t\right]}{S^{i'_0+1}(R-r)^{(N+2q+1)(i'_0+1)}}.
	\end{align*}
	We set $C_2:=4^{i'_0+1}C\widetilde{C}_{i'_0}$, $\kappa_3:=(N+2q+1)(i'_0+1)$, $\kappa_4:=i'_0+1$ and finish the proof.
\end{proof}

Although the constants in the sub-quadratic case involve additional parameter dependence compared with the super-quadratic case, the argument is simpler because no iteration is required to improve the integrability of weak solutions. The details can be found in the proof of Theorem \ref{th21}.

\begin{theorem}
	\label{th21}
	Let $p\in(1,2)$ and suppose that $u$ is a locally bounded weak solution to problem \eqref{1.1} in the sense of Definition \ref{def1}. Then, for any $q\ge p$, the following conclusions hold:
	\begin{itemize}
		\item[(i)] If $s_p\in \left( 0,\frac{p-1}{p}\right]$, we have
		\begin{align*}
			u\in L^q_{\mathrm{loc}}\bigl(I;W^{\alpha,q}_{\mathrm{loc}}(\Omega)\bigl)
		\end{align*}
		for any $\alpha\in\left(s_p,\frac{s_pp}{p-1}\right)$. Moreover, there exist constants $C_1,\lambda_1,\lambda_2$ depending on $N,p,q,s_1,s_p$ and $\alpha$ such that for any ball $B_R(x_0)\subset\subset\Omega$ with $R\in(0,1)$, $r\in (0,R)$, $(T_0,T_1)\subset\subset I$ and $S\in \left( 0,\min\left\lbrace T_1-T_0,1\right\rbrace \right) $,
		\begin{align*}
			\int_{T_0+S}^{T_1}[u(\cdot,t)]^q_{W^{\alpha,q}\left( B_r(x_0)\right) }\mathrm{d}t\le\frac{C_1\left[\left(\mathcal{T}+1\right)^q+\left( \int_{T_0}^{T_1}[u(\cdot,t)]^p_{W^{s_p,p}\left( B_R(x_0)\right)}\mathrm{d}t\right)^{\frac{q}{p}} \right]}{S^{\lambda_2}(R-r)^{\lambda_1}};
		\end{align*}
		\item[(ii)] If $s_p\in \left( \frac{p-1}{p},1\right)$, we have
		\begin{align*}
			u\in L^q_{\mathrm{loc}}\bigl(I;W^{1,q}_{\mathrm{loc}}(\Omega)\bigl).
		\end{align*}
		Moreover, there exist constants $C_2, \theta_1, \theta_2$ depending on $N, p, q, s_1$ and $s_p$ such that for any ball $B_R(x_0)\subset\subset\Omega$ with $R\in(0,1)$, $r\in (0,R)$, $(T_0,T_1)\subset\subset I$ and $S\in \left( 0,\min\left\lbrace T_1-T_0,1\right\rbrace \right) $,
		\begin{align*}
			\int_{T_0+S}^{T_1}\int_{B_r(x_0)}|\nabla u|^q\,\mathrm{d}x\mathrm{d}t\le\frac{C_2\left[\left(\mathcal{T}+1\right)^q+\left( \int_{T_0}^{T_1}[u(\cdot,t)]^p_{W^{s_p,p}\left( B_R(x_0)\right)}\mathrm{d}t\right)^{\frac{q}{p}} \right]}{S^{\theta_2}(R-r)^{\theta_1}}.
		\end{align*}
	\end{itemize}
\end{theorem}

\begin{proof}
	Unlike the super-quadratic case, Proposition \ref{pro15} holds without requiring a positive lower bound on the parameter $\sigma$, so we do not need to improve integrability through iteration.
	
	Note the fact that \eqref{14.1} also holds in the sub-quadratic case. We can choose $q_1=q$, $q_2=p$ and $\alpha=s_p$ to have
	\begin{align}
		\label{21.1}
		&\quad\int_{T_0}^{T_1}[u(\cdot,t)]^{q}_{W^{\frac{s_p p}{q},q}(B_R(x_0))}\,\mathrm{d}t\nonumber\\
		&\le C(p,q)\left[\|u\|^q_{L^\infty\left(B_R(x_0)\times(T_0,T_1)\right)}+\left( \int_{T_0}^{T_1}[u(\cdot,t)]^{p}_{W^{s_p,p}(B_R(x_0))}\,\mathrm{d}t\right)^{\frac{q}{p}}\right].
	\end{align}
	For the case $s_p\in\left( 0,\frac{p-1}{p}\right]$, by combining \eqref{21.1} and (i) of Lemma \ref{lem20}, we have the following estimate
	\begin{align*}
		\int_{T_0+S}^{T_1}[u(\cdot,t)]^q_{W^{\alpha,q}(B_r(x_0))}\,\mathrm{d}t&\le\frac{C_1\left[\left(\mathcal{T}+1\right)^q+\int_{T_0}^{T_1}[u(\cdot,t)]^q_{W^{\frac{s_p p}{q},q}\left( B_R(x_0)\right)}\mathrm{d}t \right]}{S^{\lambda_2}(R-r)^{\lambda_1}}\\
		&\le\frac{C_1\left[\left(\mathcal{T}+1\right)^q+\left( \int_{T_0}^{T_1}[u(\cdot,t)]^p_{W^{s_p,p}\left( B_R(x_0)\right)}\mathrm{d}t\right)^{\frac{q}{p}} \right]}{S^{\lambda_2}(R-r)^{\lambda_1}}.
	\end{align*}
	
	Similarly, when $s_p\in\left( \frac{p-1}{p},1\right)$, we combining \eqref{21.1} and (ii) of Lemma \ref{lem20} to obtain
	\begin{align*}
		\int_{T_0+S}^{T_1}\int_{B_r(x_0)}|\nabla u|^q\,\mathrm{d}x\mathrm{d}t\le\frac{C_2\left[\left(\mathcal{T}+1\right)^q+\left( \int_{T_0}^{T_1}[u(\cdot,t)]^p_{W^{s_p,p}\left( B_R(x_0)\right)}\mathrm{d}t\right)^{\frac{q}{p}} \right]}{S^{\theta_2}(R-r)^{\theta_1}}.
	\end{align*}
\end{proof}

\section{H\"{o}lder regularity of solutions}
\label{sec5}
In this section, we prove our main results. Unlike the elliptic case, the Sobolev regularity obtained does not immediately yield H\"{o}lder estimates of $u(\cdot,t)$ with respect to the spatial variable $x$ via embeddings. Consequently, we reapply Propositions \ref{pro4} and \ref{pro15} to extract pointwise information of the weak solution in the time variable \(t\). The following corollary makes this procedure explicit.
\begin{corollary}
	\label{cor22}
	Let $p>1$ and suppose $u$ is a locally bounded weak solution to problem \eqref{1.1} in the sense of Definition \ref{def1}. Then, for any $t\in I $, we have
	\begin{align*}
		u(\cdot,t)\in C^\alpha_\mathrm{loc}(\Omega)
	\end{align*}
	For any $\alpha\in\left( 0,\min\left\lbrace 1,\frac{s_pp}{p-1}\right\rbrace \right) $. Moreover, there exist constants $C_1,\mu_1,\mu_2$ depending on $N,p,q,s_1,s_p,\alpha$ such that for any ball $B_R(x_0)\subset\subset\Omega$ with $R\in(0,1)$, $r\in(0,R)$ and $S\in\left( 0,\min\left\lbrace t,1\right\rbrace \right) $,
	\begin{align*}
		[u(\cdot,t)]_{C^{0,\alpha}(B_r(x_0))}\le \frac{C\left[ \left( \mathcal{T}+1\right)^2+\left( \int_{t-S}^{t}[u(\cdot,t)]^p_{W^{s_p,p}\left( B_R(x_0)\right)}\mathrm{d}t\right)^{\frac{2}{p}}\right] }{S^{\mu_2}(R-r)^{\mu_1}}.
	\end{align*}
\end{corollary}

\begin{proof}
	Let us first consider the case $p\ge2$. Thanks to the proof of Proposition \ref{pro4}, we can estimate the third term on the left-hand side of \eqref{3.0} as follows
	\begin{align*}
		&\quad\int_{B_{r+d}(x_0)}|\tau_hu(x,t)|^q\,\mathrm{d}x\\
		&\le\frac{C(N,p,q,s_1,s_p)}{S(R-r)^{N+s_pp+1}}\left[\int_{t-\frac{S}{2}}^{t}\int_{B_{R-2d}(x_0)}|\tau_hu|^{q-1}\,\mathrm{d}x\mathrm{d}t\right.\\
		&\quad+\left( \int_{t-\frac{S}{2}}^{t}[u(\cdot,t)]^{q+p-2}_{W^{\sigma,q+p-2}(B_{R-d}(x_0))}\,\mathrm{d}t\right)^{\frac{p-2}{q+p-2}}\left( \int_{t-\frac{S}{2}}^{t}\int_{B_{R-2d}(x_0)}|\tau_hu|^{q+p-2}\,\mathrm{d}x\mathrm{d}t\right)^\frac{q}{q+p-2}\\
		&\quad+\mathcal{T}^{p-2}\int_{t-\frac{S}{2}}^{t}\int_{B_{R-2d}(x_0)}|\tau_hu|^q\,\mathrm{d}x\mathrm{d}t+\mathcal{T}^{p-1}|h|\int_{t-\frac{S}{2}}^{t}\int_{B_{R-2d}(x_0)}|\tau_hu|^{q-1}\,\mathrm{d}x\mathrm{d}t\\
		&\quad\left.+\int_{t-\frac{S}{2}}^{t}\int_{B_{R-2d}(x_0)}|\tau_hu|^q\,\mathrm{d}x\mathrm{d}t\right],
	\end{align*}
	where we replace the parameters $q, R, r$ by $q+p-2, R-2d, r+d$ respectively on the right-hand side of \eqref{4.0}. Combining the above estimate with \eqref{6.0}, we have
	\begin{align*}
		&\quad\int_{B_{r+d}(x_0)}|\tau_hu(x,t)|^q\,\mathrm{d}x\\
		&\le\frac{C(N,p,q,s_1,s_p)}{S(R-r)^{N+s_pp+1}}\left[ \left( \mathcal{T}+1\right)^{q+p-2}+\int_{t-\frac{S}{2}}^{t}[u(\cdot,t)]^{q+p-2}_{W^{\sigma,q+p-2}(B_{R-d}(x_0))}\,\mathrm{d}t\right]|h|^{\sigma(q-1)}.
	\end{align*}
	For any $\alpha\in\left(0,\min\left\lbrace 1,\frac{s_pp}{p-1}\right\rbrace \right)$, we set
	\begin{align*}
		\tilde{\alpha}:=
		\begin{cases}
			\frac{2}{3}\alpha+\frac{s_pp}{3(p-1)}\quad&\text{if }s_p\le\frac{p-1}{p},\\[2mm]
			\frac{2}{3}\alpha+\frac{1}{3}\quad&\text{if }s_p>\frac{p-1}{p}.
		\end{cases}
	\end{align*}
	Letting $\sigma:=2\tilde{\alpha}-\alpha$ and $q=\frac{N}{\tilde{\alpha}-\alpha}$, we obtain, by Lemma \ref{lem11},
	\begin{align*}
		[u(\cdot,t)]^q_{W^{\tilde{\alpha},q}(B_r(x_0))}\le\frac{C\left[ \left( \mathcal{T}+1\right)^{q+p-2}+\int_{t-\frac{S}{2}}^{t}[u(\cdot,t)]^{q+p-2}_{W^{\sigma,q+p-2}(B_{R-d}(x_0))}\,\mathrm{d}t\right]}{S(R-r)^{N+s_pp+1}}.
	\end{align*}
	
	For $s_p\in\left(0,\frac{p-1}{p}\right]$, we utilize (i) of Theorem \ref{th14} and embedding result Lemma\ref{lem23} to get
	\begin{align*}
		[u(\cdot,t)]_{C^{0,\alpha}(B_r(x_0))}&\le C[u(\cdot,t)]_{W^{\tilde{\alpha},q}(B_r(x_0))}\\
		&\le \frac{C\left[ \left( \mathcal{T}+1\right)^2+\left( \int_{t-S}^{t}[u(\cdot,t)]^p_{W^{s_p,p}\left( B_R(x_0)\right)}\mathrm{d}t\right)^{\frac{2}{p}}\right] }{S^{\mu_2}(R-r)^{\mu_1}}.
	\end{align*}
	
	For $s_p\in\left( \frac{p-1}{p},1\right)$, after application of (ii) of Theorem \ref{th14}, Lemma \ref{lem23} and Lemma \ref{lem14.1}, we have
	\begin{align*}
		[u(\cdot,t)]_{C^{0,\alpha}(B_r(x_0))}&\le C[u(\cdot,t)]_{W^{\tilde{\alpha},q}(B_r(x_0))}\\
		&\le\frac{C\left[ \left( \mathcal{T}+1\right)^{\frac{q+p-2}{q}}+\left( \int_{t-\frac{S}{2}}^{t}[u(\cdot,t)]^{q+p-2}_{W^{\sigma,q+p-2}(B_{R-d}(x_0))}\,\mathrm{d}t\right) ^\frac{1}{q}\right]}{S(R-r)^{N+s_pp+1}}\\
		&\le\frac{C\left[ \left( \mathcal{T}+1\right)^{\frac{q+p-2}{q}}+\left( \int_{t-\frac{S}{2}}^{t}\int_{B_{R-d}(x_0)}|\nabla u|^{q+p-2}\,\mathrm{d}x\mathrm{d}t\right)^\frac{1}{q}\right]}{S(R-r)^{N+s_pp+1}}\\
		&\le\frac{C\left[ \left( \mathcal{T}+1\right)^2+\left( \int_{t-S}^{t}[u(\cdot,t)]^p_{W^{s_p,p}\left( B_R(x_0)\right)}\mathrm{d}t\right)^{\frac{2}{p}}\right] }{S^{\mu_2}(R-r)^{\mu_1}}.
	\end{align*}
	
	Next, we deal with the case $p\in(1, 2)$. Combining \eqref{3.0}, \eqref{4.1}, \eqref{4.2}, \eqref{15.4} and \eqref{15.5}, and using the non-negativity of $P_3$ from the proof of Proposition \ref{pro15}, one can obtain the following estimate
	\begin{align}
		\label{22.1}
		&\quad\int_{B_{r+d}(x_0)}|\tau_hu(x,t)|^{q}\,\mathrm{d}x\nonumber\\
		&\le \frac{C(N,p,q,s_1,s_p)}{S(R-r)^{N+s_pp+1}}\left(\int_{t-\frac{S}{2}}^{t}\int_{B_{R-2d}(x_0)}|\tau_hu|^{q-1}\,\mathrm{d}x\mathrm{d}t+\int_{t-\frac{S}{2}}^{t}\int_{B_{R-2d}(x_0)}|\tau_hu|^{q+p-2}\,\mathrm{d}x\mathrm{d}t\right.\nonumber\\
		&\quad\left.+\mathcal{T}^{p-1}|h|\int_{t-\frac{S}{2}}^{t}\int_{B_{R-d}(x_0)}|\tau_hu|^{q-1}\,\mathrm{d}x\mathrm{d}t+\int_{t-\frac{S}{2}}^{t}\int_{B_{R-d}(x_0)}|\tau_hu|^{q}\,\mathrm{d}x\mathrm{d}t\right).
	\end{align}
	We apply the method utilized in dealing with the case $p\ge2$ to get the desired result. Precisely, we make use of \eqref{18.1} to deduce
	\begin{align*}
		[u(\cdot,t)]^q_{W^{\tilde{\alpha},q}(B_r(x_0))}\le\frac{C\left[ \left( \mathcal{T}+1\right)^{q+p-2}+\int_{t-\frac{S}{2}}^{t}[u(\cdot,t)]^{q+p-2}_{W^{\sigma,q+p-2}(B_{R-d}(x_0))}\,\mathrm{d}t\right]}{S(R-r)^{N+s_pp+1}}.
	\end{align*}
	Finally, we apply (i) of Theorem \ref{th21}, Lemma \ref{lem23} in the case $s_p\in\left(0,\frac{p-1}{p}\right]$ and apply (ii) of Theorem \ref{th21}, Lemma \ref{lem23}, Lemma \ref{lem14.1} in the case $s_p\in\left( \frac{p-1}{p},1\right)$ to conclude
	\begin{align*}
		[u(\cdot,t)]_{C^{0,\alpha}(B_r(x_0))}\le\frac{C\left[ \left( \mathcal{T}+1\right)^2+\left( \int_{t-S}^{t}[u(\cdot,t)]^p_{W^{s_p,p}\left( B_R(x_0)\right)}\mathrm{d}t\right)^{\frac{2}{p}}\right] }{S^{\mu_2}(R-r)^{\mu_1}}.
	\end{align*}
\end{proof}

We are now in a position to study the regularity of weak solutions with respect to the time variable $t$. The following lemma will be instrumental in extracting the regularity information in Proposition \ref{pro26}. Its proof clearly illustrates how regularity information is balanced by the structure of the operators $(-\Delta_1)^{s_1}$ and $(-\Delta_p)^{s_p}$, and follows the approach developed in \cite[Proposition 6.2]{BLS21}.

Let $\bar{\eta}_{r}(x_0)$ and $\bar{\eta}_{r,\theta}(x_0,t_0)$ denote the averages of $\eta$ over the ball $B_r(x_0)$ and the cylinder $B_r(x_0)\times(t_0-\theta,t_0)$, respectively. When no confusion arises, we simply write $\bar{\eta}_r$ and $\bar{\eta}_{r,\theta}$.
\begin{lemma}
	\label{lem25}
	Let $p\in[1,+\infty)$ and suppose $u\in W^{s,p}(B_r)$. Then, for any non-negative function $\eta\in C^\infty_0(B_r)$ such that $\bar{\eta}_r=1$, there holds
	\begin{align*}
		\int_{B_r}\left| u-\overline{(u\eta)}_r\right|^p\,\mathrm{d}x\le\left(\frac{2^{N+sp}}{\omega_N}\|\eta\|^p_{L^\infty(B_r)}\right)r^{sp}[u]^p_{W^{s,p}(B_r)}. 
	\end{align*}
\end{lemma}

\begin{proposition}
	\label{pro26}
	Let $p>1$, $s_1, s_p\in(0,1)$ and suppose that $u$ is a weak solution to problem \eqref{1.1} in the sense of Definition \ref{def1}, which satisfies
	\begin{align}
		\label{26.0}
		\sup\limits_{t\in(t_1,t_2)}[u(\cdot,t)]_{C^\delta(B_r(x_0))}+\operatorname*{ess\,sup}_{t\in(t_1,t_2)}\mathrm{Tail}\left(u;x_0,\frac{r}{2},t\right)+\|u\|_{L^\infty\left(B_{\frac{r}{2}}(x_0)\times(t_1,t_2)\right) }\le M_{\delta}
	\end{align}
	for any $\delta\in\left(s_p,\min\left\lbrace1,\frac{s_pp}{p-1}\right\rbrace \right)$ with $(t_1,t_2)\subset\subset(0,T)$ and some $M_\delta>1$. Then there exists a constant $C$ depending on $N, p, s_1, s_p$ and $\delta$ such that
	\begin{align*}
		|u(x,\tau_1)-u(x,\tau_2)|\le CM_\delta^p|\tau_1-\tau_2|^\gamma,
	\end{align*}
	where $\gamma:=\frac{1}{\frac{\max\left\lbrace s_1,s_pp\right\rbrace }{\delta}+1}$.
\end{proposition}

\begin{proof}
	We set the parabolic cylinder
	\begin{align}
		\mathcal{Q}_{r,\theta}(x_0,t_0)=B_r(x_0)\times(t_0-\theta,t_0],
	\end{align}
	where $\theta<\min\left\lbrace t_0,\frac{1}{4}\right\rbrace $ and $r<\frac{1}{4}$. Let $\eta\in C^{\infty}_0\left( B_{\frac{r}{2}}(x_0)\right) $ be a non-negative cut-off function satisfying $\eta\equiv\|\eta\|_{L^\infty\left( B_{\frac{r}{2}}(x_0)\right)}$ on $B_{\frac{r}{4}}(x_0)$, $\overline{\eta}_r=1$ and $|\nabla\eta|\le\frac{C}{r}$. Thus we have
	\begin{align*}
		\|\eta\|_{L^\infty\left( B_{\frac{r}{2}}(x_0)\right)}=\frac{1}{\left|B_{\frac{r}{4}}(x_0)\right|}\int_{B_{\frac{r}{4}(x_0)}}\eta\,\mathrm{d}x\le\frac{\left|B_{r}(x_0)\right|}{\left| B_{\frac{r}{4}}(x_0)\right|}\overline{\eta}_r=4^N.
	\end{align*}
	
	Now, we estimate $\mint_{\mathcal{Q}_{r,\theta}(x_0,t_0)}\left|u(x,t)-\overline{u}_{r,\theta}\right|\,\mathrm{d}x\mathrm{d}t$ as
	\begin{align}
		\label{26.1}
		&\quad\mint_{\mathcal{Q}_{r,\theta}(x_0,t_0)}\left|u(x,t)-\overline{u}_{r,\theta}\right|\,\mathrm{d}x\mathrm{d}t\nonumber\\
		&\le\mint_{\mathcal{Q}_{r,\theta}(x_0,t_0)}\left|u(x,t)-\overline{(u\eta)}_{r}(t)\right|\,\mathrm{d}x\mathrm{d}t\nonumber\\
		&\quad+\mint_{\mathcal{Q}_{r,\theta}(x_0,t_0)}\left|\overline{u}_{r,\theta}-\overline{(u\eta)}_{r,\theta}\right|\,\mathrm{d}x\mathrm{d}t\nonumber\\
		&\quad+\mint_{\mathcal{Q}_{r,\theta}(x_0,t_0)}\left|\overline{(u\eta)}_{r,\theta}-\overline{(u\eta)}_{r}(t)\right|\,\mathrm{d}x\mathrm{d}t\nonumber\\
		&=:E_1+E_2+E_3\nonumber\\
		&\le 2E_1+2E_3.
	\end{align}
	In the last line we utilize the fact that
	\begin{align*}
		E_2&=\left|\overline{u}_{r,\theta}-\overline{(u\eta)}_{r,\theta}\right|\\
		&=\mint_{\mathcal{Q}_{r,\theta}(x_0,t_0)}\left|u(x,t)-\overline{(u\eta)}_{r,\theta}\right|\,\mathrm{d}x\mathrm{d}t\\
		&\le \mint_{\mathcal{Q}_{r,\theta}(x_0,t_0)}\left( \left| u(x,t)-\overline{(u\eta)}_{r}(t)\right|+\left|\overline{(u\eta)}_{r,\theta}-\overline{(u\eta)}_{r}(t)\right| \right)\\
		&=E_1+E_3.
	\end{align*}
	Hence it is sufficient for us to estimate $E_1$ and $E_3$. Note that $u\in L^p_{\mathrm{loc}}\bigl(0,T;W^{s_p,p}_{\mathrm{loc}}(\Omega)\bigl)$, by applying Lemma \ref{lem25}, we have
	\begin{align}
		\label{26.2}
		E_1&\le C(N,s_p,p)\left( \frac{r^{s_pp}}{\mathcal{Q}_{r,\theta}(x_0,t_0)}\int_{t_0-\theta}^{t_0}[u(\cdot,t)]^p_{W^{s_p,p}(B_r(x_0))}\,\mathrm{d}t\right)^\frac{1}{p}\nonumber\\
		&\le C(N,s_p,p)M_\delta r^\delta.
	\end{align} 
	Next, we will estimate $E_3$. Observe that
	\begin{align}
		\label{26.3}
		&\quad\mint_{\mathcal{Q}_{r,\theta}(x_0,t_0)}\left|\overline{(u\eta)}_{r,\theta}-\overline{(u\eta)}_{r}(t)\right|\,\mathrm{d}x\mathrm{d}t\nonumber\\
		&=\int_{t_0-\theta}^{t_0}\left|\overline{(u\eta)}_{r,\theta}-\overline{(u\eta)}_{r}(t)\right|\,\mathrm{d}t\nonumber\\
		&=\int_{t_0-\theta}^{t_0}\left|\mint_{t_0-\theta}^{t_0}\left[ \overline{(u\eta)}_r(\upsilon)-\overline{(u\eta)}_r(t)\right]\,\mathrm{d}\upsilon\right|\,\mathrm{d}t\nonumber\\
		&\le\int_{t_0-\theta}^{t_0}\mint_{t_0-\theta}^{t_0}\left|\overline{(u\eta)}_r(\upsilon)-\overline{(u\eta)}_r(t)\right|\,\mathrm{d}\upsilon\mathrm{d}t\nonumber\\
		&\le\sup\limits_{T_0,T_1\in\left(t_0-\theta,t_0\right]}\left| \overline{(u\eta)}_r(T_0)-\overline{(u\eta)}_r(T_1)\right|.
	\end{align}
	We can use the weak formulation \eqref{1.2} with test function $\varphi(x,t)=\eta(x)$ to get
	\begin{align*}
		&\quad\left| B_r(x_0)\right| \left| \overline{(u\eta)}_r(T_0)-\overline{(u\eta)}_r(T_1)\right|\\
		&\le \int_{T_0}^{T_1}\int_{\mathbb{R}^N}\int_{\mathbb{R}^N}\frac{|\eta(x)-\eta(y)|}{|x-y|^{N+s_1}}\,\mathrm{d}x\mathrm{d}y\mathrm{d}t\\
		&\quad+\int_{T_0}^{T_1}\int_{\mathbb{R}^N}\int_{\mathbb{R}^N}\frac{J_p\left( u(x,t)-u(y,t)\right)\left( \eta(x)-\eta(y)\right)}{|x-y|^{N+s_pp}}\,\mathrm{d}x\mathrm{d}y\mathrm{d}t\\
		&\le\int_{T_0}^{T_1}[\eta]_{W^{s_1,1}(B_r(x_0))}\,\mathrm{d}t+2\int_{T_0}^{T_1}\int_{\mathbb{R}^N\backslash B_r(x_0)}\int_{B_{\frac{r}{2}}(x_0)}\frac{\eta(x)}{|x-y|^{N+s_1}}\,\mathrm{d}x\mathrm{d}y\mathrm{d}t\\
		&\quad+\int_{T_0}^{T_1}\int_{B_r(x_0)}\int_{B_r(x_0)}\frac{J_p\left( u(x,t)-u(y,t)\right)\left( \eta(x)-\eta(y)\right)}{|x-y|^{N+s_pp}}\,\mathrm{d}x\mathrm{d}y\mathrm{d}t\\
		&\quad+2\int_{T_0}^{T_1}\int_{\mathbb{R}^N\backslash B_r(x_0)}\int_{B_{\frac{r}{2}}(x_0)}\frac{J_p\left( u(x,t)-u(y,t)\right)\eta(x)}{|x-y|^{N+s_pp}}\,\mathrm{d}x\mathrm{d}y\mathrm{d}t\\
		&=:A_1+2A_2+A_3+2A_4.
	\end{align*}
	For the local term of $p$-growth $A_3$, by H\"{o}lder's inequality, we have for $\delta>s_p$,
	\begin{align*}
		A_3&\le[\eta]_{W^{s_p,p}(B_r(x_0))} \int_{T_0}^{T_1}\left(\int_{B_r(x_0)}\int_{B_r(x_0)}\frac{|u(x,t)-u(y,t)|^p}{|x-y|^{N+s_pp}}\,\mathrm{d}x\mathrm{d}y\right)^\frac{p-1}{p}\mathrm{d}t\\
		&\le CM_\delta^{p-1}r^{\frac{N}{p}-s_p}\int_{T_0}^{T_1}\left(\int_{B_r(x_0)}\int_{B_r(x_0)}|x-y|^{\delta p-N-s_pp}\,\mathrm{d}x\mathrm{d}y\right)^\frac{p-1}{p} \,\mathrm{d}t\\
		&\le CM_\delta^{p-1}|T_1-T_0|r^{N-s_pp+\delta(p-1)}\\
		&\le CM_\delta^{p-1}\theta r^{N-s_pp+\delta(p-1)}.
	\end{align*}
	For nonlocal part $A_4$, by Definition \ref{def1} and the assumption \eqref{26.0}, we obtain
	\begin{align*}
		A_4&\le C\|\eta\|_{L^\infty\bigl(B_{\frac{r}{2}}(x_0)\bigl)}\int_{T_0}^{T_1}\int_{\mathbb{R}^N\backslash B_r(x_0)}\int_{B_{\frac{r}{2}}(x_0)}\frac{\left| u(x,t)-u(y,t)\right|^{p-1}}{|x-y|^{N+s_pp}}\,\mathrm{d}x\mathrm{d}y\mathrm{d}t\\
		&\le C(N)\int_{T_0}^{T_1}\int_{\mathbb{R}^N\backslash B_r(x_0)}\int_{B_{\frac{r}{2}}(x_0)}\frac{|u(x,t)|^{p-1}}{|x-y|^{N+s_pp}}\,\mathrm{d}x\mathrm{d}y\mathrm{d}t\\
		&\quad+C(N)\int_{T_0}^{T_1}\int_{\mathbb{R}^N\backslash B_r(x_0)}\int_{B_{\frac{r}{2}}(x_0)}\frac{|u(y,t)|^{p-1}}{|x-y|^{N+s_pp}}\,\mathrm{d}x\mathrm{d}y\mathrm{d}t\\
		&\le C(N)\operatorname*{ess\,sup}_{t\in(t_1,t_2)}\mathrm{Tail}^{p-1}\left(u;x_0,\frac{r}{2},t\right)\theta r^{N-s_pp}+C(N)\|u\|^{p-1}_{L^\infty\left(B_{\frac{r}{2}}(x_0)\times(t_1,t_2)\right) }\theta r^{N-s_pp}\\
		&\le C(N)M^{p-1}_{\delta}\theta r^{N-s_pp}.
	\end{align*}
	
	To estimate $A_1$ and $A_2$, note that $|\nabla\eta|\le \frac{C}{r}$, we get
	\begin{align*}
		A_1\le \frac{C}{r}\int_{T_0}^{T_1}\int_{B_r(x_0)}\int_{0}^{2r}\frac{\mathrm{d}\rho\mathrm{d}x\mathrm{d}t}{\rho^{s_1}}\le C(N,s_1)\theta r^{N-s_1}
	\end{align*}
	and
	\begin{align*}
		A_2\le C(N)\int_{T_0}^{T_1}\int_{\mathbb{R}^N\backslash B_r(x_0)}\int_{B_{\frac{r}{2}}(x_0)}\frac{\mathrm{d}x\mathrm{d}y\mathrm{d}t}{|x-y|^{N+s_1}}\le C(N)\theta r^{N-s_1}.
	\end{align*}
	Combining the above estimates, we conclude
	\begin{align}
		\label{26.4}
		E_3\le CM_\delta^{p-1}\theta \left( r^{-s_1}+r^{-s_pp}\right).
	\end{align}
	Bringing \eqref{26.2} and \eqref{26.4} into \eqref{26.1}, one has 
	\begin{align}
		\label{26.5}
		\mint_{\mathcal{Q}_{r,\theta}(x_0, t_0)}\left|u(x,t)-\overline{u}_{r,\theta}\right|\,\mathrm{d}x\mathrm{d}t\le CM_{\delta}r^\delta+CM_\delta^{p-1}\theta \left( r^{-s_1}+r^{-s_pp}\right).
	\end{align}
	
	We now distinguish several cases as follows.
	\newline
	For the case 
	$\begin{cases}
		s_1\ge s_pp,\\
		\delta+s_1>1,
	\end{cases}$ we choose $\theta=r^{\delta+s_1}$ and obtain from \eqref{26.5} that
	\begin{align*}
		\mint_{\mathcal{Q}_{r,\theta}(x_0, t_0)}\left|u(x,t)-\overline{u}_{r, \theta}\right|\,\mathrm{d}x\mathrm{d}t\le CM_\delta^pr^\delta.
	\end{align*}
	By an application of \cite[Teorema 3.I]{D65} (see also in \cite[Theorem 3.2]{G09}), we know that $u$ is locally $\delta$-H\"{o}lder continuous in $B_\frac{r}{2}(x_0)\times(t_1, t_2)$ regarding the metric
	\begin{align*}
		\tilde{d_1}((x_1, \tau_1),(x_2, \tau_2))=|x_1-x_2|+|\tau_1-\tau_2|^\frac{1}{\delta+s_1}.
	\end{align*}
	Thus, we have the estimate
	\begin{align}
		\label{26.6}
		\sup_{x\in B_{\frac{r}{2}}(x_0),\,\tau_1,\tau_2\in (t_1,t_2)}|u(x,\upsilon_1)-u(x,\upsilon_2)|\le CM^p_\delta|\tau_1-\tau_2|^\gamma
	\end{align}
	with $\gamma=\frac{1}{\frac{s_1}{\delta}+1}$.
	
	For the case $\begin{cases}
		s_1\ge s_pp,\\
		\delta+s_1\le 1,
	\end{cases}$we set $r=\theta^{\frac{1}{\delta+s_1}}$ to deduce
	\begin{align*}
		\mint_{\mathcal{Q}_{r,\theta}(x_0,t_0)}\left|u(x,t)-\overline{u}_{r,\theta}\right|\,\mathrm{d}x\mathrm{d}t\le CM^p_\delta\theta^\frac{\delta}{\delta+s_1},
	\end{align*}
	which implies $u$ is locally $\frac{\delta}{\delta+s_1}$-H\"{o}lder continuous in $B_\frac{r}{2}(x_0)\times(t_1,t_2)$ regarding the metric
	\begin{align*}
		\tilde{d_2}((x_1,\tau_1),(x_2,\tau_2))=|x_1-x_2|^{\delta+s_1}+|\tau_1-\tau_2|
	\end{align*}
	and the estimate \eqref{26.6}.
	
	Similarly, for the remaining two cases, i.e.,
	\begin{align*}
		\begin{cases}
			s_1<s_pp,\\
			\delta+s_pp>1,
		\end{cases}\,\text{and}\quad
		\begin{cases}
			s_1<s_pp,\\
			\delta+s_pp\le1,
		\end{cases}
	\end{align*}
	we can get that $u$ is locally $\delta$-H\"{o}lder continuous in $B_\frac{r}{2}(x_0)\times(t_1,t_2)$ regarding the metric
	\begin{align*}
		\tilde{d_3}((x_1,\tau_1),(x_2,\tau_2))=|x_1-x_2|+|\tau_1-\tau_2|^\frac{1}{\delta+s_pp}
	\end{align*}
	and $u$ is locally $\frac{\delta}{\delta+s_pp}$-H\"{o}lder continuous in $B_\frac{r}{2}(x_0)\times(t_1,t_2)$ regarding the metric
	\begin{align*}
		\tilde{d_4}((x_1,\tau_1),(x_2,\tau_2))=|x_1-x_2|^{\delta+s_pp}+|\tau_1-\tau_2|.
	\end{align*}
	Hence,  in these two cases, we obtain the following estimate
	\begin{align*}
		\sup_{x\in B_{\frac{r}{2}}(x_0),\,\tau_1,\tau_2\in (t_1,t_2)}|u(x,\tau_1)-u(x,\tau_2)|\le CM^p_\delta|\tau_1-\tau_2|^\gamma
	\end{align*}
	with $\gamma=\frac{1}{\frac{s_pp}{\delta}+1}$. This completes this proof.
\end{proof}

By combining the Corollary \ref{cor22} with the Proposition \ref{pro26} and carefully computing the exponents involved, we are able to complete the proof of Theorem \ref{th27}.

\begin{proof}[Proof of Theorem \ref{th27}.]
	By Corollary \ref{cor22}, we can estimate that for any $\tilde{\alpha}\in \left( 0,\min\left\lbrace 1,\frac{s_pp}{p-1}\right\rbrace \right) $,
	\begin{align}
		\label{27.1}
		[u(\cdot, t)]_{C^{0,\tilde{\alpha}}(B_r(x_0))}\le \frac{C\left[ \left( \mathcal{T}+1\right)^2+\left( \int_{t-S}^{t}[u(\cdot,t)]^p_{W^{s_p,p}\left( B_R(x_0)\right)}\mathrm{d}t\right)^{\frac{2}{p}}\right] }{S^{\kappa'}(R-r)^{\kappa}}
	\end{align}
	for any $t\in(t_0-S,t_0]$. This inequality enables us to apply Proposition \ref{pro26} to obtain
	\begin{align}
		\label{27.2}
	&\quad	|u(x,\tau_1)-u(x,\tau_2)| \nonumber\\
		& \le \frac{C\left[ \left( \mathcal{T}+1\right)^{2p}+\left( \int_{t_0-2S}^{t_0}[u(\cdot,t)]^p_{W^{s_p, p}(B_R(x_0))}\,\mathrm{d}t\right) ^2\right] }{S^{\kappa'p}(R-r)^{\kappa p}}|\tau_1-\tau_2|^{\frac{\tilde{\alpha}}{\max\left\lbrace s_1, s_pp\right\rbrace+\tilde{\alpha}}}.
	\end{align}
	Therefore, for any $\gamma\in \left( 0,\Gamma\right)$, we fix $\tilde{\alpha}:=\frac{\gamma\max\left\lbrace s_1,s_pp\right\rbrace}{1-\gamma}$ and combine \eqref{27.1} and \eqref{27.2} to get the desired estimate.
\end{proof}

\subsection*{Acknowledgments}
This work was supported by the National Natural Science Foundation of China (No. 12471128) and China Scholarship Council (No. 202506120120). 


\subsection*{Conflict of interest} The authors declare that there is no conflict of interest. We further declare that this manuscript has no associated data.
	
\subsection*{Data availability}
No data was used for the research described in the article.

\end{document}